\documentclass[10pt,oneside]{amsart}
\usepackage[T1]{fontenc}
\usepackage{amsthm,amsmath}
\usepackage{amsbsy}
\usepackage{amstext}
\usepackage{amssymb}
\usepackage{esint}
\usepackage{mathtools}

\makeatletter
\numberwithin{equation}{section}
\numberwithin{figure}{section}

\newtheorem {theo} {Theorem} [section]
\newtheorem {prop} [theo] {Proposition}
\newtheorem {cory} [theo] {Corollary}
\newtheorem {lem} [theo] {Lemma}
\newtheorem {defi} [theo] {Definition}

\newtheorem*{thmA}{Theorem~A}
\newtheorem*{thmB}{Theorem~B}

\theoremstyle{definition}
\newtheorem{example}{Example}
\newtheorem {obs} [theo] {Remark}

\def\sideremark#1{\ifvmode\leavevmode\fi\vadjust{\vbox to0pt{\vss 
    \hbox to 0pt{\hskip\hsize\hskip1em           
 \vbox{\hsize2cm\tiny\raggedright\pretolerance10000
 \noindent #1\hfill}\hss}\vbox to8pt{\vfil}\vss}}}%


\makeatother

\begin{document}

\title{Analytic moduli for parabolic Dulac germs}
\author{P. Marde\v si\'c$^{1}$, M. Resman$^{2}$}

\begin{abstract} In this paper we give moduli of analytic classification for parabolic \emph{Dulac} i.e. \emph{almost regular germs}. Dulac germs appear as first return maps of hyperbolic polycycles. Their moduli are given by a sequence of \emph{Ecalle-Voronin-like} germs of analytic diffeomorphisms. We state the result in a bigger class of \emph{parabolic generalized Dulac germs} having power-logarithmic asymptotic expansions.
\end{abstract}

\maketitle

\noindent \emph{Acknowledgement}. This research was supported by the Croatian Unity Through Knowledge Fund (UKF) \emph{My first collaboration grant} project no. 16, by the Croatian Science Foundation (HRZZ) grants UIP-2017-05-1020 and PZS-2019-02-3055 from \emph{Research Cooperability} funded by the European Social Fund, and by the EIPHI Graduate School (contract ANR-17-EURE-0002). The UKF grant fully supported the $6$-month stay of $^2$ at University of Burgundy in 2018.

\section{Introduction}\label{sec:zero} First return maps of planar polynomial vector fields are classically used for studying limit cycles, i.e. isolated periodic orbits. They are considered when dealing with the possibility of accumulation of limit cycles (the Dulac problem) or of their creation by deformation (the cyclicity problem related to the $16^{th}$ Hilbert problem). Most interesting is the study of first return maps in a neighborhood of polycycles which can be hyperbolic or non-hyperbolic. The non-accumulation problem posed by
Dulac \cite{Dulac} was solved independently by Ecalle \cite{ecalle} and Ilyashenko \cite{ilyalim}.

We call \emph{Dulac germs} the type of germs that appear as first return maps of hyperbolic polycycles.  In this case, Ilyashenko's approach \cite{ilhyp} consists in extending a Dulac germ (he calls them \emph{almost regular} germs) from the positive real line to a sufficiently big domain in the complex plane, the so-called standard quadratic domain, which is a biholomorphic image of $\mathbb C_+$. This extension ensured the injectivity of the mapping that assigns to a Dulac germ its power-logarithmic asymptotic expansion, called the Dulac expansion. 

However, for non-hyperbolic polycycles, the proof is much more involved \cite{ecalle} and \cite{ilyalim}. The first return maps then expand asymptotically as transseries that also include iterated logarithms and exponentials. Moreover, the domain of their complex extension is not as straightforward.

Here, we consider \emph{parabolic} (i.e. \emph{tangent to identity}) Dulac germs and give  complete invariants of their analytic  classification. 

\noindent Our construction resembles the construction of Voronin \cite{voronin} for analytic invariants of
parabolic germs analytic at 0, giving the so-called \emph{Ecalle-Voronin moduli}. Only here we get a doubly infinite sequence of pairs of germs giving the moduli. In our forthcoming
paper \cite{drugi}, we study the inverse realization problem for moduli. The problem is solved in a bigger class of \emph{generalized Dulac germs}. This is the reason that here we formulate our results on moduli of analytic classification in this bigger class.

The solutions  of Dulac's problem by Ilyashenko and Ecalle are still not very well understood. They show the necessity of a very precise knowledge of germs of first retun maps and, in particular, of the relationship between their assymptotic transserial expansion and the germ itself. 
In this work we develop techniques such as \emph{integral summability}, its iterations and \emph{$\log$-Gevrey asymptotic expansions}.  They are related to 
difference equations or differential equations that the objects in question satisfy. We hope that these techniques will be of use also in more general problems of associating germs to given transserial expansions, in particular in studying the first return map of non-hyperbolic polycycles.

On the other hand, note that, among Dulac germs, the parabolic germs 
studied in this paper are the most interesting from the point of view of cyclicity (Hilbert's $16^{th}$ problem). Even in the hyperbolic case, except for the result of Mourtada \cite{mou} in the generic (non-parabolic case), and some results for polycyles with a few vertices, very little is known about the cyclicity. In the generic case the study of the cyclicity of a hyperbolic polycyle relies on Mourtada's normal form \cite{mounor} which is a composition of an alternating sequence of powers and translations. The argument breaks down in the parabolic case and a more precise knowledge of the Dulac map is required. We hope that some combinatorial argument based on our moduli (and their unfolding) can give this extra information.

\medskip

\section{Main definitions}\label{sec:introduction}








Let $\mathcal R_C$ be a subset of the Riemann surface of the logarithm, in the logarithmic chart $\zeta=-\log z$ given by: 
$$
\varphi\big(\mathbb C^+\setminus \overline {K(0,R)}\big),\ \varphi(\zeta)=\zeta+C(\zeta+1)^{\frac{1}{2}},\ C>0,\, R>0.
$$
Here, $\mathbb C^+=\{\zeta\in\mathbb C:\ \text{Re}(\zeta)>0\}$ and $\overline{K(0,R)}=\{\xi\in\mathbb C:|\xi|\leq R\}$. Such $\mathcal R_C$ (or its image in the logarithmic chart $\zeta=-\log z$) is called a \emph{standard quadratic domain} \cite{ilyalim}:

\begin{figure}[h!]
\includegraphics[scale=0.4]{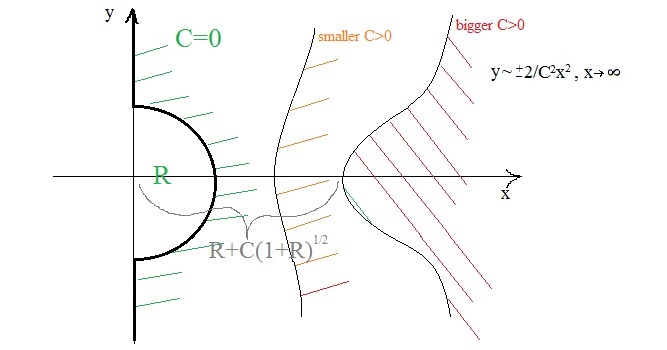}
\caption{The standard quadratic domain $\mathcal R_C$, $C>0$, in the logarithmic chart.}
\end{figure}

Note that, for $C=0$, we get the Riemann surface of the logarithm with bounded radii, while for $C>0$ we get its subdomain such that the radii become smaller on each level. Moreover, they tend to zero at an exponential rate. Indeed, if by $\varphi\in[(k-1)\pi,(k+1)\pi)$ we denote the $k$-th level of the surface $\mathcal R_C$, $k\in\mathbb Z,$ and by $\varphi_k=k\pi$, $k\in\mathbb Z,$ then the radii $r(\varphi_k)$ by levels $k\in\mathbb Z$ decrease at most at the rate:
\begin{equation}\label{eq:sie}
Ce^{-D\sqrt{\frac{|k|\pi}{2}}},\ |k|\to\infty,\ \text{ for some }C>0,\,D>0.
\end{equation}

By a \emph{germ on a standard quadratic domain} \cite{ilyalim}, we mean an equivalence class of functions that coincide on some standard quadratic domain (for arbitrarily big $R>0$ and $C>0$).

\begin{defi}[adapted from \cite{ilyalim}, \cite{roussarie}] We say that a germ $f$ is a \emph{Dulac germ} if it is
\begin{enumerate}

\item \emph{holomorphic} and bounded on some standard quadratic domain $\mathcal R_C$ and real on $\{z\in\mathcal R_C:\ \mathrm{Arg} (z)=0\}$,

\item admits in $\mathcal R_C$ a \emph{Dulac asymptotic expansion}\footnote{\emph{uniformly} on $\mathcal R_C$, see \cite[Section 24E]{ilya}: for every $\lambda>0$ there exists $n\in\mathbb N$ such that
$$
\big|f(z)-\sum_{i=1}^{n}z^{\lambda_i}P_i(-\log z)\big|=o(z^{\lambda}),
$$
uniformly on $\mathcal R_C$ as $|z|\to 0$. }
\begin{equation}\label{eq:dulac}
\widehat f(z)=\sum_{i=1}^{\infty}z^{\lambda_i} P_i(-\log z),\ c>0,\ z\to 0,
\end{equation}
where $\lambda_i\geq 1$, $i\in\mathbb N$, are strictly positive, finitely generated\footnote{\emph{Finitely generated} in the sense that there exists $n\in\mathbb N$ and $\alpha_1>0,\ldots,\alpha_n>0,$ such that each $\lambda_i$, $i\in\mathbb N$, is a finite linear combination of $\alpha_j$, $j=1,\ldots,n$, with coefficients from $\mathbb Z_{\geq 0}.$ For Dulac maps that are the first return maps of saddle polycycles, $\alpha_j$, $j=1,\ldots,n$, are related to the ratios of hyperbolicity of the saddles.} and strictly increasing to $+\infty$ and $P_i$ is a sequence of polynomials with real coefficients, and $P_1\equiv A$, $A>0$.

\end{enumerate}
A Dulac germ \emph{tangent to identity}, i.e. with $\lambda_1=1$ and $P_1\equiv 1$ in \eqref{eq:dulac}, is called a \emph{parabolic Dulac germ.}
\end{defi}
A series of the form \eqref{eq:dulac} is called a formal \emph{Dulac series}.  Note that all coefficients in the expansion are \emph{real}.

The choice of the class of \emph{Dulac} germs is motivated by the fact that the first return maps around hyperbolic saddle polycycles are analytic germs on open positive real line having a Dulac power-logarithm asymptotic expansions, as $x\to 0$, which extend holomorphically to a standard quadratic domain $\mathcal{R}_C$. Note that in Ilyashenko \cite{ilyalim} Dulac germs are called \emph{almost regular}. The existence of an extension to a sufficiently big complex domain was important in proving the fact that almost regular germs are uniquely determined by their Dulac asymptotic expansion, which was the key step in proving the non-accumulation theorem for limit cycles, see \cite{Dulac,ilyalim}. 

By \cite{ilyalim}, Dulac germs form a \emph{group} under the operation of composition. 
Here, we work with \emph{parabolic} Dulac germs, that is, with germs tangent to the identity: $f(z)=z+o(z)$. They form a subgroup under the operation of composition. We expect the \emph{hyperbolic} cases $f(z)=\lambda z+o(z)$, $|\lambda|\neq 1$, or $f(z)=Az^{\lambda_1}+o(z^{\lambda_1})$, $A\in\mathbb C$, $\lambda_1\neq 1$, to be analytically linearizable, as was the case of germs analytic at $0$, see e.g. \cite{gamelin}. Parabolic Dulac germs are thus the most interesting from the point of view of cyclicity, see a comment in the introduction of \cite{MRRZ2Fatou} for more details. 

\smallskip

In the forthcoming paper \cite{drugi}, we deal with the realization problem for moduli of analytic classification for parabolic Dulac germs. We solve there the realization problem in a bigger class of parabolic \emph{generalized Dulac germs}, which we introduce here. This class contains parabolic Dulac germs. The definition of this larger class is well-adapted to the solution of the realization problem considered in \cite{drugi}. 

For convenience, put:
$$
\boldsymbol\ell:=-\frac{1}{\log z}.
$$
Generalized Dulac expansions and germs are natural generalizations of Dulac expansions and germs: instead of polynomials in $\boldsymbol\ell^{-1}$ multiplying powers of $z$ in the Dulac expansion, in the generalized Dulac expansion we take \emph{series} in $\boldsymbol\ell$ that are \emph{canonically} summable on petals (in some generalized Gevrey sense) as analytic germs.
Here, in order to be able to uniquely define the asymptotic expansion of such a germ, the problem of canonical summation at limit ordinal steps occurs, since we do not request series in $\boldsymbol\ell$ multiplying each power of $z$ to be polynomial. This problem is solved by introducing the notion of \emph{$\log$-Gevrey summable series} in Section~\ref{sec:classes}. In Section~\ref{sec:proofB} we then construct the moduli of analytic classification for parabolic generalized Dulac germs.
\smallskip

Recall from \cite{mrrz2} that by $\widehat{\mathcal L}_1$ (or simply by $\widehat{\mathcal L}$) we denote the class of power-logarithmic transseries of the form:
$$
\widehat{f}(z)=\sum_{i=1}^{\infty}\sum_{m=N_i}^{\infty} a_{i,m} z^{\alpha_i} \boldsymbol\ell^{m},\ a_{i,m}\in\mathbb C,
$$
where $\alpha_i>0$ are \emph{finitely generated} and strictly increasing to $+\infty$, and $N_i\in\mathbb Z$, $i\in\mathbb N$.
By $\widehat{\mathcal L}_0$, we denote the subclass of $\widehat{\mathcal L}$ of power series, with strictly increasing, finitely generated powers. Furthermore, by $\widehat{\mathcal L}_k$, $k\in\mathbb N$, we denote the class of power-iterated logarithm transseries in variables $z,\,\boldsymbol\ell,\ \boldsymbol\ell_2=-\frac{1}{\log\boldsymbol\ell},\ldots,\boldsymbol\ell_k=-\frac{1}{\log \boldsymbol\ell_{k-1}}$, where powers of $z$ are strictly increasing and finitely generated, and powers of $\boldsymbol\ell_i$, $i=1,\ldots,k,$ belong to $\mathbb Z$. We put $\widehat{\mathfrak{L}}:=\cup_{i=0}^{\infty}\widehat{\mathcal L}_i$. If we allow the powers $\alpha_i$ of $z$ to start with negative powers, but stay strictly increasing, finitely generated and tending to $+\infty$, we denote the classes by $\widehat{\mathcal L}_k^\infty$, $k\in\mathbb N_0$, and $\widehat{\mathfrak L}^{\infty}$. The subset of all formal transseries  from $\widehat{\mathcal L}_k$ (or $\widehat{\mathcal L}_k^{\infty}$), $k\in\mathbb N_0$, whose \emph{leading term does not contain logarithm} is a group under composition. Note additionally that, if $\widehat f\in\widehat{\mathcal L}$ does not contain logarithm in the leading term, then $\boldsymbol\ell\mapsto -\frac{1}{\log \widehat f(e^{-1/\boldsymbol\ell})}$ is an exponential-power transseries in $\boldsymbol\ell$. By $\widehat{\mathcal L}_k^{inv}\subseteq \widehat{\mathcal L}_k$ we denote the subset of power-logarithm formal \emph{diffeomorphisms} (of the form $f(z)=cz+o(z),\ c\neq 0$), and by $\widehat{\mathcal L}_k^{\mathrm{id}}\subseteq \widehat{\mathcal L}_k$ we denote the subset of \emph{tangent to the identity} transseries (of the form $\widehat f(z)=z+\ldots$). Both are groups under composition. For more details, see \cite{mrrz2}. 

To express that the coefficients of transseries are \emph{real}, we will write simply $\widehat{\mathcal L}_k(\mathbb R),\ k\in\mathbb N_0$, and $\widehat{\mathfrak L}(\mathbb R)$. Parabolic Dulac and parabolic generalized Dulac series belong to $\widehat{\mathcal L}(\mathbb R)$. They both form groups under compositions, see Proposition~\ref{prop:grupa}.
\smallskip

In \cite[Theorem A]{mrrz2}, we have provided the formal classification theorem for transseries in $\widehat{\mathcal L}(\mathbb R)$. The results can be directly applied to formal classification of formal parabolic Dulac series and of formal parabolic generalized Dulac series. The formal classification is provided in the subgroup $\widehat{\mathcal L}^{inv}(\mathbb R)$ of $\widehat{\mathcal L}(\mathbb R)$ of power-logarithm formal diffeomorphisms (that is, of the form $f(z)=cz+o(z),\ c\neq 0$) with real coefficients and with finitely generated support. Let:
$$
f(z)=z-az^\alpha\boldsymbol \ell^{m}+o(z^\alpha\boldsymbol \ell^{m}),
$$
$\alpha>1$,\ $m\in\mathbb Z$, be a parabolic generalized Dulac series. Then, it is formally equivalent in $\widehat{\mathcal L}^{inv}(\mathbb R)$ to any of the normal forms:
\begin{align}
&f_F(z)=z- z^\alpha\boldsymbol\ell^{m}+\rho z^{2\alpha-1}\boldsymbol\ell^{2m+1},\nonumber\\
&f_1(z)=\text{Exp}(X_1).\mathrm{id},\ X_1(z)=\frac{-z^\alpha\boldsymbol\ell^{m}}{1+\frac{-\alpha}{2}z^{\alpha-1}\boldsymbol\ell^{m}+\big(\frac{m}{2}+\rho\big) z^{\alpha-1}\boldsymbol\ell^{m+1}}\frac{d}{dz}.\label{eq:norma}
\end{align}
The formal invariants are $(\alpha,m,\rho)$, $\alpha>1$, $m\in\mathbb Z$, $\rho\in\mathbb R$. Note that, if $f$ is normalized, that is, if $a=1$, then the reduction to the normal forms can be done in the smaller subgroup of $\widehat{\mathcal L}^{\mathrm{id}}(\mathbb R)$, of \emph{tangent to the identity} transseries. 

\noindent Let $f_c$, $c\in\mathbb R$, be the time-$c$ map of the vector field $X_1$:
$$f_c(z):=\mathrm{Exp}(cX_1).\mathrm{id}.$$ It is analytic on $\mathcal R_C$. Its (unique) power-logarithmic expansion will be denoted by $\widehat f_c(z)$.
\medskip

Note that the formal changes of variables needed for reduction to the formal normal form of a parabolic Dulac series exit the class of formal Dulac series and are transserial (belong to $\widehat{\mathcal L}(\mathbb R)$). Moreover, the blocks in $\boldsymbol\ell$ are in general not convergent, neither do they belong to the class of \emph{parabolic generalized Dulac series} that we introduce below in Definitionw~\ref{def:gD}. The class of parabolic generalized Dulac germs is here introduced as the class in which we obtain the \emph{full} realization result for analytic moduli in the forthcoming paper \cite{drugi} (realization of \emph{every} symmetric sequence of diffeomorphisms). For the subclass of parabolic Dulac germs, we expect a weaker realization result. The characterization of moduli that can be realized for true parabolic Dulac germs is left for future work. On the other hand, the formal changes of variables for parabolic Dulac and parabolic generalized Dulac series belong to a class of so-called \emph{block iterated integrally summable series}, introduced in Subsection~\ref{subsec:iis}. It is an even bigger subclass of $\widehat{\mathcal L}(\mathbb R)$, containing the parabolic generalized Dulac series. 

To be able to uniquely define generalized Dulac expansions in the definition of parabolic generalized Dulac germs below, we introduce in Section~\ref{sec:classes} the definition of \emph{$log$-Gevrey asymptotic expansions} on particular cusps that are $\boldsymbol\ell$-images of sectors, and state some of their properties. The definition is motivated by the notion of \emph{Gevrey asymptotic expansions of order $k$}, $k\in\mathbb N$, on sectors. For the precise definition of an asymptotic expansion on a sector and of Gevrey asymptotic expansion of some order on a sector, see e.g. \cite[Section 1]{Mloday} or \cite{bahlser}. The origin of the name \emph{$\log$-Gevrey} is related to the fact that $\log$-Gevrey asymptotic expansion on some domain is Gevrey of every order $k$, but not necessarily convergent at $0$, therefore being somewhere inbetween Gevrey-type and convergent.
\smallskip

We will prove in Theorem~A (1) in Section~\ref{sec:proofA} a more general fact that flower-like dynamics of a germ $f$ defined on a standard quadratic domain $\mathcal R_C$ follows merely from the assumptions that $f$ is an analytic germ on $\mathcal R_C$ which satisfies the uniform estimate for the leading terms:
\begin{align*}
\big|f(z)-(z-az^{\alpha}\boldsymbol\ell^m)\big|\leq  c |z^\alpha\boldsymbol\ell^{m+1}|,\ z\in\mathcal R_{C},\,\ a>0,\ c>0.
\end{align*} More precisely, we will prove in Section~\ref{sec:proofA} the following proposition:

\begin{prop}[Flower-like dynamics of a parabolic germ]\label{prop:leau} The invariant sets for the dynamics of a parabolic germ $f$ on $\mathcal R_C$ which satisfies: \begin{equation}\label{eq:q}|f(z)-z+az^\alpha\boldsymbol\ell^{m}|\leq c|z^\alpha\boldsymbol\ell^{m+1}|, \ c>0,\ z\in\mathcal R_C,\ \alpha>1,\ m\in\mathbb Z,\ a>0,\end{equation} are open \emph{attracting petals} $V_j^+$, $j\in\mathbb Z$, on $\mathcal R_C$ $($or possibly $\mathcal R_{C'}\subset\mathcal R_C$$)$, centered along the tangential directions $$1^{-\frac{1}{\alpha-1}},$$ and of opening $\frac{2\pi}{\alpha-1}$. Analogously, the invariant sets for the dynamics of the inverse germ  $f^{-1}(z)=z+az^\alpha\boldsymbol\ell^{m}+o(z^\alpha\boldsymbol\ell^{m})$ are open \emph{repelling} petals $V_j^-$, $j\in\mathbb Z$, on $\mathcal R_C$ $($or possibly $\mathcal R_{C'}\subset\mathcal R_C$$)$, centered along the tangential directions $$(-1)^{-\frac{1}{\alpha-1}},$$ and of opening $\frac{2\pi}{\alpha-1}$. 
At the intersections of attracting and repelling petals the orbits for the discrete dynamics are \emph{closed}.
\end{prop}

This in particular applies to parabolic generalized Dulac germs defined below. Therefore, in Definition~\ref{def:gD} below we may assume in advance the existence of attracting and repelling petals $V_j^\pm$ for the dynamics of $f$, $j\in\mathbb Z$. The petals are described in detail in Theorem~A (1).

\begin{defi}[Parabolic generalized Dulac germs]\label{def:gD} We say that a parabolic germ $f$, analytic on a standard quadratic domain $\mathcal R_C$,  that maps $\{\mathrm{arg}(z)=0\}\cap\mathcal R_C$ to $\{\mathrm{arg}(z)=0\}\cap\mathcal R_C$, satisfying \begin{align}\label{eq:uniest}|f(z)-z+az^{\alpha}\boldsymbol\ell^m|\leq c|z^{\alpha}\boldsymbol\ell^{m+1}|,\ a\neq 0,\ \alpha>1,\ m\in\mathbb Z, \ c>0,\ z\in\mathcal R_C,\end{align}  is a \emph{parabolic generalized Dulac germ} if, on its every invariant petal $V_j^{\pm},\ j\in\mathbb Z$, of opening $\frac{2\pi}{\alpha-1} $, it admits an asymptotic expansion of the form:
\begin{equation}\label{eq:gDasy}
f(z)=z+\sum_{i=1}^{n} z^{\alpha_i} R_i^{j,\pm}(\boldsymbol\ell)+o(z^{\alpha_n+\delta_n}), \ \delta_n>0,
\end{equation}
for every $n\in\mathbb N$, as $z\to 0$ on $V_j^{\pm}$. Here, $\alpha_1=\alpha$, $\alpha_i>1$ are strictly increasing to $+\infty$ and finitely generated, and $R_i^{j,\pm}(\boldsymbol\ell)$ are analytic functions on open cusps $\boldsymbol\ell(V_j^{\pm})$ which admit common \emph{$\log$-Gevrey asymptotic expansions}\footnote{Adaptation of \emph{Gevrey asymptotic expansions} of some order, see Section~\ref{sec:classes} for precise definitions.} $\hat R_i(\boldsymbol\ell)$ of order strictly bigger than $\frac{\alpha-1}{2}$, as $\boldsymbol\ell\to 0$:
$$
\hat R_i(\boldsymbol\ell)=\sum_{k=N_i}^{\infty}a_k^i \boldsymbol\ell^k,\ a_k^i\in\mathbb R,\ N_i\in\mathbb Z.
$$
\smallskip

We then say that the series $\widehat f\in\widehat{\mathcal L}$ given by:
$$
\widehat f(z):=z+\sum_{i=1}^{\infty} z^{\alpha_i} \widehat R_i(\boldsymbol\ell)
$$
is the generalized Dulac asymptotic expansion of $f$. Such $\widehat f$ is called a parabolic \emph{generalized Dulac series}.
\end{defi}

We prove in Proposition~\ref{prop:grupa} in Section~\ref{sec:proofA} that parabolic generalized Dulac germs form a group under composition.

\smallskip

Note that the assumption \eqref{eq:uniest} is automatically satisfied for parabolic Dulac germs (see Proposition~\ref{prop:unidulac} in Section~\ref{sec:proofA}), since Dulac asymptotic expansions are uniform on whole $\mathcal R_C$. For parabolic generalized Dulac germs, on the other hand, we request asymptotic expansions \eqref{eq:gDasy} only by petals, and the bounds are not necessarily uniform by petals. The additional uniform request \eqref{eq:uniest} is therefore important in their definition to ensure that the rate of decrease of radii of petals for the dynamics of a parabolic generalized Dulac germ by levels of the Riemann surface of the logarithm is not quicker than dictated by a standard quadratic domain, see Section~\ref{sec:proofA}. 
\smallskip

Note that a generalized Dulac asymptotic expansion is an asymptotic expansion in the class of transseries $\widehat {\mathcal L}(\mathbb R)$. As discussed in \cite{MRRZ2Fatou}, an asymptotic expansion of a germ in $\widehat{\mathcal L}(\mathbb R)$ is in general not unique and well-defined. The generalized Dulac expansion is a sectional asymptotic expansion that becomes unique after canonical choice of section functions (\emph{summation rule}) at limit ordinal steps. See \cite{MRRZ2Fatou} for precise definition of sectional asymptotic expansions on $\mathbb R$ and, for generalization to complex sectors, see Definition~\ref{def:assy} in the Appendix. The existence of a canonical choice of section functions is here guaranteed by the request that $R_i^{j,\pm}(\boldsymbol\ell)$ admit \emph{$\log$-Gevrey power asymptotic expansions of some order} on cusps which are $\boldsymbol\ell$-images of sectors of sufficiently big opening, called the \emph{$\boldsymbol\ell$-cusps}. Then, by Corollary~\ref{cory:vary} (variation of \emph{Watson's lemma}),  their \emph{$\log$-Gevrey sum} on each petal is \emph{unique}. This $\log$-Gevrey condition is a stronger condition than \emph{Gevrey summability of order $m$} \cite{Mloday}, for any $m>0$, see Section~\ref{sec:classes}. It ensures unique summability on $\boldsymbol\ell$-cusps, that do not contain sectors of any positive opening. The request of $\log$-Gevrey summability for parabolic generalized Dulac germs generalizes the polynomiality request for Dulac germs.

Parabolic Dulac germs are trivially parabolic generalized Dulac germs. In that case we have a canonical choice of polynomial sections. Polynomial functions in $\boldsymbol\ell^{-1}$ are convergent Laurent series in $\boldsymbol\ell$. 
\smallskip

In the forthcoming paper \cite{drugi} about the converse problem of realization of a sequence of diffeomorphisms as analytic moduli of power-logarithmic germs, as realizing germs we will obtain parabolic generalized Dulac germs.

\medskip

We additionally suppose in the sequel that $f(z)=z-az^\alpha\boldsymbol\ell^m+o(z^\alpha\boldsymbol\ell^m)$, where $a>0$, $\alpha>1$, $m\in\mathbb Z$, so that $\{\mathrm{arg}(z)=0\}\cap \mathcal R_C$ is an invariant attracting direction. In the case $a<0$, we switch to the inverse germ $f^{-1}$, which is by Proposition~\ref{prop:grupa} also a parabolic generalized Dulac germ. By a real homothecy $\varphi(z)=a^{-\frac{1}{\alpha-1}}z$, $a^{-\frac{1}{\alpha-1}}\in\mathbb R$, we transform $f(z)=z-az^\alpha\boldsymbol\ell^m+o(z^\alpha\boldsymbol\ell^m)$ to a \emph{normalized} germ \begin{equation}\label{eq:normali}f(z)=z-z^\alpha\boldsymbol\ell^m+o(z^\alpha\boldsymbol\ell^m).\end{equation} Note that we cannot by homothecy conjugation transform the first coefficient $-1$ to $+1$, since we would lose the invariance of the positive real line. Therefore, in the sequel, we always suppose that we work with a germ (parabolic Dulac or generalized Dulac) whose first coefficient is equal to $a=1$. This first normalization step, performed for simplicity, is not important from the viewpoint of analytic conjugacies of generalized Dulac germs discussed later, since homothecy  is an analytic germ on a standard quadratic domain. \smallskip

We will work simultaneously in the $z$-chart (which we will call the \emph{original chart}) and in the $(\zeta=-\log z)$-chart (which we will call the \emph{logarithmic chart}), choosing the chart which is more convenient.
\medskip

Here, we are interested in \emph{complete invariants of analytic classification} in the sense of Ecalle-Voronin moduli for analytic diffeomorphisms. We first define in Section~\ref{sec:ane} what we mean by saying that two parabolic (generalized) Dulac germs $f$ and $g$ are analytically conjugated. In the case of analytic germs $\mathbb C\{z\}$ it was clear that $f,\ g\in\mathbb C\{z\}$ are analytically equivalent if there exists a germ of an analytic diffeomorphism $\varphi\in\mathbb C\{z\}$ such that
$$
g=\varphi^{-1}\circ f\circ \varphi.
$$
Evidently, the analytic equivalence implies formal equivalence in the sense of formal Taylor series in $\widehat{\mathcal L}_0(\mathbb R)$. In this paper, we will \emph{respect} the $\widehat{\mathcal L}(\mathbb R)$-formal classes, and we adopt the following definition of analytic equivalence of parabolic generalized Dulac germs:
\begin{defi}[Analytic equivalence]\label{def:jedan} We say that two normalized parabolic generalized Dulac germs $f$ and $g$ of the form \eqref{eq:normali} defined on a standard quadratic domain $\mathcal R_C$  are \emph{analytically conjugated} if:
\begin{enumerate}
\item $f$ and $g$ are $\widehat{\mathcal L}(\mathbb R)$-formally conjugated\footnote{i.e. have the same formal invariants $(\alpha,m,\rho),\ \alpha>1,\ m\in\mathbb Z,\ \rho\in\mathbb R.$}, and
\item there exists a germ of diffeomorphism $h(z)=z+o(z)$ of a standard quadratic domain $\mathcal R_C$, such that:
$$
g=h^{-1}\circ f\circ h, \text{ on }\mathcal R_C.
$$
\end{enumerate}
\end{defi} 

If we allow non-normalized germs $f$ and $g$, then $h$ is possibly composed with additional homothecy, so $h$ is requested to be of the more general form $h(z)=cz+o(z)$, $c>0$. For simplicity, we assume that $f$ and $g$ are normalized and that $h$ is tangent to the identity. Two non-normalized germs are analytically conjugated if and only if their normalizations are analytically conjugated. Definition~\ref{def:jedan} is analyzed in more detail in Section~\ref{sec:ane}.
\smallskip

We prove in Section~\ref{subsec:types} the following Proposition~\ref{prop:natkonj}. It explains why it is natural to define analytic conjugacy of two parabolic generalized Dulac germs in Definition~\ref{def:jedan} just as a tangent to the identity bounded analytic germ on a standard quadratic domain, \emph{without any further asymptotic condition}. Proposition~\ref{prop:natkonj} states that it is automatically \emph{coherent} with the formal conjugacy in $\widehat{\mathcal L}^{inv}(\mathbb R)$.
\begin{prop}\label{prop:natkonj} Let $f$ and $g$ be two normalized parabolic generalized Dulac germs, analytically conjugated on a standard quadratic domain $\mathcal R_C$ $($in the sense of Definition~\ref{def:jedan}$)$. Let $h(z)=z+o(z)$, $z\to 0$, $z\in\mathcal R_C$, be their analytic conjugacy. Let $\widehat h\in\widehat {\mathcal L}^{\mathrm{id}}(\mathbb R)$ be their formal conjugacy:
$$
\widehat g=\widehat h^{-1}\circ \widehat f\circ\widehat h.
$$ 
Then $h$ admits $\widehat h$ as its \emph{generalized block iterated integral asymptotic expansion}\footnote{For notion of a \emph{generalized block iterated integral asymptotic expansion}, see Section~\ref{subsec:types}}, up to non-uniqueness of formal conjugacy $\widehat h$ described after Proposition~\ref{prop:gDul}.
\end{prop}

In Section~\ref{subsec:types}, we define in which sense an analytic conjugacy of two parabolic (generalized) Dulac germs expands asymptotically in $\widehat{\mathcal L}(\mathbb R)$. Indeed, due to the fact that we deal with transserial asymptotic expansions, which are not unique, we need to determine the section function (a \emph{summation rule}) that we use at each limit ordinal step. The conjugacies are not again parabolic generalized Dulac germs, but more complicated. We introduce \emph{integral sums of length $k$, $k\in\mathbb N_0,$} in Definition~\ref{def:isk} and the notion of \emph{generalized block iterated integral summability} in Definitions~\ref{def:giis} and \ref{def:giiaa}. We prove in Proposition~\ref{prop:gic} that any analytic conjugacy of two parabolic (generalized) Dulac germs admits a formal conjugacy in $\widehat{\mathcal {L}}(\mathbb R)$ as its generalized block integral asymptotic expansion, up to non-uniqueness of the formal conjugacy described after Proposition~\ref{prop:gDul}.
\medskip

\section{Main theorems}

Theorem~A describes the local dynamics of a parabolic (generalized) Dulac germ on a standard quadratic domain $\mathcal R_C$ (considered on the Riemann surface of the logarithm or in the logarithmic chart $\zeta=-\log z$, as the ramified neighborhood of $z=0$ or $\text{Re}(\zeta)=+\infty$). The local dynamics resembles the well-known flower dynamics for parabolic analytic diffeomorphisms given by the \emph{Leau-Fatou flower theorem}, see e.g. \cite{loray}, but with countably many tangential attracting and repelling directions situated equidistantly on the Riemann surface of the logarithm.
\begin{thmA}[Leau-Fatou for parabolic generalized Dulac germs]\label{thm:leau}  
Let $$f(z)=z-az^\alpha\boldsymbol\ell^m+o(z^\alpha\boldsymbol\ell^m),\ \alpha>1,\ a>0,\ m\in\mathbb Z,$$ be a \emph{parabolic generalized Dulac germ} on a standard quadratic domain $\mathcal R_C$, $C>0$. 
\begin{enumerate}
\item There exist countably many overlapping attracting and repelling open petals\footnote{Classically, as in \cite{loray}, a petal of opening $\theta\in(0,2\pi]$ is a union of open sectors of openings continuously increasing towards $\theta$, with continuously decreasing radii.} $V_i^+$ resp. $V_i^-$, $i\in\mathbb Z$, for $f$ on $\mathcal R_C$, which are maximal invariant domains for the dynamics of $f$ resp. $f^{-1}$. They are centered at complex directions $1^{\frac{1}{\alpha-1}}$ resp. $(-1)^{\frac{1}{\alpha-1}}$ on $\mathcal R_C$, which are tangential directions for the dynamics of $f$ resp. $f^{-1}$. The opening angle of each petal is equal to $\frac{2\pi}{\alpha-1}$.

\item On each open petal $V_i^{\pm}$, there exists a $($unique up to an additive constant$)$ holomorphic Fatou coordinate $\Psi_{i}^{\pm}$, which conjugates $f$ to the translation:
$$
\Psi_i^{\pm}(f(z))-\Psi_i^{\pm}(z)=1,\ z\in V_i^{\pm},
$$
and admits a formal Fatou coordinate $\widehat \Psi\in\widehat{\mathcal L}_2^\infty(\mathbb R)$ as its integral asymptotic expansion\footnote{For details, see Section~\ref{subsec:dva}.}, as $z\to 0$ $($up to a constant$)$. 

\item On each open petal $V_i^{\pm}$, the germ $f$ is conjugated analytically by a germ $\varphi_i^{\pm}(z)=a^{-\frac{1}{\alpha-1}}z+o(z)$, to its formal normal form $f_1=\mathrm{Exp}(X_1).\mathrm{id}$ from \eqref{eq:norma}. The conjugation on each petal is unique, up to a precomposition $\varphi_i\circ f_c$ with time-$c$ map, $c\in\mathbb R$, of vector field $X_1$. Its block iterated integral asymptotic expansion\footnote{For details, see Section~\ref{subsec:dva}. Here, this expansion is considered after postcomposition of $\widehat\varphi$ and $\varphi_i^{\pm}$ by the inverse of the first change of variables that is a homothecy, that is, by $\varphi_0(z)=a^{\frac{1}{\alpha-1}}z$, since we work here with non-normalized germs.} is, up to precomposition with $\widehat f_c$, equal to the formal conjugacy $\widehat{\varphi}\in\widehat{\mathcal L}^{inv}(\mathbb R)$.
\end{enumerate}

\end{thmA}
The formal Fatou coordinate $\widehat\Psi$ is a formal solution of Abel equation for the parabolic generalized Dulac expansion $\widehat f$ of $f$:
$$
\widehat{\Psi}(\widehat f)-\widehat\Psi=1.
$$
The proof of Theorem A is given is Section~\ref{sec:proofA}. As was done in \cite{MRRZ2Fatou} for the Fatou coordinate of a parabolic Dulac germ, we prove in Section~\ref{sec:proofA} that the formal Fatou coordinate of a parabolic generalized Dulac germ is unique (up to a constant) in $\widehat{\mathfrak L}(\mathbb R)$, and belongs moreover to $\widehat{\mathcal L}_2^\infty(\mathbb R)$. 
\medskip

Note that we prove in Proposition~\ref{prop:leau} the statement of Theorem~A (1) under weaker conditions on the germ $f$, not assuming the full generalized Dulac asymptotic expansion, but just the beginning of it. However, the proof of the existence of the petalwise analytic Fatou coordinate (and then also of the sectorial conjugacy) in Section~\ref{sec:proofA} relies on the existence of the transasymptotic expansion of a generalized Dulac germ $f$ \emph{at least up to the residual term} $z^{2\alpha-1}\boldsymbol\ell^{2m+1}$. This enables us to deduce the infinite part of the Fatou coordinate in the proof of Theorem A (2) and then prove the convergence and analyticity of the infinitesimal remainder using a modified Abel equation.
\medskip 

The local dynamics of a parabolic generalized Dulac germ $f$ is illustrated on a standard quadratic domain in the logarithmic chart in Figure~\ref{figu:fig3}.
\begin{figure}[h!]
\includegraphics[scale=0.2]{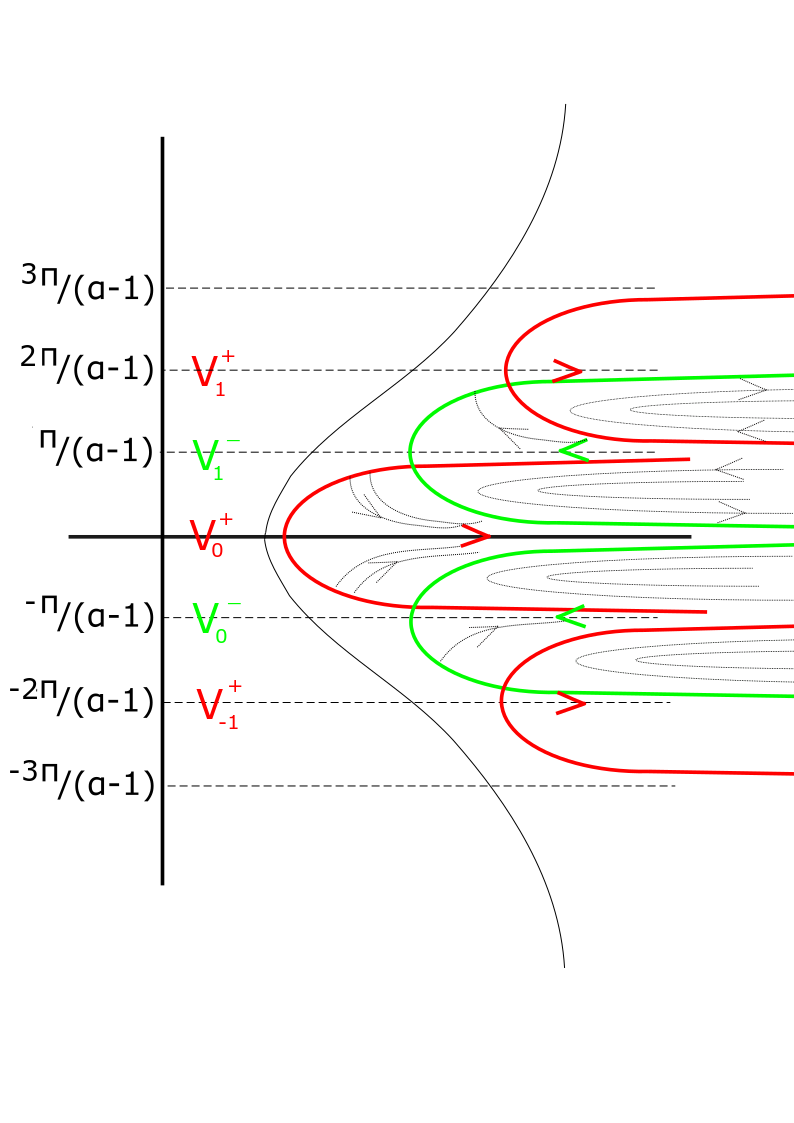}
\vspace{-1cm}
\caption{Local dynamics of a parabolic generalized Dulac germ in the logarithmic chart $\zeta=-\log z$.}\label{figu:fig3}
\end{figure}
\smallskip

In Theorem B we describe the complete system of analytic invariants for a normalized parabolic generalized Dulac germ $f$, with respect to analytic conjugation $\varphi(z)=z+o(z),\ z\to 0,$ on a standard quadratic domain. This result is similar to the result of Ecalle, Voronin \cite{voronin} for parabolic analytic germs of diffeomorphisms, where the analytic class was given by the formal class and a finite number of analytic diffeomorphisms (analytic moduli), after appropriate identifications, see \cite{voronin} or \cite{loray}.

Let $f$ be a parabolic normalized generalized Dulac germ defined on a standard quadratic domain. We suppose for simplicity that $\alpha=2$. This can always be done without loss of generality, since we can apply an analytic change of variables $h(z)=(\alpha-1)^{-\frac{m}{\alpha-1}}z^{\frac{1}{\alpha-1}}$ from one standard quadratic domain to another, see Proposition~\ref{prop:dva} for more details. It is just more convenient, since every level of the surface of the logarithm contains only one petal of opening $2\pi$. It is then easier to express the diminishing of the size of the domain of the definition of analytic moduli, see \eqref{eq:fall} below. The case $\alpha>1,\ \alpha\neq 2$, can be analyzed in the same way, but then a finite number of moduli that correspond to one level of the surface of the logarithm ($\lfloor2\pi/\alpha\rfloor$ of them) have the same radius. 

Let
$$
f(z)=z-z^2\boldsymbol\ell^m+o(z^2\boldsymbol\ell^m),\ a>0,\ m\in\mathbb Z,\ z\in\mathcal{R}_C,
$$
be a parabolic generalized Dulac germ.
Let $(\Psi_{\pm}^i)_{i\in\mathbb Z}$ be its analytic Fatou coordinates on attracting i.e. repelling petals $V_{\pm}^i$ (of opening $2\pi$) along the domain. We prove in Section~\ref{sec:proofB} that there exists a \emph{symmetric} (with respect to $\{\mathrm{arg}(z)=0\}$-axis) sequence $(h_{0,\infty}^i)_{i\in\mathbb Z}$ of analytic germs of diffeomorphism, called the \emph{horn maps} for $f$, that satisfy:
\begin{align*}
&h_0^i(t):=e^{-2\pi i \Psi_+^{i-1}\circ (\Psi_{-}^{i})^{-1}(-\frac{\log t}{2\pi i})},\ t\in(\mathbb C,0),\\
&h_\infty^i(t):=e^{2\pi i \Psi_-^i\circ (\Psi_{+}^i)^{-1}(\frac{\log t}{2\pi i})},\ t\in(\mathbb C,0),\ i\in\mathbb Z,
\end{align*}
and that this sequence and the formal class $(2,m,\rho)$ are a \emph{complete system of analytic invariants} of a parabolic generalized Dulac germ $f$.

Due to the standard quadratic domain of definition of $f$, the radii of definition $R_i$ of the sequence of its horn maps are bounded from below by:
\begin{equation}\label{eq:fall}
R_i\geq K_1 e^{-Ke^{C\sqrt{|i|}}},\ K_1,\ K,\ C>0,\ i\in\mathbb Z.
\end{equation}
Also the converse holds: if all the horn maps are the identities defined at least on the discs of radii given by \eqref{eq:fall}, then the sectorial Fatou coordinates glue to an analytic map at least on a standard quadratic domain.

\begin{thmB}[Analytic invariants of a parabolic generalized Dulac germ]\label{thm:modi} Let $f$ and $g$ be two normalized parabolic generalized Dulac germs that belong to the same $\widehat{\mathcal L}(\mathbb R)$-formal class $(2,m,\rho)$, $m\in\mathbb Z$, $\rho\in\mathbb R$. Let $$(h_{0}^i,h_{\infty}^i;\ R_i^f)_{i\in\mathbb Z} \text{ and } (k_{0}^i,k_{\infty}^i;\ R_i^g)_{i\in\mathbb Z}$$ be their sequences of horn maps  with radii of convergence $(R_i^f)_i$ resp. $(R_i^g)_i$, where $R_i^f$ and $R_i^g$ satisfy bounds of the type \eqref{eq:fall}. Then $f$ and $g$ are analytically conjugated (by tangent to the identity analytic change of variables) on some standard quadratic domain if and only if there exist sequences $(a_i)_i,\,(b_i)_i\in\mathbb C^*$ such that:
\begin{equation}\label{eq:ekvij} h_0^i(t)=a_{i-1}\cdot k_0^i\big(b_i t\big),\ h_\infty^i(t)=b_i\cdot  k_\infty^i\big(a_{i} t\big),\ t\in(\mathbb C,0),\ i\in\mathbb Z.\end{equation} 
\end{thmB}

\noindent More details and the proof of Theorem~B are given in Section~\ref{sec:proofB}. We also show in Proposition~\ref{prop:sim} in Section~\ref{sec:proofB} that the horn maps of a parabolic generalized Dulac germ from a $\mathcal L(\mathbb R)$-formal class $(2,m,\rho)$, due to realness of coefficients, are \emph{symmetric} with respect to the real axis. That is, 
\begin{equation*}\label{eq:simimm}
\big(h_{0}^{-i+1}\big)^{-1}(t)\equiv \overline{h_{\infty}^i(\,\overline t\,)},\ i\in\mathbb Z,
\end{equation*}
on their domains of definition.
\smallskip

Note that the request that the germs are \emph{normalized} as in \eqref{eq:normali} is just convenient for simplicity, but not important. Indeed, two parabolic generalized Dulac germs are analytically conjugated by a change of variables $\varphi(z)=bz+o(z),\ z\to 0,$ $b>0$, on a standard quadratic domain, if and only if their normalizations are analytically conjugated by a tangent to the identity change of variables on a standard quadratic domain. Indeed, we only need one additional homothecy that is analytic on a standard quadratic domain. Therefore it suffices to consider the moduli and the question of analytic conjugacy only for normalized generalized Dulac germs. That is, to consider conjugacies that are tangent to the identity.

\section{$\log$-Gevrey classes $LG_m(S)$ on $\boldsymbol\ell$-cusps.} \label{sec:classes}

Here we define $\log$-Gevrey classes of germs and formal series used in Definition~\ref{def:gD} of parabolic generalized Dulac germs. We also state and prove some of their properties: closedness to summation, almost closedness to multiplication and to differentiation. That will be needed for the description of the Fatou coordinate and of sectorial analytic reductions of parabolic generalized Dulac germs in Section~\ref{subsec:dva} and later in \cite{drugi}. The proofs are somewhat similar to the proofs in \cite{bahlser, Mloday} for classical Gevrey classes of functions, which are differential algebras. 
\smallskip

In the sequel, we will call an \emph{$\boldsymbol\ell$-cusp} an open cusp that is the image of an open sector $V$ of positive opening at $0$ by the change of variables $\boldsymbol\ell=-\frac{1}{\log z}$, and we will denote it by $S=\boldsymbol\ell(V)$. See Figure~\ref{figura:fig3}. Any open $\boldsymbol\ell$-cusp $\boldsymbol\ell(V')\subset S$, where $V'\subset V$ is a proper subsector, will be called a \emph{proper $\boldsymbol\ell$-subcusp} of $S$.

\begin{figure}[h!]
\includegraphics[scale=0.2]{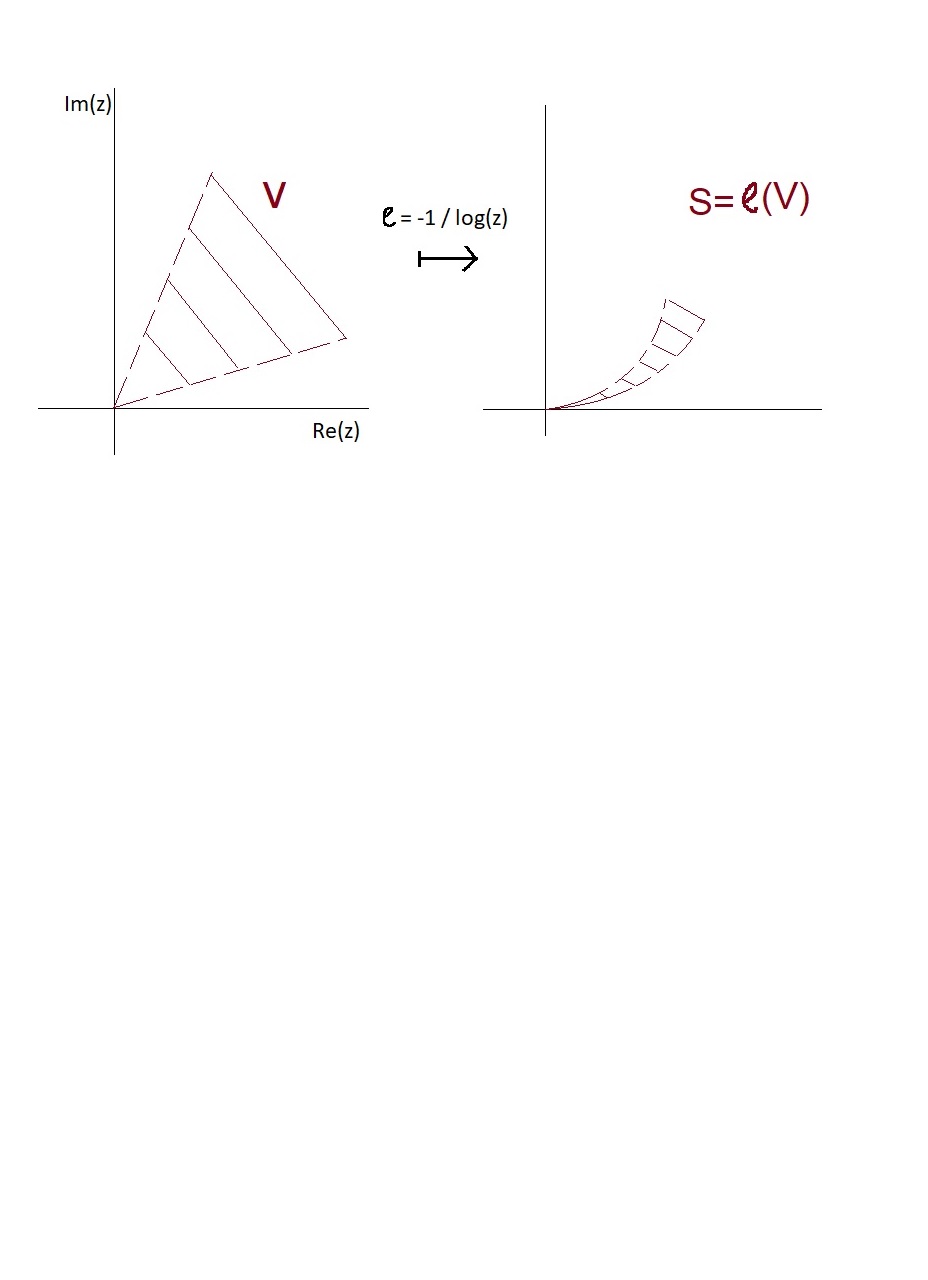}
\vspace{-4.5cm}
\caption{$\boldsymbol\ell$-cusp.}\label{figura:fig3}
\end{figure}

Such $\boldsymbol\ell$-cusps will inherit the property of differentiation of power asymptotic expansions term by term from the corresponding sectors of positive opening (a property that cannot be claimed for $\mathbb R_+$ or for general cusps), as will be shown in Corollary~\ref{cor:diffe} below. 


\begin{defi}[$\log$-Gevrey asymptotic expansions on $\boldsymbol\ell$-cusps]\label{def:df} Let $F$ be a germ analytic on an $\boldsymbol\ell$-cusp $S=\boldsymbol\ell(V)$. We say that $F$ admits $\widehat F(\boldsymbol\ell)=\sum_{k=0}^{\infty} a_k \boldsymbol\ell^k$, $a_k\in\mathbb C$, as its \emph{$\log$-Gevrey asymptotic expansion of order $m>0$} if, for every proper $\boldsymbol\ell$-subcusp $S'=\boldsymbol\ell(V')\subset S$, $V'\subset V,$ there exists a constant $C_{S'}>0$ such that, for every $n\in\mathbb N$, $n\geq 2,$ it holds that:
\begin{equation}\label{eq:lg1}
|F(\boldsymbol\ell)-\sum_{k=0}^{n-1}a_k \boldsymbol\ell^k|\leq C_{S'} \cdot m^{-n}\cdot (\log n)^n \cdot e^{-\frac{n}{\log n}}|\boldsymbol\ell|^n,\ \boldsymbol\ell\in S'.
\end{equation}
By $LG_m(S)\subseteq\mathcal O(S)$, $S$ an $\boldsymbol\ell$-cusp, $m>0$, we denote the set of all germs analytic on $S$ that admit a $\log$-Gevrey asymptotic expansion of order $m>0$ on $S$, as $\boldsymbol\ell\to 0$. By $\widehat{LG}_m(S)\subseteq \mathbb C[[\boldsymbol\ell]]$, we denote the set of their $\log$-Gevrey asymptotic expansions of order $m>0$ on $S$.
\end{defi}

For the definition of generalized Dulac germs, we will need the following generalization of the definition.
\begin{obs}[Generalization of Definition~\ref{def:df} to Laurent formal expansions] We will say that $F$ analytic on an $\boldsymbol\ell$-cusp $S$ admits a formal \emph{Laurent series} $\widehat F(\boldsymbol\ell)\in\mathbb C((\boldsymbol\ell))$\,\footnote{By $\mathbb C((\boldsymbol\ell))$ we denote the fraction field of $\mathbb C[[\boldsymbol\ell]]$, that is, the set of all formal Laurent series in $\boldsymbol\ell$ with complex coefficients. Similarly for $\mathbb R((\boldsymbol\ell)).$}, $\widehat F(\boldsymbol\ell)=\sum_{k=N}^{\infty}a_k\boldsymbol\ell^k$, $N\in\mathbb Z$, as its $\log$-Gevrey expansion of order $m>0$, if $F(\boldsymbol\ell)\boldsymbol\ell^{-k}$ admits $\widehat F(\boldsymbol\ell)\boldsymbol\ell^{-k}$ as its $\log$-Gevrey expansion of order $m>0$ on $S$ in the sense of the above definition. We will denote by the same notation $LG_m(S)$ the set of all such germs, and by $\widehat{LG}_m(S)\subseteq \mathbb C((\mathbb \ell))$ the set of their $\log$-Gevrey asymptotic expansions of order $m>0$.
\end{obs}

Recall that we say that $F(z)$ admits $\hat F(z)$ as its Gevrey asymptotic expansion of order $k>0$ on sector $V$ if, for every proper subsector $V'\subset V$, there exist constants $C,\ D>0$ such that, for every $n\in\mathbb N$, it holds that:
$$
|F(z)-\sum_{k=0}^{n-1}a_k z^k|\leq C D^n (n!)^{1/k} |z|^n,\ z\in V'.
$$
We can now check using \eqref{eq:lg1} and \emph{Stirling's formula} $\sqrt n (n/e)^n\sim n!,\ n\to\infty,$ that $\log$-Gevrey asymptotic expansion of some order $m>0$ is a Gevrey asymptotic expansion of \emph{every order\footnote{This can be seen by writing: $(\log n)^ne^{-\frac{n}{\log n}}=e^{n\log\log n-\frac{n}{\log n}}$ and, by Stirling's formula, $(n!)^{\frac{1}{k}}=e^{\frac{1}{2k}\log n+\frac{n}{k}\log n-\frac{n}{k}}$, $k>0$, and comparing the order of growth of the exponents as $n\to\infty$.} $k>0$}, as introduced in \cite{Mloday, bahlser, loray2}. 

Here, we have to consider functions on $\boldsymbol\ell$-cusps that do not contain any sectors of positive opening. This is why we need a stronger growth estimate, $\log$-Gevrey, in order to apply \emph{Watson-type} lemma. 
Recall that, classically, if $F$ admits zero Gevrey asymptotic expansion of order $k>0$ on $V$, then, for every proper subsector $V'\subset V$, there exist constants $C>0$ and $M>0$ such that:
$$
|F(z)|\leq C e^{-\frac{M}{|z|^{k}}},\ z\in V'.
$$
Also, classical \emph{Watson's lemma} guarantees that, if $F$ is an analytic function with zero Gevrey asymptotic expansion of order $k>0$ on a sector of opening strictly bigger than $\frac{\pi}{k}$, then $F$ is necessarily zero on this sector.

\begin{prop} Let $F$ admit zero $\log$-Gevrey asymptotic expansion of order $m>0$ on an $\boldsymbol\ell$-cusp $S=\boldsymbol\ell(V)$. Then, for every $\boldsymbol\ell$-subcusp $S'\subset S$ and every $\delta>0$, there exist constants $C,\ M>0$, such that
\begin{equation}\label{eq:c}
|F(\boldsymbol\ell)|\leq Ce^{-Me^{\frac{m-\delta}{|\boldsymbol\ell|}}},\ \boldsymbol\ell\in S'.
\end{equation}
\end{prop}
\begin{proof} Let $S'\subset S$. Putting $|\boldsymbol\ell|=a$ and zero expansion in \eqref{eq:lg1}, we get:
\begin{equation}\label{eq:min}
|F(\boldsymbol\ell)|\leq C m^{-n} (\log n)^n \cdot e^{-\frac{n}{\log n}} a^n,\ |\boldsymbol\ell|=a, \ n\in\mathbb N,\ n\geq 2.
\end{equation}
We maximize by $n=n(a)$ (we take the logarithm, differentiate by $n$ and find the critical point $n_a$). We get $n_a\sim e^{m/a}$ and $\log n_a\sim \frac{m}{a}$, as $a\to 0$. Now from \eqref{eq:min} we get that there exists $C>0$ such that:
$$
|F(\boldsymbol\ell)|\leq C e^{-\frac{1}{m} a e^\frac{m}{a}},\ \boldsymbol\ell\in S'.
$$
In other words, putting back $a=|\boldsymbol\ell|$, 
$$
|F(\boldsymbol\ell)|\leq  C e^{-\frac{1}{m}|\boldsymbol\ell|e^\frac{m}{|\boldsymbol\ell|}},\ \boldsymbol\ell\in S'.
$$
The result follows.
\end{proof}

\begin{cory}[Variation of \emph{Watson's lemma} for $\log$-Gevrey expansions]\label{cory:vary} Let $m>0$. Let $F(\boldsymbol\ell)$ be analytic on $\boldsymbol\ell$-cusp $S_m=\boldsymbol\ell(V_m)$, where $V_m$ is a sector of opening strictly bigger than $\frac{\pi}{m}$. Let $F$ admit the zero asymptotic expansion $\log$-Gevrey of order $m$ on $S_m$. Then $F$ is equal to zero on $S_m$.
\end{cory}
\begin{proof}
Take any $\boldsymbol\ell$-subcusp $S'=\boldsymbol\ell(V')\subset S_m$. Here, $V'$ is any subsector of $V$ of opening strictly less than the opening of $V_m$. By \eqref{eq:c}, since $\widehat F(\boldsymbol\ell)=0$, we get that, for any $\delta>0$, there exist constants $C,\ M>0$ such that
$$
|F(\boldsymbol\ell)|\leq Ce^{-Me^{\frac{m-\delta}{|\boldsymbol\ell|}}},\ \boldsymbol\ell\in S'.
$$
Passing to the variable $z=e^{-1/\boldsymbol\ell}$ and putting $\tilde F(z):=F(\boldsymbol\ell)$, we get ($|\log z|\geq -\log |z|,\ z\in\mathbb C$):
$$
|\tilde F(z)|\leq C e^{-\frac{M}{|z|^{m-\delta}}},\ z\in V'\subset V_m.
$$
Now take $V'$ of opening strictly bigger than $\pi/m$. Then take $\delta>0$ sufficiently small, such that $V'$ is moreover of opening strictly bigger than $\pi/(m-\delta)$. By classical \emph{Watson's lemma} (see \cite{bahlser}), we get $\tilde F\equiv 0$ on $V'$. Therefore, $\tilde F\equiv 0$ on whole $V_m$ by the uniqueness of the analytic extension.
\end{proof}

Unlike $m$-Gevrey classes, see e.g. \cite{Mloday}, or \cite{sauzin} for $1$-Gevrey, the $\log$-Gevrey classes $LG_m(S)$, $m>0$, $S$ an $\boldsymbol\ell$-cusp, from Definition~\ref{def:df},  are not differential algebras. However, we prove a weaker statement about closedness to operations $+,\cdot,\frac{d}{d\boldsymbol\ell}$ in Propositions~\ref{prop:closum} - \ref{prop:clodif}:
\begin{prop}[$\log$-Gevrey classes under summation]\label{prop:closum} Let $m>0$ and let $S=\boldsymbol\ell(V)$ be an $\boldsymbol\ell$-cusp, such that the opening of $V$ is strictly bigger than $\frac{\pi}{m}$. The sets $LG_m(S)$ and $\widehat{LG}_m(S)$, are groups by operation $+$. The mapping $f\mapsto\widehat f$ that attributes $\widehat f\in \widehat{LG}_m(S)$ to $f\in LG_m(S)$ is an isomorphism of these groups, respecting the operation $+$. 
\end{prop}
\begin{proof} The uniqueness of the $\log$-Gevrey asymptotic expansion is obvious. The injectivity follows by Corollary~\ref{cory:vary}.
The homomorphism property is obvious by the definition of $\log$-Gevrey asymptotic expansions and formula \eqref{eq:lg1}.
\end{proof}

\begin{prop}[$\log$-Gevrey classes under multiplication]\label{prop:clomult} Let $m,\,n>0$. Let $S=\boldsymbol\ell(V)$ be an $\boldsymbol\ell$-cusp. Let $f\in LG_m(S)$ and $g\in LG_n(S)$, with asymptotic expansions $\widehat f\in\widehat{LG}_m(S)$ and $\widehat g\in\widehat{LG}_n(S)$. Let $s:=\min\{m,n\}$. Then $fg\in LG_r(S)$ for every $0<r<s$. Moreover, $fg$ admits the formal product $\widehat f\cdot \widehat g$ as its $\log$-Gevrey asymptotic expansion of order $r$ on $S$.
\end{prop}

\noindent The proof is in the Appendix.

\begin{prop}[$\log$-Gevrey classes under differentiation]\label{prop:clodif} Let $f\in LG_m(S)$ admit $\widehat f(\boldsymbol\ell)$ as a $\log$-Gevrey asymptotic expansion of order $m$ on an $\boldsymbol\ell$-cusp $S=\boldsymbol\ell(V)$. Then, for every $0<r<m$, all derivatives $f^{(k)}$, $k\in\mathbb N$, belong to $LG_r(S)$ and admit formal derivatives $\widehat f^{(k)}(\boldsymbol\ell)$ as their $\log$-Gevrey asymptotic expansions of order $r$ on $S$.
\end{prop}

\noindent The proof is in the Appendix.

\medskip

Note that, in the course of the proof of Proposition~\ref{prop:clodif}, we have also proved the following corollary. We use the Cauchy formula in the variable $z$, similarly as in the proof of Proposition~\ref{prop:clodif}.
\begin{cory}[Termwise differentiation of power asymptotic expansions on $\boldsymbol\ell$-cusps]\label{cor:diffe} The power asymptotic expansions on $\boldsymbol\ell$-cusps $S:=\boldsymbol\ell(V)$, where $V$ is a sector at $0$ of positive opening, $f(\boldsymbol\ell)\sim \sum_{n=0}^{\infty} a_n \boldsymbol\ell^n,\ \boldsymbol\ell\in S,\ \boldsymbol\ell\to 0,$ may be differentiated term by term. That is, $$\frac {d}{d\boldsymbol\ell}f(\boldsymbol\ell)\sim \sum_{n=1}^{\infty} n a_n \boldsymbol\ell^{n-1},\ \boldsymbol\ell\in S,\ \boldsymbol\ell\to 0.$$
\end{cory}
\noindent More precisely, we say here that $f$ admits an asymptotic expansion $\widehat f(\boldsymbol\ell)=\sum_{n=0}^{\infty} a_n \boldsymbol\ell^n$ on an $\boldsymbol\ell$-cusp $S=\boldsymbol\ell(V)$ if, for every proper $\boldsymbol\ell$-subcusp $\boldsymbol\ell(V')\subset S$, where $V'\subset V$ is a proper subsector, there exists a constant $C_n>0$ such that:
$$
\Big|f(\boldsymbol\ell)- \sum_{k=0}^{n-1} a_k \boldsymbol\ell^{k}\Big|\leq C_n|\boldsymbol\ell|^n,\ \boldsymbol\ell\in S',\ n\in\mathbb N.
$$ 
\medskip

Note that, classically, by the Cauchy formula, we need a sector of positive opening to be able to differentiate asymptotic expansions term by term, see e.g. \cite{wasow}. Here, the property is inherited due to the special form of $\boldsymbol\ell$-cusps, that are $\boldsymbol\ell$-images of sectors.

\begin{prop}[$\log$-Gevrey classes under compositions with analytic germs]\label{prop:clocomp}
Let $S=\boldsymbol\ell(V)$ be an $\boldsymbol\ell$-cusp\footnote{of arbitrarily small radius} . Let $f,\,h\in LG_m(S)$, $g\in LG_p(S)$, $m,\,p>0$, with $\log$-Gevrey expansions $\widehat f,\ \widehat h\in\widehat{LG}_m(S)$ and $\widehat g\in\widehat{LG}_p(S)$. Let $f(\boldsymbol\ell)=o(1),$ as $\boldsymbol\ell\to 0$ on $V$. Let $F$ be an analytic or meromorphic germ at $0$, with Taylor i.e. Laurent expansion $\widehat F$ at $0$. 

\begin{enumerate}

\item $F\circ f\in LG_r(S)$, for every $0<r<m$, with $\log$-Gevrey expansion $\widehat F\circ \widehat f\in\widehat{LG}_r(S)$ of order $r$.

\item If $\widehat g\neq 0$, then $\frac{h}{g}\in LG_q(S)$\footnote{Note that $\frac{h}{g}$ is well-defined and analytic on $\boldsymbol\ell$-cusp $S$ (of sufficiently small radius), since $g$ cannot have an accumulation of zero points in $S$ at $0$. Otherwise, it is easy to see that its asymptotic expansion $\widehat g$ in power-log scale on $S$ would be $0$.}, for every $0<q<\min\{m,p\}$, and admits $\frac{\widehat h}{\widehat g}\in \widehat{LG}_q(S)$ as its $\log$ Gevrey asymptotic expansion of order $q$.
\end{enumerate}
\end{prop}

\noindent The proof is in the Appendix.

\section{Analytic equivalence of parabolic generalized Dulac germs}\label{sec:ane}
Note that in the case of regular parabolic diffeomorphisms analytic equivalence implies formal equivalence. For parabolic generalized Dulac germs $f$ defined on a standard quadratic domain $\mathcal R_C$ the situation is more complicated. We will see that a tangent to identity analytic conjugacy of two parabolic (generalized) Dulac germs on $\mathcal R_C$ will not necessarily imply the asymptotic expansion of the conjugacy germ in $\widehat{\mathcal L}(\mathbb R)$. However, this will be true in $\widehat{\mathcal L}_2(\mathbb R)$, see Lemma~\ref{lem:loic}. The reason is the possibility of formal reduction of $\widehat f(z)$ to a simpler normal form by formal conjugacy belonging to a wider class $\widehat{\mathcal L}_2(\mathbb R)\supset \widehat{\mathcal L}(\mathbb R)$, that allows also elimination of the residual term, as in Remark~\ref{rem:elim} below. 

However, in this paper, when we define the \emph{analytic equivalence} of two parabolic generalized Dulac germs, we \emph{respect} the $\widehat{\mathcal L}(\mathbb R)$-formal classes. The analytic equivalence of two parabolic generalized Dulac germs is defined in Definition~\ref{def:jedan} in Section~\ref{sec:introduction}.
\smallskip

In the following example we show that, in Definition~\ref{def:jedan}, $(1)$ does not follow from $(2)$. Indeed, if only $(2)$ is satisfied, then $h$ allows a generalized block iterated integral expansion in the larger class $\widehat{\mathcal L}_2^{\mathrm{id}}(\mathbb R)$, a subgroup of $\widehat{\mathcal L}_2(\mathbb R)$ of tangent to the identity transseries.

\begin{example}\label{ex:pr} There exist parabolic generalized Dulac germs that verify $(2)$ in Definition~\ref{def:jedan}, but are not formally equivalent in $\widehat {\mathcal L}(\mathbb R)$. They are in fact formally equivalent in $\widehat{\mathcal L}_2(\mathbb R)$. This wider class can be used to eliminate the residual invariant, see Proposition~\ref{elim} below.
Take e.g. two parabolic generalized Dulac germs that are models for two different formal classes with respect to $\widehat{\mathcal L}(\mathbb R)$: the time-one maps of two vector fields:
$$
f(z)=\text{Exp}\Big(\frac{z^2}{1+z-z\boldsymbol\ell}\frac{d}{dz}\Big).\mathrm{id},\ \text{and } g(z)=\text{Exp}\Big(\frac{z^2}{1+z}\frac{d}{dz}\Big).\mathrm{id}.
$$
Both are obviously parabolic Dulac germs, defined on a whole neighborhood of zero of the Riemann surface of the logarithm $\mathcal R$, and polynomial in $\boldsymbol\ell$. They both admit \emph{global} Fatou coordinates $\Psi_f,\ \Psi_g:\mathcal R\to\widetilde{\mathcal R}$, where $\widetilde{\mathcal R}$ denotes a neighborhood of the infinity on the Riemann surface of the logarithm. The Fatou coordinates are, up to a complex additive constant, given by 
\begin{align*}
&\Psi_f(z)=\int \frac{dz}{z^2}+\int\frac{dz}{z}-\int\frac{\boldsymbol\ell}{z}\,dz=-\frac{1}{z}+\log z+\log(-\log z),\\
&\Psi_g(z)=-\frac{1}{z}+\log z,\ z\in\mathcal R.
\end{align*}
Function $\Psi_g$ is injective on $\mathcal R$. By the Abel equation for the Fatou coordinate, one analytic conjugacy conjugating $g$ to $f$, defined and analytic on $\mathcal R$ and tangent to identity, is given by:
$$
\varphi:=\Psi_g^{-1}\circ \Psi_f.
$$
On the other hand, $f$ and $g$ have a different formal invariants $\rho$ in the class $\widehat{\mathcal L}(\mathbb R)$. Indeed, the asymptotic expansion $\widehat\varphi$ of $\varphi$ belongs to $\widehat{\mathcal L}_2(\mathbb R)$, which can easily be seen e.g. by Taylor expansion of the left-hand side in $\Psi_g(x+h(x))=\Psi_f,\ \varphi=\mathrm{id}+h$, or by writing $\tilde\Psi_f=-e^{1/\boldsymbol\ell}-\frac{1}{\boldsymbol\ell}-\log\boldsymbol\ell$ and $\tilde\Psi_g=-e^{1/\boldsymbol\ell}-\frac{1}{\boldsymbol\ell}$ in the $\boldsymbol\ell$-variable. Here, it is obvious that by precomposition of $\tilde\Psi_g$ with power-exponential transseries, we cannot generate a logarithmic term in $\tilde\Psi_f$.
\end{example}
\medskip

Let $\widehat{h}\in\widehat{\mathcal L}(\mathbb R)$ denote the formal conjugacy between generalized Dulac expansions $\widehat f$ and $\widehat g$ of two analytically conjugated (in the sense of Definition~\ref{def:jedan}) parabolic generalized Dulac germs $f$ and $g$, $$\widehat g=\widehat h^{-1}\circ \widehat f\circ \widehat h,\ \widehat h\in\widehat{\mathcal L}(\mathbb R).$$ The formal conjugacy $\widehat h\in \widehat{\mathcal L}(\mathbb R)$ is derived in \cite{mrrz2}. We will show in Section~\ref{subsec:types} that $h$ from Definition~\ref{def:jedan}, up to some controlled change, then admits the formal conjugacy $\widehat h \in\widehat{\mathcal L}(\mathbb R)$ as its \emph{generalized block iterated integral} sectional asymptotic expansion, see Definition~\ref{def:giiaa} and Proposition~\ref{prop:gic}. Indeed, the transserial asymptotic expansions of germs in $\widehat{\mathfrak L}(\mathbb R)$ are not unique, see \cite{MRRZ2Fatou}. To make them unique, we must choose a canonical \emph{summation rule} at limit ordinal steps. That is, we must fix an appropriate section function (see \cite{MRRZ2Fatou} and Definition~\ref{def:assy} in the Appendix for a complex version). 
\bigskip

We repeat here the $\widehat{\mathcal L}(\mathbb R)$-normal form result for parabolic generalized Dulac germs derived in \cite{mrrz2}. We give here an alternative proof working \emph{block-by-block}, not \emph{term-by-term} as in \cite{mrrz2}. As opposed to termwise eliminations done previously in \cite{mrrz2}, the importance of blockwise eliminations is that they determine the integral form of each block. This will be important to define the appropriate integral asymptotic expansions for conjugacies in Section~\ref{subsec:types}. The proof here is inductive, by \emph{block-by-block} eliminations. By a \emph{block}, we mean all monomials in a transseries having the same power of $z$.  In the proof we need the following Lemma~\ref{lem:elim}.


\begin{prop}[Formal normal forms of parabolic generalized Dulac transseries, \cite{mrrz2}]\label{elim}
Let $\widehat f(z)=z-z^\alpha\boldsymbol\ell^{m}+\text{h.o.t.}$\footnote{higher order terms}, $\alpha>1,\ m\in\mathbb Z,$ be a normalized parabolic generalized Dulac transseries. By a formal tangent to the identity formal series $\widehat\varphi\in\widehat{\mathcal L}^{\mathrm{id}}(\mathbb R)$ it can be reduced to any of the normal forms:
\begin{align}
&\widehat f_F(z)=z- z^\alpha\boldsymbol\ell^{m}+\rho z^{2\alpha-1}\boldsymbol\ell^{2m+1},\nonumber\\
&\widehat f_1(z)=\mathrm{Exp}(X_1).\mathrm{id},\ X_1(z)=\frac{-z^\alpha\boldsymbol\ell^{m}}{1+\frac{-\alpha}{2}z^{\alpha-1}\boldsymbol\ell^{m}+\big(\frac{m}{2}+\rho\big) z^{\alpha-1}\boldsymbol\ell^{m+1}}\frac{d}{dz}.\label{eq:norma1}
\end{align}
The $\widehat{\mathcal L}(\mathbb R)$-formal invariants are $(\alpha,m,\rho)$, $\alpha>1$, $m\in\mathbb Z$, $\rho\in\mathbb R$.
\end{prop}
Evidently, since $\widehat{\mathcal L}^{inv}(\mathbb R)$ is a group under composition, any two parabolic generalized Dulac series $\widehat f$ and $\widehat g$ with the same formal invariants $(\alpha,m,\rho)$, are formally conjugated in $\widehat{\mathcal L}^{inv}(\mathbb R)$.

\begin{lem}[Blockwise eliminations]\label{lem:elim} Let $\widehat f\in \widehat{\mathcal L}(\mathbb R)$,\ $\widehat f(z)=z-z^\alpha\boldsymbol\ell^{m}+z^{\beta_i}\widehat T_i(\boldsymbol\ell)+\text{h.o.b.}$\footnote{higher order blocks}, $\alpha>1$, $m\in\mathbb Z$, $\widehat T_i(\boldsymbol\ell)\in\mathbb R((\boldsymbol\ell))$. We can eliminate the block $z^{\beta_i}\widehat T_i(\boldsymbol\ell)$, except possibly the monomial $z^{2\alpha-1}\boldsymbol\ell^{2m+1}$ in the residual block $\beta_i=2\alpha+1$, by \emph{elementary} change of variables $\widehat\varphi_i(z)=z+z^{\gamma_i}\widehat{R}_i(\boldsymbol\ell)$, where $\gamma_i=\beta_i-\alpha+1$ and $\widehat R_i(\boldsymbol\ell)\in \mathbb R((\boldsymbol\ell))$ is given by:
\begin{equation}\label{eq:ma}
\begin{cases}
\widehat R_i(\boldsymbol\ell)=-\frac{1}{z^{\gamma_i-\alpha}\boldsymbol\ell^{-m}}\int z^{\gamma_i-\alpha}\boldsymbol\ell^{-2m-2}\widehat T_i(\boldsymbol\ell)d\boldsymbol\ell
,&\ \ \beta_i>\alpha,\\
\widehat R_i(\boldsymbol\ell)=\frac{1}{z}\Big[\big(\int\frac{z^{1-\alpha}}{\boldsymbol\ell^{m+2}}d\boldsymbol\ell\big)^{-1}\circ \big(\int\frac{z^{1-\alpha}}{\widehat T_0(\boldsymbol\ell)\boldsymbol\ell^2}d\boldsymbol\ell\big)\Big]-1,&\ \ \beta_i=\alpha,
\end{cases}
\end{equation}
with $0$ as constant of formal integration. Here, $\widehat T_0(\boldsymbol\ell)\in\mathbb R((\boldsymbol\ell))$, $\widehat T_0(\boldsymbol\ell)=-\boldsymbol\ell^m+h.o.t.$, is the whole first block of $\widehat f(z)$, with $z^\alpha$. In the case $\beta_i=2\alpha+1$, $\widehat T_i(\boldsymbol\ell)$ is the remainder of the residual block, without the term $\rho z^{2\alpha-1}\boldsymbol\ell^{2m+1},\ \rho\in\mathbb Z$.
\end{lem}

The above integrals are formally integrated by parts, putting $dv=e^{\frac{\alpha-\gamma_i}{\boldsymbol\ell}}\boldsymbol\ell^{-2}$, and $u$ equal to the remainder of the subintegral function. We always choose $0$ as the constant of formal integration in \eqref{eq:ma}, in order that $\widehat \varphi_i(z)$ be an \emph{elementary} change of variables with only one block. 
Note that by taking $\beta_1=\alpha$ we eliminate the first block, except for the first term. The proof of Lemma~\ref{lem:elim} is in the Appendix.
\smallskip

We also state the following generalization of Lemma~\ref{lem:elim}, which is proved in the same way as Lemma~\ref{lem:elim}, so we omit the proof. We will need Lemma~\ref{lem:elim1} in Subsection~\ref{subsec:iis} to prove the block iterated integral summability of formal conjugacies of two parabolic generalized Dulac germs.

\begin{lem}[Formal conjugation of two transseries]\label{lem:elim1} Let $\widehat f,\ \widehat g\in \widehat{\mathcal L}(\mathbb R)$ with the same formal invariants $(\alpha,m,\rho)$. Let, for $n\in\mathbb N_0$,
\begin{align*}
&\widehat f(z)=z-z^\alpha\boldsymbol\ell^{m}+\sum_{i=1}^{n} z^{\alpha_i} \widehat T_i(\boldsymbol\ell)+z^{\alpha_{n+1}}\widehat F_{n+1}(\boldsymbol\ell)+\text{h.o.b.},\\
&\widehat g(z)=z-z^\alpha\boldsymbol\ell^{m}+\sum_{i=1}^{n} z^{\alpha_i} \widehat T_i(\boldsymbol\ell)+z^{\alpha_{n+1}}\widehat G_{n+1}(\boldsymbol\ell)+\text{h.o.b.},
\end{align*}
where $\alpha>1$, $m\in\mathbb Z$, $\alpha_i\geq \alpha$ strictly increasing as $i\to\infty$, $\widehat T_i\in\mathbb R((\boldsymbol\ell)),\ i=1,\ldots,n,\ \widehat F_{n+1},\,\widehat G_{n+1}\in\mathbb R((\boldsymbol\ell))$, and $\widehat F_{n+1}\neq 0$ $(\widehat G_{n+1}$ can be $0)$. Then
$$
\widehat \varphi_{n+1}\circ \widehat g\,\circ \widehat  \varphi_{n+1}^{-1}=z-z^\alpha\boldsymbol\ell^{m}+\sum_{i=1}^{n} z^{\alpha_i} \widehat T_i(\boldsymbol\ell)+z^{\alpha_{n+1}}\widehat F_{n+1}(\boldsymbol\ell)+\text{h.o.b.},
$$
where $\widehat{\varphi}_{n+1}(z)=z+z^{\gamma_{n+1}}\widehat R_{n+1}(\boldsymbol\ell)$, where $\gamma_{n+1}=\alpha_{n+1}-\alpha+1$ and $\widehat R_{n+1}(\boldsymbol\ell)\in \mathbb R((\boldsymbol\ell))$ is given by:
\begin{equation}\label{eq:ma1}
\begin{cases}
\widehat R_{n+1}(\boldsymbol\ell)=-\frac{1}{z^{\gamma_{n+1}-\alpha}\boldsymbol\ell^{-m}}\int z^{\gamma_{n+1}-\alpha}\boldsymbol\ell^{-2m-2}\big(\widehat G_{n+1}(\boldsymbol\ell)-\widehat F_{n+1}(\boldsymbol\ell)\big)d\boldsymbol\ell
,&\alpha_{n+1}>\alpha,\\[0.2cm]
\widehat R_{n+1}(\boldsymbol\ell)=\frac{1}{z}\Big[\big(\int\frac{z^{1-\alpha}}{\widehat F_1(\boldsymbol\ell)\boldsymbol\ell^2}d\boldsymbol\ell\big)^{-1}\circ \big(\int\frac{z^{1-\alpha}}{\widehat G_1(\boldsymbol\ell)\boldsymbol\ell^2}d\boldsymbol\ell\big)\Big]-1,&\alpha_{n+1}=\alpha,
\end{cases}
\end{equation}
with $0$ as constant of formal integration. Here, $\widehat F_1(\boldsymbol\ell)=-\boldsymbol\ell^m+h.o.t.$, $\widehat G_1(\boldsymbol\ell)=-\boldsymbol\ell^m+h.o.t.,\ \widehat F_1,\,\widehat G_1\in\mathbb R((\boldsymbol\ell))$, are the whole first blocks of $\widehat f$, resp.\ $\widehat g$ $($that is, blocks of $z^\alpha)$.
\end{lem}

\medskip
\noindent \emph{Proof of Proposition~\ref{elim}}. Let 
$$
\widehat f(z)=z+z^\alpha \widehat T_0(\boldsymbol\ell)+z^{\alpha_1} \widehat T_1(\boldsymbol\ell)+\ldots
$$
Here, $\widehat T_i(\boldsymbol\ell)\in \mathbb R((\boldsymbol\ell))$, $i\in\mathbb N_0$, with the leading term of $\widehat T_0$ equal to $\mathrm{Lt}(\widehat T_0(\boldsymbol\ell))=-\boldsymbol\ell^m$, $m\in\mathbb Z$.

We proceed similarly as in \cite[Theorem A]{mrrz2}, but eliminate block by block instead of term by term. We first eliminate the initial block $z^\alpha \widehat T_0(\boldsymbol\ell)$, which can be eliminated except for the first term $-z^\alpha\boldsymbol\ell^m$. Indeed, in each step of the elimination we search for a formal change of variables:
$$
\widehat \varphi_i(z)=z+z^{\gamma_i} \widehat R_i(\boldsymbol\ell),\ \text{ord}(z^{\gamma_i} \widehat R_i)\succ (1,0),\ i\in\mathbb N_0,
$$
where $\widehat R_i\in\mathbb R((\boldsymbol\ell))$. Here, the order of a transseries $\text{ord}(.)$ means the lexicographic order of its first term.

By Lemma~\ref{lem:elim},
in order to remove the first block $z^\alpha \widehat T_0(\boldsymbol\ell)$ (that is, its part that can be eliminated), we need to take $\gamma_0=1$. But then $\mathrm{ord}_{\boldsymbol\ell}(\widehat R_0)\geq 1$. Therefore, the first monomial $-z^\alpha\boldsymbol \ell^{m}$ cannot be eliminated. The remainder of the first block is eliminated by the change $\widehat\varphi_0(z)=z+z\widehat R_0(\boldsymbol\ell)$, where $\widehat R_0(\boldsymbol\ell)\in\boldsymbol\ell\,\mathbb R[[\boldsymbol\ell]]$ is given in \eqref{eq:ma} (it is deduced by integration by parts and formal composition). 

Furthermore, it is obvious by integration by parts of the explicit formula \eqref{eq:ma} in Lemma~\ref{lem:elim} that, if $\gamma_i\neq 1$ and $\gamma_i\neq \alpha$, then $\widehat T_i(\boldsymbol\ell)\in\mathbb R((\boldsymbol\ell))$ implies $\widehat R_i(\boldsymbol\ell)\in\mathbb R((\boldsymbol\ell))$
 and no iterated logarithms are generated. Therefore, we do not need iterated logarithms to remove blocks before the residual block, and consequently all $\widehat T_i(\boldsymbol\ell)$ up to the residual block (inclusive) belong to $\mathbb R((\boldsymbol\ell))$. That is, $\widehat f_i,\ \widehat\varphi_i\in\widehat{\mathcal L}(\mathbb R)$ for all eliminatons before the residual. Here, $\widehat f_i$ denotes the initial transseries after first $i$ changes of variables, that is, $\widehat f_i=(\circ_{j=0,\ldots,i} \widehat\varphi_j)^{-1}\circ\widehat f \circ (\circ_{j=0,\ldots,i} \widehat\varphi_j)$, $i\in\mathbb N_0$.

If $\beta_i=2\alpha-1$ (elimination of the residual block $z^{2\alpha-1} \widehat T_i(\boldsymbol\ell)$, $\widehat T_i(\boldsymbol\ell)\in\mathbb R((\boldsymbol\ell))$), by Lemma~\ref{lem:elim} we choose $\gamma_i=\alpha$ and
$$
\widehat R_i(\boldsymbol\ell)=-\boldsymbol\ell^{m}\int \boldsymbol\ell^{-2m-2}\widehat T_i(\boldsymbol\ell)d\boldsymbol\ell.
$$
It is easy to see that, by an appropriate $\widehat R_i(\boldsymbol\ell)\in\mathbb R((\boldsymbol\ell))$, we can eliminate all terms in $\widehat T_i(\boldsymbol\ell)$ except for the residual term $z^{2\alpha-1}\boldsymbol\ell^{2m+1}$. Set of parabolic transseries from $\widehat{\mathcal L}(\mathbb R)$ is a group for composition, so after eliminating the residual block we have $\widehat f_i:=\widehat \varphi_i^{-1}\circ \widehat f_{i-1}\circ \widehat\varphi_i\in \widehat{\mathcal L}(\mathbb R)$. After eliminating the residual block, further blocks in $\widehat f_i$ do not contain iterated logarithms. By Lemma~\ref{lem:elim}, for $\beta_i> 2\alpha-1$ ($\gamma_i> \alpha$), we remove all further blocks $z^{\beta_i}\widehat T_i(\boldsymbol\ell)$, $\widehat T_i(\boldsymbol\ell)\in\mathbb R((\boldsymbol\ell))$, using changes of variables with $\widehat R_i(\boldsymbol\ell)\in\mathbb R((\boldsymbol\ell))$, as was explained before. 

Inductively, to remove all possible terms, we solve a sequence of Lie bracket equations, see \cite[Theorem A]{mrrz2} for details. But in contrast with \cite{mrrz2} where we eliminate term by term, we solve here only countably many equations and not a transfinite number of equations. Indeed, due to the fact that $\alpha_i$ in $\widehat f$ are finitely generated, the set of all $\gamma_i$ needed in blockwise eliminations is also finitely generated. For details, see \cite{mrrz2}. 
\hfill $\Box$

\begin{obs}\label{rem:elim}
If we allow formal conjugation in Proposition~\ref{elim} in the larger class $\widehat{\mathcal L}_2(\mathbb R)$, we are able to eliminate also the residual term $\rho z^{2\alpha-1}\boldsymbol\ell^{2m+1}$. By a formal diffeomorphism $\widehat\varphi\in\widehat{\mathcal L}_2(\mathbb R)$, a parabolic generalized Dulac series $\widehat f$ from Proposition~\ref{elim} can be reduced to any of the normal forms:
\begin{align*}
&f_F(z)=z-z^\alpha\boldsymbol\ell^{m},\nonumber\\ &\widehat f_2(z)=\mathrm{Exp}( -z^{\alpha}\boldsymbol\ell^m\frac{d}{dz}).\mathrm{id}.
\end{align*}
The $\widehat{\mathcal L}_2(\mathbb R)$-formal invariants are $(\alpha,m)$. The proof is similar to the proof of Proposition~\ref{elim}. Now Lemma~\ref{lem:elim} can be simply rewritten with $\widehat R_i(\boldsymbol\ell),\ \widehat T_i(\boldsymbol\ell)\in\widehat{\mathcal L}_{\boldsymbol\ell}^\infty(\mathbb R)$\footnote{$\widehat R_i(\boldsymbol\ell)$ and $\widehat T_i(\boldsymbol\ell)$ are transseries in integer powers of $\boldsymbol\ell$ and $\boldsymbol\ell_2$, see the notation \eqref{eq:nozn}.}, instead of $\widehat R_i(\boldsymbol\ell),\ \widehat T_i(\boldsymbol\ell)\in\mathbb R((\boldsymbol\ell))$. In order to remove the residual term $z^{2\alpha-1}\boldsymbol\ell^{2m+1}$ in appropriate $\widehat{T}_i(\boldsymbol\ell)$, we need a double logarithm monomial $\boldsymbol\ell^{m}\boldsymbol\ell_2^{-1}$ in the corresponding $\widehat R_i(\boldsymbol\ell)$ in the change of variables $\widehat\varphi_i$. Therefore, the residual elementary change $\widehat \varphi_i$ belongs to $\widehat{\mathcal L}_2(\mathbb R)$. The blocks of $\widehat f_i$ after removing the residual block in general also contain double logarithms, that is, $\widehat T_i(\boldsymbol\ell)\in \widehat{\mathcal L}_{\boldsymbol\ell}^\infty(\mathbb R)$. By a generalization of Lemma~\ref{lem:elim} to $\widehat{\mathcal L}_2(\mathbb R)$, we remove them by $\widehat R_i(\boldsymbol\ell)\in \widehat{\mathcal L}_{\boldsymbol\ell}^\infty(\mathbb R)$, in the subsequent changes of variables. The parabolic transseries in $\widehat{\mathcal L}_2(\mathbb R)$ also form a group under composition.  
\end{obs}

\section{Integral summability}\label{sec:summa}

We describe here the \emph{integrally summable} nature of blocks in formal conjugacies $\widehat\varphi\in\widehat{\mathcal L}(\mathbb R)$ that reduce a parabolic generalized Dulac germ to its formal normal form, and in its formal Fatou coordinate $\widehat\Psi\in\widehat{\mathcal L}_2(\mathbb R)$. After introducing some necessary definitions, the statement is given in Propositions~\ref{prop:fffatou} and \ref{prop:ok} below.
\smallskip

\subsection{The notion of integral summability}\label{subsec:iis}

In this subsection we apply Definition~\ref{def:isk} of \emph{integral summability} of length $1$ to blocks of formal Fatou coordinates of a generalized Dulac germ, see Proposition~\ref{prop:fffatou}. Definition~\ref{def:isk} is a generalization of the notion of integral summability introduced before in \cite[Definition 3.9]{MRRZ2Fatou} for the formal Fatou coordinate of parabolic Dulac germs. We will again use the same name for simplicity. For the formal conjugacies, we will moreover need integral summability of higher lengths from Definition~\ref{def:isk} and the notion of \emph{block iterated integral summability} from Definition~\ref{def:dai} in the next subsection. This notion is motivated by Lemmas~\ref{lem:elim} and \ref{lem:elim1}, showing that we apply one integration at every step of the construction of $\widehat\varphi$. By Proposition~\ref{elim}, up to some initial transformation due to the elimination of the first block, $\widehat\varphi$ is constructed as composition of countably many elementary changes $\widehat{\varphi}_i=z+z^{\beta_i}\widehat R_i(\boldsymbol\ell),\ i\in\mathbb N$, with strictly increasing $\beta_i>1$. By formulas \eqref{eq:ma} and \eqref{eq:ma1}, $\widehat R_1(\boldsymbol\ell)$ is \emph{integrally summable of length $1$}, while $\widehat R_2(\boldsymbol\ell),\ \widehat R_3(\boldsymbol\ell)$, etc., become \emph{integrally summable of higher lengths} with respect to the previous $\widehat R_i(\boldsymbol\ell)$, as defined in the following inductive definition. 

The notion of integral summability resembles the notion of block iterated integrals introduced by Chen in \cite{chen} and studied extensively in the context of first return maps, \cite{gavrilov}, \cite{jessie}, among others.

\smallskip
For the subsequent use, let us introduce the notation $\widehat{\mathcal L}_{\boldsymbol\ell}^{\infty}(\mathbb R)$ for the set of transseries in $\boldsymbol\ell$ and $\boldsymbol\ell_2$ with \emph{integer exponents} and real coefficients:
\begin{equation}\label{eq:nozn}
\widehat F(\boldsymbol\ell)=\sum_{(m,n)\in A\subseteq \mathbb Z\times \mathbb Z} a_{m,n}\boldsymbol\ell^m\boldsymbol\ell_2^n,\ a_{m,n}\in\mathbb R,
\end{equation}
where $A\subseteq \mathbb Z\times \mathbb Z$ is well-ordered.
\smallskip

\begin{defi}[Integral summability of series of length $k$, $k\in\mathbb N_0$]\label{def:isk} Let $0<\theta\leq 2\pi$ and let $S_\theta=\boldsymbol\ell(V_\theta)$ be an $\boldsymbol\ell$-cusp, where $V_\theta$ is a sector or a petal of opening $\theta$. We say that a series $\widehat F(\boldsymbol\ell)\in\widehat{\mathcal L}_{\boldsymbol\ell}^\infty(\mathbb R)$, with at most one double-logarithmic term, is:
\begin{itemize}
\item[(1)]\emph{integrally summable of length $0$} on $S_\theta$ if there exists $m>\frac{\pi}{\theta}$ such that $\widehat F(\boldsymbol\ell)\in \widehat{LG}_m(S_{\theta})$. We call $\boldsymbol\ell\mapsto F(\boldsymbol\ell)\in LG_m(S_\theta)$ its  \emph{$0$-integral sum}.
\smallskip

\item[(2)] \emph{integrally summable of length $1$} on $S_\theta$, if $\widehat F(\boldsymbol\ell)$ is not integrally summable of length $0$ on $S_\theta$, and if there exist exponents $\alpha_1\in\mathbb R$ and $p_1\in\mathbb Z$ and $\widehat R(\boldsymbol\ell)\in\widehat{\mathcal L}_{\boldsymbol\ell}^\infty(\mathbb R)$ integrally summable of length $0$ on $S_\theta$ $($with $0$-integral sum $R)$, such that
\begin{equation}\label{eq:is1}
\frac{d}{d\boldsymbol\ell}\Big(e^{-\frac{\alpha_1}{\boldsymbol\ell}}\boldsymbol\ell^{p_1}\widehat F(\boldsymbol\ell)\Big)=e^{-\frac{\alpha_1}{\boldsymbol\ell}} \boldsymbol\ell^{2p_1-2}\widehat R(\boldsymbol\ell).
\end{equation}
The germ 
$$F(\boldsymbol\ell):=\frac{\int_*^{\boldsymbol\ell} e^{-\frac{\alpha_1}{\eta}} \eta^{2p_1-2} R(\eta) d\eta}{e^{-\frac{\alpha_1}{\boldsymbol\ell}}\boldsymbol\ell^{p_1}},$$
analytic on $S_\theta$, is called  a \emph{$1$-integral sum of $\widehat F$ on $S_\theta$}. Here, $*$ is $0$ if $(\alpha_1,\mathrm{ord}(\widehat R)+2p_1-1)\succ (0,0)$, that is, if the subintegral function is bounded at $0$, or $\boldsymbol\ell_0\in S_\theta$, if not. It is unique up to an additive term $Ce^{\frac{\alpha_1}{\boldsymbol\ell}}\boldsymbol\ell^{-p_1}=Cz^{-\alpha_1}\boldsymbol\ell^{-p_1}$, $C\in\mathbb R$. The pair of exponents $(\alpha_1,p_1)\in(\mathbb R,\mathbb Z)$ is called \emph{the exponent of integration} of $\widehat F(\boldsymbol\ell)$.

\smallskip

\item[(3)] in general, a series $\widehat F(\boldsymbol\ell)\in\widehat{\mathcal L}_{\boldsymbol\ell}^\infty(\mathbb R)$ is \emph{integrally summable of length $k$, $k\geq 2,$ on $S_\theta$, with respect to a set $\widehat U:=\big\{\widehat R_i^{j_i}(\boldsymbol\ell):\,i=1,\ldots,k-1,\ j_i=1,\ldots,r_i\}\subseteq \widehat{\mathcal L}_{\boldsymbol\ell}^{\infty}(\mathbb R)$, where $\widehat R_i^{j_i}(\boldsymbol\ell)\in \widehat U$, $j_i=1,\ldots,r_i$, are integrally summable of length $i$ on $S_\theta$, $i=1,\ldots,k-1$}, with respect to previous $\big\{\widehat R_n^{j_n}(\boldsymbol\ell):\,n=1,\ldots,i-1,\ j_n=1,\ldots,r_n\}\subseteq \widehat U$ and where all elements of $\widehat U$ have different exponents of integration, if $\widehat F(\boldsymbol\ell)$ is not of the form 
$$
\widehat F(\boldsymbol\ell)=\widehat S_{\leq l}^{\widehat R_1^{1},\ldots,\widehat R_1^{r_1};\ldots;\widehat R_{l}^{1},\ldots,\widehat R_l^{r_l}}(\boldsymbol\ell)\in\widehat {\mathcal L}_{\boldsymbol\ell}^\infty(\mathbb R), \ l\leq k-1,
$$
where $\widehat S_{\leq l}^{\widehat R_1^{1},\ldots,\widehat R_1^{r_1};\ldots;\widehat R_{l}^{1},\ldots,\widehat R_l^{r_l}}(\boldsymbol\ell)\in\widehat {\mathcal L}_{\boldsymbol\ell}^\infty(\mathbb R)$ is:
\begin{itemize}

\item[$i)$] $l=1$: $\widehat S_{\leq 1}^{\widehat R_1^{1},\ldots,\widehat R_1^{r_1}}(\boldsymbol\ell)$ is a finite algebraic combination $($with operations $+,\cdot,/,\frac{d}{d\boldsymbol\ell}$$)$ in integrally summable series of length $0$ and in $\widehat R_1^{j_1}(\boldsymbol\ell)$, $j_1=1,\ldots,r_1$, which is not integrally summable of length $0$, 

\item[$ii)$] $l\geq 2$: $\widehat S_{\leq l}^{\widehat R_1^{1},\ldots,\widehat R_1^{r_1};\ldots;\widehat R_{l}^{1},\ldots,\widehat R_l^{r_l}}(\boldsymbol\ell)$ is a finite algebraic combination $($with operations $+,\cdot,/,\frac{d}{d\boldsymbol\ell}$$)$ in integrally summable series of length $0$, in $\widehat R_1^{j_1}(\boldsymbol\ell)\subseteq \widehat U$, $j_1=1,\ldots,r_1,$ and in integrally summable series $\widehat R_i^{j_i}(\boldsymbol\ell)\in \widehat U$, $j_i=1,\ldots,r_i$, of length $i$ with respect to previous $\big\{\widehat R_n^{j_n}(\boldsymbol\ell):\,n=1,\ldots,i-1,\ j_n=1,\ldots,r_n\}\subseteq \widehat U$, $1\leq i\leq l$, that does not belong to $\widehat S_{\leq l-1}^{\widehat R_1^{1},\ldots,\widehat R_1^{r_1};\ldots;\widehat R_{l-1}^{1},\ldots,\widehat R_{l-1}^{r_{l-1}}}(\boldsymbol\ell)$,
\end{itemize}
\smallskip 

\noindent and if there exists a pair of exponents $(\alpha_k,p_k)\in(\mathbb R,\mathbb Z)$ such that 
\begin{equation}\label{eq:is2}
\frac{d}{d\boldsymbol\ell}\Big(e^{-\frac{\alpha_k}{\boldsymbol\ell}} \boldsymbol\ell^{p_k} \widehat F(\boldsymbol\ell)\Big)=e^{-\frac{\alpha_k}{\boldsymbol\ell}} \boldsymbol\ell^{2p_k-2}\widehat S_{\leq k-1}^{\widehat R_1^{1},\ldots,\widehat R_1^{r_1};\ldots;\widehat R_{k-1}^{1},\ldots,\widehat R_{k-1}^{r_{k-1}}}(\boldsymbol\ell),
\end{equation}
where $\widehat S_{\leq k-1}^{\widehat R_1^{1},\ldots,\widehat R_1^{r_1};\ldots;\widehat R_{k-1}^{1},\ldots,\widehat R_{k-1}^{r_{k-1}}}(\boldsymbol\ell)\in\widehat{\mathcal L}_{\boldsymbol\ell}^\infty(\mathbb R)$ are of the above described form. \newline Denote by $S_{\leq k-1}^{R_1^{1},\ldots,R_1^{r_1};\ldots;R_{k-1}^{1},\ldots,R_{k-1}^{r_{k-1}}}(\boldsymbol\ell)$  a sum of $\widehat S_{\leq k-1}^{\widehat R_1^{1},\ldots,\widehat R_1^{r_1};\ldots;\widehat R_{k-1}^{1},\ldots,\widehat R_{k-1}^{r_{k-1}}}(\boldsymbol\ell)$ $($taken in the natural way, respecting operations, and with some fixed choice of constants in integrals for $R_i^{j_i}(\boldsymbol\ell)$, $i=1,\ldots,l$, $j_i=1,\ldots,r_i$, from previous steps$)$. Then the germ \begin{equation*}F(\boldsymbol\ell):=\frac{\int_*^{\boldsymbol\ell} e^{-\frac{\alpha_k}{\eta}} \eta^{2p_k-2} S_{\leq k-1}^{R_1^{1},\ldots,R_1^{r_1};\ldots;R_{k-1}^{1},\ldots,R_{k-1}^{r_{k-1}}}(\eta) d\eta}{e^{-\frac{\alpha_k}{\boldsymbol\ell}}\boldsymbol\ell^{p_k}},\end{equation*}  analytic on $S_\theta$, is called its \emph{$k$-integral sum}. Here, the choice of a base point $*$ is as above in $(2)$. This $k$-sum is unique $\big($after fixing one particular sum  $S_{\leq k-1}^{R_1^{1},\ldots,R_1^{r_1};\ldots;R_{k-1}^{1},\ldots,R_{k-1}^{r_{k-1}}}(\boldsymbol\ell)$ of $\widehat S_{\leq k-1}^{\widehat R_1^{1},\ldots,\widehat R_1^{r_1};\ldots;\widehat R_{k-1}^{1},\ldots,\widehat R_{k-1}^{r_{k-1}}}(\boldsymbol\ell)\big)$ up to an additive term $Ce^{\frac{\alpha_k}{\boldsymbol\ell}}\boldsymbol\ell^{-p_k}=Cz^{-\alpha_k}\boldsymbol\ell^{-p_k}$, $C\in\mathbb R$ $($due to the possible change in $\boldsymbol\ell_0\in \boldsymbol\ell(V_j^{\pm})$ if the subinegral function is not bounded at $0)$.
\end{itemize}
\end{defi}

Note that equation \eqref{eq:is1} in Definition~\ref{def:isk} is equivalent to solving a non-homogenous linear ordinary differential equation:
$$
\boldsymbol\ell^2\frac{d}{d\boldsymbol\ell}\widehat F+(\alpha_1+p_1\boldsymbol\ell)\widehat F=\boldsymbol\ell^{p_1}\widehat R.
$$
Similarly for equation \eqref{eq:is2}.
\smallskip

Note that integration paths in the integrals from $*$ to $\boldsymbol\ell$ do not matter. Indeed, $\boldsymbol\ell(V_\theta)$ is a cusp with sufficiently small radius (a germ), simply connected, and the subintegral function is analytic on $\boldsymbol\ell(V_\theta)$. 
\smallskip

By the name a \emph{block iterated integral sum} $$U:=\big\{R_n^{j_n}(\boldsymbol\ell):\,n=1,\ldots,k-1,\ j_n=1,\ldots,r_n\}$$ of a set $$\widehat U=\big\{\widehat R_n^{j_n}(\boldsymbol\ell):\,n=1,\ldots,k-1,\ j_n=1,\ldots,r_n\}$$ of integrally summable series $\widehat R_i^r(\boldsymbol\ell)$ with respect to previous ones $\widehat R_j^p(\boldsymbol\ell)$, $1\leq j<i$, of strictly lower lengths, as described in Definition~\ref{def:isk} above, we assume the choice of only one constant of integration $\boldsymbol\ell_0$ in each step, while the subintegral functions in each step are determined by constants of integration from previous steps.

\begin{obs} If $\widehat F$ in \eqref{eq:is1} is already integrally summable of length $0$ on $S_\theta$, then \eqref{eq:is1} is satisfied with every exponent of integration $(\alpha,p)\in(\mathbb R,\mathbb Z)$. Similarly, let $k\geq 2$. If $\widehat F(\boldsymbol\ell)$ is already a finite algebraic combination in integrally summable series of length $0$, in $\widehat R_1^{j_1}(\boldsymbol\ell)$, $j_1=1,\ldots,r_1$, which are integrally summable of length $1$, and in integrally summable series $\widehat R_l^{j_l}(\boldsymbol\ell)$, $j_l=1,\ldots,r_l,$ of all lengths $2\leq l\leq k-1$ with respect to previous $\{\widehat R_n^{j_n}(\boldsymbol\ell):\,n=1,\ldots,l-1,\ j_n=1,\ldots,r_n\}$, then \eqref{eq:is2} holds with \emph{every} exponent of integration $(\alpha,p)\in(\mathbb R,\mathbb Z)$. Then their sums may be taken in the natural way, respecting the operations (with ambiguity in the choice of additive terms in each step). That is why, in each step of the inductive Definition~\ref{def:isk}, we exclude such \emph{trivial} cases, to have uniqueness of exponents of integration in each step.
\end{obs}

\begin{prop}[Uniqueness of integral sums of length $k$, $k\geq 1$]\label{prop:u} Let $\widehat F(\boldsymbol\ell)\in\widehat{\mathcal L}_{\boldsymbol\ell}^\infty(\mathbb R)$ be integrally summable of length $1$ on $\boldsymbol\ell$-cusp $S_\theta$, or integrally summable on $S_\theta$ of length $k$, $k\geq 2$, with respect to $\{\widehat R_n^{j_n}(\boldsymbol\ell):\,n=1,\ldots,k-1,\ j_n=1,\ldots,r_n\}\subset\widehat{\mathcal L}_{\boldsymbol\ell}^\infty(\mathbb R)$, which are integrally summable of lengths $1,\ldots k-1$ resp. on $S_\theta$, as defined in Definition~\ref{def:isk}. Then the exponent of integration $(\alpha,p)\in(\mathbb R,\mathbb Z)$ of $\widehat F$ is unique.
\end{prop}    

The proof, similar as in \cite{MRRZ2Fatou}, is in the Appendix. Proposition~\ref{prop:u} states that the $1$-integral sums and the $k$-integral sums with respect to given integrally summable series of lower lengths are \emph{unique}, up to the choice of the constants of integration in the iterative Definition~\ref{def:isk}. 

\medskip

In Definition~\ref{def:isk}, we also implicitely assume that, if $\widehat S_{\leq k}^{\widehat R_1^{1},\ldots,\widehat R_1^{r_1};\ldots;\widehat R_{k}^{1},\ldots,\widehat R_k^{r_k}}(\boldsymbol\ell)\in\widehat{\mathcal L}_{\boldsymbol\ell}^\infty(\mathbb R)$ can be represented as a finite algebraic combination with operations $+,\cdot,/,\frac{d}{d\boldsymbol\ell}$ in integrally summable series $\widehat R_i^{j_i}(\boldsymbol\ell)$, $j_i=1,\ldots,r_i$, of lengths $i\in\{0,\ 1,\ldots,\ k\}$ respectively, with respect to previous $\{\widehat R_n^{j_n}(\boldsymbol\ell):\,n=1,\ldots,i-1,\ j_n=1,\ldots,r_n\}$, then it can be given a unique sum (up to addition of constants of integration). We prove this in Proposition~\ref{prop:uis} below. In the proof we use the following auxiliary Lemma~\ref{lem:axi}, whose proof is in the Appendix.

\begin{lem}\label{lem:axi} Any finite algebraic combination $\widehat S_{\leq k}^{\widehat R_1^{1},\ldots,\widehat R_1^{r_1};\ldots;\widehat R_{k}^{1},\ldots,\widehat R_k^{r_k}}(\boldsymbol\ell)\in\widehat{\mathcal L}_{\boldsymbol\ell}^\infty(\mathbb R)$, $k\in\mathbb N$, with respect to operations $+,\cdot,/,\frac{d}{d\boldsymbol\ell}$, with $\widehat R_i^{j_i}(\boldsymbol\ell)$, $j_i=1,\ldots,r_{i},$ integrally summable of length $i$, $i=1,\ldots,k$,  as in Definition~\ref{def:isk}, is equal\footnote{after all differentiations} to a \emph{rational function} in series $\widehat R_n^{j_n}(\boldsymbol\ell)$, $n=1,\ldots,k$, $j_n=1,\ldots,r_n$, and in integrally summable series of length $0$. 
\end{lem}

\begin{prop}\label{prop:uis}\

\noindent Let $\widehat U=\{\widehat R_n^{j_n}(\boldsymbol\ell):\,n=1,\ldots,k,\ j_n=1,\ldots,r_n\}\subseteq \widehat{\mathcal L}_{\boldsymbol\ell}^{\infty}(\mathbb R)$ be a set of consecutively integrally summable series with respect to all previous ones which are of strictly lower length, as described in Definition~\ref{def:isk}, with a set of strictly increasing integral exponents $\{\beta_1,\beta_2,...,\beta_{r_1+\ldots+r_k}\}$ respectively.
\ Here, the subscripts of the elements of the set $\widehat U$ denote the length of integral summability of the corresponding transseries.\\ Let $\{R_n^{j_n}(\boldsymbol\ell):\,n=1,\ldots,k,\ j_n=1,\ldots,r_n\}$ be one fixed choice of their integral sums, as described after Definition~\ref{def:isk}. Let $\widehat S\in\widehat{\mathcal L}_{\boldsymbol\ell}^{\infty}(\mathbb R)$ be a transseries that can be represented as a rational function in elements of $\widehat U$ and in integrally summable series of lengths $0$. Then this representation is unique\footnote{We say that two rational functions in elements of $\widehat U$ and in integrally summable series of length $0$ are equal, if, up to operations $+,\cdot,/$ on integrally summable series of order $0$, they are equal in the sense of algebraic expressions. This means that, after substituting each series (from $\widehat U$ or $0$-integrally summable) participating in these rational expressions with a formal variable, the rational algebraic expressions thus obtained are equal up to standard algebraic simplifications.}. Therefore, the sum of $S$, respecting algebraic operations and with respect to fixed sums $\{R_n^{j_n}(\boldsymbol\ell):\,n=1,\ldots,k,\ j_n=1,\ldots,r_n\}$ of elements of $\widehat U$ is \emph{unique}.
\end{prop}
Note that, by Lemma~\ref{lem:axi}, Proposition~\ref{prop:uis} states that, if $\widehat S(\boldsymbol\ell)\in\widehat{\mathcal L}_{\boldsymbol\ell}^\infty(\mathbb R)$ is representable as a finite algebraic combination of elements of $\widehat U$ and of integrally summable series of length $0$, then its sum, respecting the algebraic operations, with respect to given sums of $\widehat U$, is unique. 

\begin{proof} Suppose that there exists a rational function in elements of $\widehat U$ and in $0$-integrally summable series that equals $0$. 
Then, its numerator is a polynomial in integrally summable series of length $0$ and in series from $\widehat U$. We prove that, up to some operations on $0$-integrally summable series, the algebraic expression of this polynomial, as described in the footnote, is equal to $0$. Therefore, the sum taken in the natural way, respecting operations, is zero, and the statement is proven.

Suppose that the length $0\leq l_0\leq k$ is the highest possible length of integrally summable series in this polynomial. We express now the monomial of the highest order in variables $\widehat R_{l_0}^{j_{l_0}}(\boldsymbol\ell)$, $j_{l_0}=1,\ldots,r_{l_0}$, as a rational function in different monomials (of the order less or equal) in variables $\widehat R_{l_0}^{j_{l_0}}(\boldsymbol\ell)$, $j_{l_0}=1,\ldots,r_{l_0}$, and in monomials made of integrally summable series of length $0$ and of series $\widehat R_n^{j_n}(\boldsymbol\ell),\ n=1,\ldots,l_0-1,\ j_n=1,\ldots,r_{l_0-1},$ of strictly lower lengths. In the denominator, there are only monomials in integrably summable series of length $0$ and in series $\widehat R_n^{j_n}(\boldsymbol\ell),\ n=1,\ldots,l_0-1,\ j_n=1,\ldots,r_{l_0-1},$ of strictly lower lengths. We \emph{decrease the order} by multiplying this highest-order monomial with appropriate $e^{-\frac{\alpha}{\boldsymbol\ell}}\boldsymbol\ell^p$ (as many of them as there are series in this monomial) and differentiating. Now the order of the monomials on the left-hand side is decreased by $1$. We multiply by the denominator of the right-hand side. We repeat the procedure with now the highest order monomial in variables $\widehat R_{l_0}^{j_{l_0}}(\boldsymbol\ell)$, $j_{l_0}=1,\ldots,r_{l_0}$, whose order is now equal or strictly smaller. In each step, we simplify the algebraic expression as described in the footnote, up to operations on $0$-integrally summable series. In the case that in some step we get trivial algebraic combination, the initial rational function was trivial, and the procedure stops (meaning that the initial polynomial in the numerator was algebraically trivial, what we wanted to prove). If not, we proceed. After repeating the procedure finitely many times, if the procedure does not stop before, we end up with a simple monomial $\widehat R_{l_0}^r(\boldsymbol\ell)$, $r\in\{1,\ldots,r_{l_0}\}$, of length $l_0$, expressed as a polynomial in integrally summable series of strictly lower lengths and in one other monomial $\widehat R_{l_0}^{r_1}(\boldsymbol\ell)$, $r_1\neq r$, of length $l_0$, divided by a polynomial in integrally summable series of strictly lower lengths than $l_0$. Since $\widehat R_{l_0}^{r_{1}}(\boldsymbol\ell)$ and $\widehat R_{l_0}^r(\boldsymbol\ell)$ have different exponents of integration by assumption, by one more multiplication by appropriate $e^{-\frac{\alpha}{\boldsymbol\ell}}\boldsymbol\ell^p$ and differentiation, we end up with series $\widehat R_{l_0}^{r_1}(\boldsymbol\ell)$ of length $l_0$ being equal to a rational function in series of strictly lower lengths, which is a contradiction. 
\end{proof}




\noindent Note that Definition~\ref{def:isk}, point (2), of integral summability of length $1$ corresponds to Definition 3.9 in \cite{MRRZ2Fatou}, since $\frac{d}{dz}=\frac{\boldsymbol\ell^2}{z}\frac{d}{d\boldsymbol\ell}$. 

\begin{obs} It can be checked directly using formula \eqref{eq:is1} that the sum of two integrally summable series $\widehat F(\boldsymbol\ell)$ and $\widehat G(\boldsymbol\ell)$ of length $1$  with the same exponents of integration $(\alpha,m)$ can either be integrally summable of length $0$, or integrally summable of length $1$ with the same exponent of integration $(\alpha,m)$. 

The sum of two integrally summable series $\widehat F(\boldsymbol\ell)$ and $\widehat G(\boldsymbol\ell)$ of length $k\geq 2$ with respect to the same integrally summable series of lower lengths, where exponents of integration of $\widehat F$ and $\widehat G$ are not necessarily equal, is equal to some algebraic combination $\widehat S_{\leq l}(\boldsymbol\ell)$, for some $l\leq k,$ with respect to the same integrally summable series of lower lengths. Indeed, one can show by differentiation and using formulas \eqref{eq:is1} and \eqref{eq:is2} that the sum satisfies formula \eqref{eq:is2} with an algebraic combination $\widehat S_{\leq l}(\boldsymbol\ell),\ l\leq k,$ on the right-hand side, and with \emph{every} exponent of integration. Let $(\alpha,m),\ (\beta,n)$ be the exponents of integration of $\widehat F(\boldsymbol\ell)$, $\widehat G(\boldsymbol\ell)$ respectively, in general not equal. For \emph{every} $(\gamma,q)\in\mathbb R$, it holds that:
{\small
\begin{align}\label{eq:prr}
&\frac{d}{d\boldsymbol\ell}\Big(e^{-\frac{\gamma}{\boldsymbol\ell}}\boldsymbol\ell^q \big(\widehat F(\boldsymbol\ell)+\widehat G(\boldsymbol\ell)\big)\Big)=\\
&=\frac{d}{d\boldsymbol\ell}\Big(e^{-\frac{\alpha}{\boldsymbol\ell}}\boldsymbol\ell^m\widehat F(\boldsymbol\ell)\cdot e^{-\frac{\gamma-\alpha}{\boldsymbol\ell}}\boldsymbol\ell^{q-m}+e^{-\frac{\beta}{\boldsymbol\ell}}\boldsymbol\ell^n\widehat G(\boldsymbol\ell)\cdot e^{-\frac{\gamma-\beta}{\boldsymbol\ell}}\boldsymbol\ell^{q-n}\Big)\nonumber\\
&=e^{-\frac{\alpha}{\boldsymbol\ell}}\boldsymbol\ell^{2m-2}\widehat R_1(\boldsymbol\ell)\cdot e^{-\frac{\gamma-\alpha}{\boldsymbol\ell}}\boldsymbol\ell^{q-m}+e^{-\frac{\alpha}{\boldsymbol\ell}}\boldsymbol\ell^m\widehat F(\boldsymbol\ell)\cdot e^{-\frac{\gamma-\alpha}{\boldsymbol\ell}}\boldsymbol\ell^{q-m-2}\big((q-m)\boldsymbol\ell+(\gamma-\alpha)\big)+\nonumber\\
&+e^{-\frac{\beta}{\boldsymbol\ell}}\boldsymbol\ell^{2n-2}\widehat R_2(\boldsymbol\ell)\cdot e^{-\frac{\gamma-\beta}{\boldsymbol\ell}}\boldsymbol\ell^{q-n}+e^{-\frac{\beta}{\boldsymbol\ell}}\boldsymbol\ell^n\widehat F(\boldsymbol\ell)\cdot e^{-\frac{\gamma-\beta}{\boldsymbol\ell}}\boldsymbol\ell^{q-n-2}\big((q-n)\boldsymbol\ell+(\gamma-\beta)\big) =\nonumber\\
&=e^{-\frac{\gamma}{\boldsymbol\ell}}\boldsymbol\ell^{2q-2}\Big(\widehat R_1(\boldsymbol\ell)\cdot\boldsymbol\ell^{m-q}+\boldsymbol\ell^{-q}\cdot \big((q-m)\boldsymbol\ell+(\gamma-\alpha)\big)\cdot \widehat F(\boldsymbol\ell)+\nonumber\\
&\quad\qquad\qquad\qquad\qquad\qquad\qquad +\widehat R_2(\boldsymbol\ell)\cdot\boldsymbol\ell^{n-q}+\boldsymbol\ell^{-q}\cdot \big((q-n)\boldsymbol\ell+(\gamma-\beta)\big)\cdot \widehat G(\boldsymbol\ell)\Big).\nonumber
\end{align}}Here, since $\widehat F(\boldsymbol\ell)$ and $\widehat G(\boldsymbol\ell)$ are integrally summable of length $k$, $\widehat R_1(\boldsymbol\ell)$ and $\widehat R_2(\boldsymbol\ell)$ are both of the type $\widehat S_{\leq (k-1)}(\boldsymbol\ell)$. The term in brackets on the right-hand side is then obviously of the type $\widehat S_{\leq l}(\boldsymbol\ell)$, for some $l\leq k$ (due to possible cancellations). In particular, if $(\alpha,m)=(\beta,n)$, $\widehat F(\boldsymbol\ell)+\widehat G(\boldsymbol\ell)$ is by \eqref{eq:prr} integrally summable of order $k$ with the same exponent of integration $(\gamma,q):=(\alpha,m)=(\beta,n)$.

\noindent Now, by \eqref{eq:prr}, for any exponents of integration $(\gamma,q)\neq(\delta,p)$, it holds that
\begin{align*}
\frac{d}{d\boldsymbol\ell}\Big(e^{-\frac{\gamma}{\boldsymbol\ell}}\boldsymbol\ell^q \big(\widehat F(\boldsymbol\ell)+\widehat G(\boldsymbol\ell)\big)\Big),\ \frac{d}{d\boldsymbol\ell}\Big(e^{-\frac{\delta}{\boldsymbol\ell}}\boldsymbol\ell^p \big(\widehat F(\boldsymbol\ell)+\widehat G(\boldsymbol\ell)\big)\Big)
\end{align*}
are of the type $\widehat S_{\leq l_1}(\boldsymbol\ell)$ resp. $\widehat S_{\leq l_2}(\boldsymbol\ell)$, for some $l_1,\,l_2\leq k$. Then
\begin{align}\label{eq:ou}
e^{-\frac{\gamma}{\boldsymbol\ell}}\boldsymbol\ell^{2q-2}\widehat R_3(\boldsymbol\ell)&=\frac{d}{d\boldsymbol\ell}\Big(e^{-\frac{\gamma}{\boldsymbol\ell}}\boldsymbol\ell^q \big(\widehat F(\boldsymbol\ell)+\widehat G(\boldsymbol\ell)\big)\Big)=\\
&=\frac{d}{d\boldsymbol\ell}\Big(e^{-\frac{\delta}{\boldsymbol\ell}}\boldsymbol\ell^p \big(\widehat F(\boldsymbol\ell)+\widehat G(\boldsymbol\ell)\big)\cdot e^{-\frac{\gamma-\delta}{\boldsymbol\ell}}\boldsymbol\ell^{q-p}\Big)=\nonumber\\
&=e^{-\frac{\delta}{\boldsymbol\ell}}\boldsymbol\ell^{2p-2}\widehat R_4(\boldsymbol\ell)\cdot e^{-\frac{\gamma-\delta}{\boldsymbol\ell}}\boldsymbol\ell^{q-p}+\nonumber\\
&\quad\qquad +e^{-\frac{\delta}{\boldsymbol\ell}}\boldsymbol\ell^p \big(\widehat F(\boldsymbol\ell)+\widehat G(\boldsymbol\ell)\big)\cdot e^{-\frac{\gamma-\delta}{\boldsymbol\ell}}\boldsymbol\ell^{q-p-2}\big((q-p)\boldsymbol\ell+(\gamma-\delta)\big)=\nonumber\\
&=e^{-\frac{\gamma}{\boldsymbol\ell}}\Big(\boldsymbol\ell^{p+q-2}\widehat R_4(\boldsymbol\ell)+\boldsymbol\ell^{q-2}\big(\widehat F(\boldsymbol\ell)+\widehat G(\boldsymbol\ell)\big)\big((q-p)\boldsymbol\ell+(\gamma-\delta)\big)\Big).\nonumber
\end{align}
Here, $\widehat R_3(\boldsymbol\ell)$ is of the type $\widehat S_{\leq l_1}(\boldsymbol\ell)$, and $\widehat R_4(\boldsymbol\ell)$ of the type $\widehat S_{\leq l_2}(\boldsymbol\ell)$. Comparing the sides in \eqref{eq:ou}, we conclude that $\widehat F(\boldsymbol\ell)+\widehat G(\boldsymbol\ell)$ is of the type $\widehat S_{\leq l}(\boldsymbol\ell)$, for some $l\leq\max\{l_1,l_2\}$. Note that $\max\{l_1,l_2\}\leq k$. Therefore, $\widehat F(\boldsymbol\ell)+\widehat G(\boldsymbol\ell)$ is of the type $S_{\leq l}$, for some $l\leq k$.

A similar claim can be proven for products $\widehat F(\boldsymbol\ell)\cdot \widehat G(\boldsymbol\ell)$. It can be proven similarly as in \eqref{eq:prr} and \eqref{eq:ou} above that, if both $\widehat F(\boldsymbol\ell)$ and $\widehat G(\boldsymbol\ell)$ are integrally summable of length $k\geq 2$ with respect to the same integrally summable series of lower lengths, then their product is of the type $\widehat S_{\leq l}(\boldsymbol\ell)$, $l\leq k$, with respect to the same integrally summable series of lower lengths. This holds independently of the exponents of integration of $\widehat F(\boldsymbol\ell)$,\ $\widehat G(\boldsymbol\ell)$. 
\end{obs}

\begin{prop}[Formal Fatou coordinate for a parabolic generalized Dulac germ]\label{prop:fffatou}
Let $f$ be a parabolic generalized Dulac germ on a standard quadratic domain and let $\widehat f(z)=z-z^\alpha\widehat R_1(\boldsymbol\ell)+\text{h.o.b.},\ \alpha>1,\ m\in\mathbb Z,$ be its generalized Dulac expansion. There exists a unique, up to an additive constant, formal Fatou coordinate $\widehat\Psi\in\widehat{\mathfrak L}^\infty(\mathbb R)$. Moreover, $\widehat\Psi\in\widehat{\mathcal L}_2^\infty(\mathbb R)$ and it is of the form:
$$
\widehat\Psi(z)=\sum_{i\in\mathbb N} z^{\beta_i} \widehat T_i(\boldsymbol\ell),
$$
where $\widehat T_i(\boldsymbol\ell)\in\widehat{\mathcal L}_{\boldsymbol\ell}^\infty(\mathbb R)$ are integrally summable of length $1$ with exponent of integration $(\beta_i,0)$, as in Definition~\ref{def:isk}, on $\boldsymbol\ell$-cusps $\boldsymbol\ell(V_{\pm}^j)$ of petals\footnote{For definition of petals, see Theorem~A in Section~\ref{sec:introduction} and its proof in Section~\ref{sec:proofA}.} $V_{\pm}^j$, $j\in\mathbb Z$, of opening $\frac{2\pi}{\alpha-1}$, and $(\beta_i)_i\in\mathbb R$ are finitely generated, strictly increasing to $+\infty$, with $\beta_1=-\alpha+1$ and first finitely many of them negative.
\end{prop}

\begin{proof} The proof is similar to the proof of \cite[Theorem]{MRRZ2Fatou} for formal Fatou coordinate of a parabolic Dulac germ, by Abel difference equation and blockwise construction of the formal Fatou coordinate. The only difference is that, instead of polynomials, in every block of $\widehat f$ we have a \emph{$\log$-Gevrey series} in $\boldsymbol\ell$, but still canonically summable by Section~\ref{sec:classes}. Accordingly, our Definition~\ref{def:isk} of \emph{integrally summable series of length $1$} is a generalization of the notion of \emph{integrally summable series} from Definition 3.8 in \cite{MRRZ2Fatou}, used in the formal Fatou coordinate of parabolic Dulac germs. 

The proof is following the same lines as proof of Theorem~A (2). Let $\widehat g(z)=\widehat f(z)-\mathrm{id}=z^{\alpha}\widehat R_1(\boldsymbol\ell)+z^{\alpha_1}\widehat R_2(\boldsymbol\ell)+\mathrm{h.o.b.}$, $\alpha_1>\alpha>1$, $\widehat R_1(\boldsymbol\ell)$ and $\widehat R_2(\boldsymbol\ell)$ $\log$-Gevrey of order strictly bigger than $\frac{\alpha-1}{2}$ on $\boldsymbol\ell(V_\pm^{j})$. We construct the formal Fatou coordinate block-by-block, using the formal Taylor expansion of the Abel equation $\widehat\Psi(\widehat f)-\widehat\Psi=1$:
\begin{equation}\label{eq:aabel}
\widehat\Psi'(z)\widehat g(z)+\frac{1}{2!}\widehat \Psi''(z)\widehat g^2(z)+\ldots=1.
\end{equation}
Let $\widehat\Psi_1$ be the lowest-order block od $\widehat\Psi$. Since the order (in $z$) of the summands in the Taylor expansion strictly increases, we get
$$
\widehat\Psi_1'(z)\cdot z^\alpha \widehat R_1(\boldsymbol\ell)=1.
$$ 
Since $\widehat R_1(\boldsymbol\ell)$ is $\log$-Gevrey of order strictly bigger than $\frac{\alpha-1}{2}$ on $\boldsymbol\ell(V_\pm^j)$ by definition of generalized Dulac expansions, and since $\frac{1}{\widehat R_1(\boldsymbol\ell)}$ is by Proposition~\ref{prop:clocomp} also $\log$-Gevrey of order strictly bigger than $\frac{\alpha-1}{2}$ on $\boldsymbol\ell(V_\pm^j)$, $z^{\alpha-1}\widehat \Psi_1(z)=z^{\alpha-1}\int \frac{z^{-\alpha}\,dz}{\widehat R_1(\boldsymbol\ell)}$ is, by Definition~\ref{def:isk}, $1$-integrally summable in $\boldsymbol\ell$ with integration exponent $(-\alpha+1,0)$ on $\boldsymbol\ell(V_\pm^j)$, $j\in\mathbb Z$. Now $\beta_1:=-\alpha+1$.

To deduce the second block $\widehat\Psi_2(\boldsymbol\ell)$, we put $\widehat\Psi(z)=\widehat\Psi_1(z)+\widehat \Psi_2(z)+\mathrm{h.o.b.}$ in the equation \eqref{eq:aabel}. After cancelations, estimating the lowest-order block of the right and of the left-hand side as before, we get:
\begin{equation}\label{eq:fati}
\widehat\Psi_2'(z)\cdot z^\alpha\widehat R_1(\boldsymbol\ell)=\begin{cases}z^{\alpha_1-\alpha}\frac{\widehat R_2(\boldsymbol\ell)}{\widehat R_1(\boldsymbol\ell)},\ &\alpha_1<2\alpha-1,\\
z^{\alpha-1}(-\frac{1}{2})(\alpha \widehat R_1+\widehat R_1'\cdot\boldsymbol\ell^2),\ &\alpha_1>2\alpha-1,\\
\text{the sum of both},&\alpha_1=2\alpha-1.
\end{cases}
\end{equation}
In the denominator we can get only powers of $\widehat R_1(\boldsymbol\ell)\neq 0$. By closedness of $\log$-Gevrey expansions to algebraic operations and to differentiation from Propositions~\ref{prop:closum}-\ref{prop:clocomp}, we conclude that
$
z^{-\min\{\alpha-\alpha_1+1,-\alpha+2\}}\widehat\Psi_2(z)
$
is again $1$-integrally summable on $\boldsymbol\ell(V_\pm^j)$, with exponent of integration $(\min\{\alpha-\alpha_1+1,-\alpha+2\},0)$. We put $\beta_2:=\min\{\alpha-\alpha_1+1,-\alpha+2\}$. We continue by induction. Precisely, in the $i$-th step of the induction, $i\in\mathbb N$, we suppose that $z^{-\beta_1}\widehat\Psi_1,\ldots,z^{-\beta_i}\widehat \Psi_i$,  are $1$-integrally summable on $\boldsymbol\ell(V_\pm^j)$ with exponents of integration $(\beta_1,0),\ldots,(\beta_i,0)$ respectively. Therefore, all $\widehat\Psi_1'(z),\ldots,\widehat \Psi_i'(z)$ are products of some power of $z$ and of some $\log$-Gevrey series of order strictly bigger than $\frac{\alpha-1}{2}$ on $\boldsymbol\ell(V_\pm^j)$. Due to the closedness of $\log$-Gevrey series to algebraic operations, all their derivatives are of the same form. Moreover, all blocks of $\widehat g^k(z)$, $k\in\mathbb N$, are products of powers of $z$ with $\log$-Gevrey series of order strictly bigger than $\frac{\alpha-1}{2}$, since $\widehat f$ is generalized Dulac. Now, from \eqref{eq:aabel}, it follows that $\widehat \Psi_{i+1}'(z)$ is again a product of a power of $z$ with a $\log$-Gevrey series of order strictly bigger than $\frac{\alpha-1}{2}$ on $\boldsymbol\ell(V_\pm^j)$, well-defined since only powers of $\widehat R_1(\boldsymbol\ell)\neq 0$ appear in the denominator. Therefore, there exists a unique $\beta_{i+1}$ such that $z^{-\beta_{i+1}}\widehat \Psi_{i+1}(z)$ is integrally summable of order $1$ in $\boldsymbol\ell$ on $\boldsymbol\ell(V_{\pm}^j)$. The exponent of integration is $(\beta_{i+1},0)$, and by the algorithm $\beta_{i+1}>\beta_i$.
\end{proof}

\subsection{The notion of block iterated integral summability}\label{subsec:iis}\

The Definition~\ref{def:dai} below of \emph{block iterated integral summability} is adapted to formal normalizations of parabolic generalized Dulac germs. The \emph{technical} complication in the definition of the final composition with $\widehat h_0^{-1}$ comes from the elimination of the first block in reducing a parabolic generalized Dulac transseries to its normal form.  The form of this first change of variables is more complicated than the form of the changes of variables for eliminations of the higher-order blocks, as can be seen in Lemma~\ref{lem:elim}. By performing this elimination as the first step, we would change substantially the type of blocks in the initial generalized parabolic Dulac expansion (they would no more be $\log$-Gevrey, that is, integrally summable of order $0$). This would imply technical difficulties in describing subsequent changes by the notion of integral summability of length $k$, $k\geq 1$, introduced in Definition~\ref{def:isk}. We avoid this difficulty by first reducing the initial parabolic generalized Dulac transseries to a parabolic generalized Dulac transseries with the \emph{same initial block}, which we call an \emph{auxiliary normal form}. Finally, applying one additional change of variables, $\widehat h_0^{-1}$, we reduce it to a normal form from \eqref{eq:norma1}. This will be clarified in the proof of Proposition~\ref{prop:rednorm} below.

\begin{defi}[Block iterated integrally summable transseries] \label{def:dai}

Let $\widehat{\varphi}\in\widehat{\mathcal L}^{\mathrm{id}}(\mathbb R)$, $\widehat\varphi(z)=z+z\widehat R_0(\boldsymbol\ell)+\mathrm{h.o.b.}$, where $\widehat R_0(\boldsymbol\ell)\in\boldsymbol\ell\mathbb R[[\boldsymbol\ell]]$, be a parabolic transseries. Here, it may also be $\widehat R_0=0$. 

We say that $\widehat\varphi$ is \emph{block iterated integrally summable} with parameters $(\alpha,m,\rho)\in\mathbb R_{>1}\times \mathbb Z\times \mathbb R$  on a petal $V$ of opening $\frac{2\pi}{\alpha-1}$, if the following holds:
\begin{enumerate}
\item there exists a $\widehat T_0(\boldsymbol\ell)\in\mathbb R((\boldsymbol\ell))$, $\widehat T_0(\boldsymbol\ell)=-\boldsymbol\ell^m+\mathrm{h.o.t.}$, $\log$-Gevrey summable of order strictly bigger than $\frac{\alpha-1}{2}$ on $\boldsymbol\ell$-cusp $\boldsymbol\ell(V)$, such that\footnote{with $0$ as constant of formal integration}:
\begin{equation}\label{eq:kkk}
\Big(-\int\frac{z^{1-\alpha}}{\boldsymbol\ell^{m+2}}d\boldsymbol\ell\Big)\circ \Big(z+z\widehat R_0(\boldsymbol\ell)\Big)=\int\frac{z^{1-\alpha}}{\widehat T_0(\boldsymbol\ell)\boldsymbol\ell^2}d\boldsymbol\ell,
\end{equation}

\item For $\widehat h_0\in\widehat{\mathcal L}^{\mathrm{id}}(\mathbb R)$ defined by $($with formal integration constant equal to $0$$)$:
\begin{equation}\label{eq:kkkk}
\widehat h_0(z):=\Big(-\int \frac{z^{1-\alpha}}{\boldsymbol\ell^{m+2}}d\boldsymbol\ell-\frac{\alpha}{2}\boldsymbol\ell^{-1}+\big(\frac m 2+\rho\big)\boldsymbol\ell_2^{-1}\Big)^{-1}\circ\Big(\int \frac{z^{-\alpha+1}}{\widehat T_0(\boldsymbol\ell)\boldsymbol\ell^2}d\boldsymbol\ell-\frac{\alpha}{2}\boldsymbol\ell^{-1}+b \boldsymbol\ell_2^{-1}\Big),
\end{equation}
where $b\in\mathbb R$ is uniquely chosen such that $\widehat h_0(z)$ does not contain iterated logarithms,

\noindent $\widehat\varphi\circ \widehat h_0$ can be decomposed in a sequence\footnote{Here and in the sequel: $\circ_{i\in\mathbb N}\widehat\varphi_i:=\widehat\varphi_1\circ\widehat\varphi_2\circ\ldots$.} $\widehat\varphi\circ \widehat h_0=\circ_{i\in\mathbb N}\widehat\varphi_i$ of \emph{elementary} changes of variables of the form \begin{equation*}\widehat\varphi\circ \widehat h_0=\circ_{i\in\mathbb N}\widehat\varphi_i,\ \widehat\varphi_i(z)=z+z^{\beta_i}\widehat R_i(\boldsymbol\ell),\end{equation*} where $(\beta_i)_{i\in\mathbb N}$, $\beta_i>1$, is a strictly increasing sequence of real numbers, either finite or tending to $+\infty$, and $\widehat R_i(\boldsymbol\ell)\in\mathbb R((\boldsymbol\ell))$, $\widehat R_i\neq 0$, $i\in\mathbb N$, form a sequence \begin{equation}\label{eq:hak}\widehat{\mathcal H}=\{\widehat R_i(\boldsymbol\ell):i\in\mathbb N\}\subseteq \mathbb R((\boldsymbol\ell))\end{equation} with the following properties:
\begin{enumerate}
\item each series $\widehat R_i(\boldsymbol\ell)\in\widehat{\mathcal H}$ is either integrally summable of some order $n_i\in\mathbb N_0$, $n_i\leq i$, with respect to all previous elements of $\widehat{\mathcal H}$ that are integrally summable of strictly lower lengths, or an algebraic combination of all previous series $\widehat  R_1(\boldsymbol\ell),\ldots,\widehat R_{i-1}(\boldsymbol\ell)$;
\item the orders $i\mapsto n_i$ are not necessarily increasing.
\end{enumerate}
\end{enumerate}
\end{defi}
Note that we can also present the block iterated integrally summable transseries $\widehat \varphi$ in Definition~\ref{def:dai} as the following compositon:
\begin{equation}\label{eq:jkl}
\widehat \varphi=\big(\circ_{i\in\mathbb N} \widehat\varphi_i\big)\circ\widehat h_0^{-1},
\end{equation}
where $\widehat h_0$ and $\widehat\varphi_i$, $i\in\mathbb N$, are as described in Definition~\ref{def:dai}.

\medskip

Note that to $\widehat {\mathcal H}$ we can attribute a collection (non-unique, depending on a choice of consecutive constants of integration) of analytic \emph{integral sums} of its terms on $\boldsymbol\ell(V)$, \begin{equation}\label{eq:ih}\mathcal{H}^{h_0}=\{R_i(\boldsymbol\ell):\,i\in\mathbb N\}.\end{equation} Note that sums in $\mathcal H^{h_0}$ also depend on the choice of the constant in the definition of $h_0$. Here, $R_i(\boldsymbol\ell)$ are integral sums of given length of $\widehat R_i(\boldsymbol\ell)$, $i\in\mathbb N$, on $\boldsymbol\ell(V)$, with respect to previous integral sums, as in Definition~\ref{def:isk}. They differ by the choice of the constants of integration in iterative definitions of the integral sums. The series $\widehat R_i(\boldsymbol\ell)$ are their common asymptotic expansions.
\medskip

\begin{prop}[Uniqueness of decomposition]\label{prop:dai}
Let $\widehat{\varphi}\in\widehat{\mathcal L}^{\mathrm{id}}(\mathbb R)$, $\widehat\varphi(z)=z+z\widehat R_0(\boldsymbol\ell)+\mathrm{h.o.b.}$, $\widehat R_0(\boldsymbol\ell)\in\boldsymbol\ell\mathbb R[[\boldsymbol\ell]]$, be block iterated integrally summable with parameters $(\alpha,m,\rho)$ on a petal $V$ of opening $\frac{2\pi}{\alpha-1}$. Its decomposition
$$
\widehat\varphi=\big(\circ_{i\in\mathbb N}\widehat \varphi_i\big)\circ \widehat h_0^{-1},
$$
with  $\widehat h_0(z)$ and elementary changes $\widehat\varphi_i(z)$  as described in 
Definition~\ref{def:dai}, is unique.
\end{prop}
\begin{proof} Since $\alpha$, $m$ and $\widehat R_0(\boldsymbol\ell)$ are given, $\widehat T_0(\boldsymbol\ell)\in\mathbb R((\boldsymbol\ell))$ is uniquely determined from equation \eqref{eq:kkk}, and it is of the form $\widehat T_0(\boldsymbol\ell)=-\boldsymbol\ell^m+\mathrm{h.o.t.}$ Then $\widehat h_0(z)$ is uniquely given by \eqref{eq:kkkk}, where $0$ is taken as the formal integration costant. Note that $\widehat h_0\in\widehat{\mathcal L}^{\mathrm{id}}(\mathbb R)$ is parabolic, so $\widehat\varphi\circ\widehat h_0(z)$ is also parabolic. Let us denote its first block by: $\widehat\varphi\circ\widehat h_0(z)=z+z^{\beta_1}\widehat R_1(\boldsymbol\ell)+\text{h.o.b.}$, where $\widehat R_1(\boldsymbol\ell)\in\mathbb R((\boldsymbol\ell))$. Now, since $\widehat\varphi\circ\widehat h_0(z)$ admits a decomposition $\circ_{i\in\mathbb N}\widehat\varphi_i$, where $\beta_i>1$ and strictly increasing, the first change of variables $\widehat\varphi_1$ should be of the form $\widehat\varphi_1(z)=z+z^{\beta_1}\widehat R_1(\boldsymbol\ell)$, and, moreover, $\widehat R_1(\boldsymbol\ell)\in\mathbb R((\boldsymbol\ell))$ is integrally summable of length $0$ or $1$.  We then compose on the left with the inverse $\widehat{\varphi}_1^{-1}$, and continue by induction, to uniquely deduce $\widehat\varphi_i$, $i\in\mathbb N$.
\end{proof}

\begin{prop}[Block iterated integral summability of formal reduction to the formal normal form]\label{prop:rednorm} Let $\widehat f(z)=z-z^\alpha\boldsymbol\ell^m+\text{h.o.t.}$, $\alpha>1$ and $m\in\mathbb Z$, be a normalized parabolic generalized Dulac transseries, belonging to the formal class $(\alpha,m,\rho)$, $\alpha>1,\ m\in\mathbb Z,\ \rho\in\mathbb R$. Let $V$ be a petal of opening $\frac{2\pi}{\alpha-1}$. Let $\widehat f_1$ be its formal normal form defined in \eqref{eq:norma1}. Then:
\begin{enumerate}
\item there \emph{exists} a formal normalization $\widehat\varphi\in\widehat{\mathcal L}^{\mathrm{id}}(\mathbb R)$ which reduces $\widehat f$ to $\widehat f_1$ and which is block iterated integrally summable with parameters $(\alpha,m,\rho)$ on $V$.
\item every formal normalization $\widehat \varphi\in\widehat{\mathcal L}^{\mathrm{id}}(\mathbb R)$ reducing $\widehat f$ to $\widehat f_1$ is, up to precomposition by $\widehat f_c$, $c\in\mathbb R$, block iterated integrally summable with parameters $(\alpha,m,\rho)$ on $V$. Here, $f_c=\mathrm{Exp}\big(c X_1\frac{d}{dz}\big).\mathrm{id},\ c\in\mathbb R,$ where $X_1$ is given in \eqref{eq:norma1}. 
\end{enumerate}
\end{prop}

\begin{proof} \

$(1)$ We construct $\widehat\varphi$ algorithmically, eliminating block by block from $\widehat f$, and choosing $0$ as formal integration constant for each block. 
Let $\widehat f=z+z^\alpha \widehat T_0(\boldsymbol\ell)+\mathrm{h.o.b.}$, where $\widehat T_0(\boldsymbol\ell)=-\boldsymbol\ell^m+\mathrm{h.o.t.}$ is $\log$-Gevrey of order strictly bigger than $\frac{\alpha-1}{2}$ on $V$.  We first reduce $\widehat f$ to an \emph{auxiliary normal form} $\widehat f_2$, given as the time one map of the following formal field:
\begin{equation}\label{eq:t2}
\widehat f_2(z)=\mathrm{Exp}\Big(\frac{z^\alpha \widehat T_0(\boldsymbol\ell)}{1+\frac{\alpha}{2} z^{\alpha-1}\widehat T_0(\boldsymbol\ell)+b z^{\alpha-1}\widehat T_0(\boldsymbol\ell)\boldsymbol\ell}\frac{d}{dz}\Big).\mathrm{id}.
\end{equation}
It can easily be checked that $\widehat f_2$ is again a parabolic generalized Dulac transseries, since all blocks are $\log$-Gevrey of order strictly bigger than $\frac{\alpha-1}{2}$ on $\boldsymbol\ell$-cusp $\boldsymbol\ell(V)$, by closedness of $\log$-Gevrey classes to summation, multiplication and differentiation in Propositions~\ref{prop:closum}-\ref{prop:clodif}. Moreover, $\widehat f_2$ has the unchanged initial block $z^\alpha \widehat T_0(\boldsymbol\ell)$, and $b\in\mathbb R$ is chosen (unique) such that $\widehat f_2$ has the same formal invariants $(\alpha,m,\rho)$ as $\widehat f_1$. Also, its formal Fatou coordinate is simply given by:
$$
\widehat \Psi_{2}=\int \frac{z^{-\alpha+1}}{T_0(\boldsymbol\ell)\boldsymbol\ell^2}d\boldsymbol\ell-\frac{\alpha}{2}\boldsymbol\ell^{-1}+b \boldsymbol\ell_2^{-1}.
$$
Therefore, the \emph{auxiliary normal form} $\widehat f_2$ can be reduced to a normal form $\widehat f_1$ of $\widehat f$ given by \eqref{eq:norma1} simply by one additional parabolic change of variables that can be explicitely expressed as $\widehat h_0^{-1}\in\widehat{\mathcal L}^{\mathrm{id}}(\mathbb R)$: 
\begin{align}\label{eq:lll}
&\widehat h_0^{-1}:=\widehat \Psi_{2}^{-1}\circ \widehat\Psi_{0}=\\
&=\!\!\Big(\int \frac{z^{-\alpha+1}}{\widehat T_0(\boldsymbol\ell)\boldsymbol\ell^2}d\boldsymbol\ell-\frac{\alpha}{2}\boldsymbol\ell^{-1}+b \boldsymbol\ell_2^{-1}\Big)^{-1}\!\!\!\!\!\circ\! \Big(-\int \frac{z^{1-\alpha}}{\boldsymbol\ell^{m+2}}d\boldsymbol\ell-\frac{\alpha}{2}\boldsymbol\ell^{-1}+\big(\frac m 2+\rho\big)\boldsymbol\ell_2^{-1}\Big).\nonumber
\end{align}
Here, we fix the formal integration constant to be $0$. Note that $\widehat h_0^{-1}$ belongs to $\widehat{\mathcal L}^{\mathrm{id}}(\mathbb R)$, since $b$ is chosen such that $\widehat f_2$ and $\widehat f_1$ belong to the same $\widehat{\mathcal L}$-formal class. We now reduce $\widehat f$ to its auxiliary normal form $\widehat f_2$ by a parabolic reduction $\widehat h\in\widehat{\mathcal L}^{\mathrm{id}}(\mathbb R)$. Since all blocks in $\widehat f$ and $\widehat f_2$ are integrally summable of order $0$ ($\log$-Gevrey), and we do not eliminate the first block, it is obvious by construction of blockwise changes of variables in elimination algorithm  described in Lemma~\ref{lem:elim1} and by Definition~\ref{def:dai} that $\widehat h=\circ_{i\in\mathbb N}\widehat \varphi_i$, where $\widehat\varphi_i(z)=z+z^{\beta_i}\widehat R_i(\boldsymbol\ell)$ are elementary changes, $\beta_i$ are strictly increasing and $\widehat R_i(\boldsymbol\ell)$ in $\widehat\varphi_i$, $i\in\mathbb N$, belong to a sequence  $\widehat{\mathcal H}$ of integrally summable series and their algebraic combinations of the form \eqref{eq:hak}.  This is concluded by precise control of blocks in the reduced $\widehat f$ after each blockwise elimination, together with formula \eqref{eq:ma1} in Lemma~\ref{lem:elim1} for blockwise eliminations. It is important to note that in formula \eqref{eq:ma1} for deducing elementary changes we choose $0$ as the formal integration constant, in order to get \emph{elementary changes} (containing only one block after $z$). Note also by formula \eqref{eq:ma1} that exponents $p$ from Definition~\ref{def:isk} of integral summability of $\widehat R_i(\boldsymbol\ell)$ are equal to $m$ for all $\widehat R_i(\boldsymbol\ell)\in\widehat{\mathcal H}$ that are integrally summable of some length.

Finally, we reduce $\widehat f_2$ to $\widehat f_1$ by $\widehat h_0^{-1}$ given above. Now, the reduction $\widehat \varphi:=\widehat h\circ\widehat h_0^{-1}$ of $\widehat f$ to $\widehat f_1$ is block iterated integrally summable in the sense of Definition~\ref{def:dai}. Indeed, let us denote by $\widehat R_0(\boldsymbol\ell)\in\boldsymbol\ell R[[\boldsymbol\ell]]$ the first block of $\widehat\varphi$, $\widehat\varphi(z)=z+z\widehat R_0(\boldsymbol\ell)+\mathrm{h.o.b.}$ It eliminates the first block of $\widehat f$, and is therefore, by Lemma~\ref{lem:elim}, related to the first block $\widehat T_0(\boldsymbol\ell)$ of $\widehat f$ by formula \eqref{eq:kkk}.

$(2)$ In Proposition~\ref{prop:ok} below, we prove that, if $\widehat \varphi$ is a formal normalization of $\widehat f$, then \emph{any} other formal normalization of $\widehat f$ is given by $\widehat\varphi\circ \widehat f_c$, $c\in\mathbb R$. Therefore, any formal normalization $\widehat\varphi$ of $\widehat f$ admits the decomposition~\eqref{eq:jkl} as its unique decomposition \eqref{eq:jkl}, up to precomposition by $\widehat f_c$. 
\hfill $\Box$

\bigskip

The algorithm of reduction of a generalized Dulac series to $\widehat{\mathcal L}(\mathbb R)$-normal form $\widehat f_1$ by blockwise eliminations described in Lemma~\ref{lem:elim} is not unique. The following proposition shows, however, that all formal changes of variables differ from one another in a controlled way.

\begin{prop}\label{prop:ok}\ Let $\widehat f(z)=z-z^\alpha\boldsymbol\ell^m+\ldots$, $\alpha>1$,\ $m\in\mathbb Z$, be a parabolic normalized generalized Dulac series and let $\widehat \Psi\in \widehat{\mathcal L}_2^\infty(\mathbb R)$ be its formal Fatou coordinate\footnote{A formal Fatou coordinate is a formal solution of \emph{Abel's equation} $\widehat\Psi\circ\widehat f-\widehat \Psi=1$, unique in $\widehat{\mathfrak L}(\mathbb R)$ up to an additive constant. It is related to the embedding in a vector field, as it represents flow time. For more information, as well as construction of formal Fatou coordinate for parabolic Dulac germs, see \cite{MRRZ2Fatou}.} $($with constant term equal to $0)$. Let 
$\widehat{f}_1(z)$ from \eqref{eq:norma1} and let
$$
\widehat \Psi_0(z)=-\int \frac{dz}{z^\alpha\boldsymbol\ell^m} +\frac{\alpha}{2}\log z+\big(\frac{m}{2}+\rho\big)\boldsymbol\ell_2^{-1}
$$
be its formal Fatou coordinate $($with constant term equal to $0)$. 

\begin{enumerate}
\item The formal change of variables $\widehat\varphi\in\widehat{\mathcal L}^{\mathrm{id}}(\mathbb R)$  such that $\widehat f=\widehat \varphi\circ \widehat f_1\circ \widehat{\varphi}^{-1}$ is \emph{unique up to precomposition with $\widehat f_c$, $c\in\mathbb R$}. Here, $\widehat f_c=\mathrm{Exp}\big(c X_1\frac{d}{dz}\big).\mathrm{id},\ c\in\mathbb R,$ and $X_1$ from \eqref{eq:norma1}. 
\item For any formal change of variables $\widehat\varphi$ reducing $\widehat f$ to $\widehat f_1$, there exists a constant $C\in\mathbb R$ such that $\widehat \Psi+C=\widehat \Psi_0\circ \widehat\varphi^{-1}$.
\end{enumerate}
\end{prop}
\noindent The proof is classical and can be found in the Appendix.
\medskip
\subsection{The notion of generalized block iterated integral summability.}\
\smallskip

The following definition is introduced for the purpose of describing formal conjugacies between two parabolic generalized Dulac transseries in the same formal class, and not just formal reductions to normal forms, see Proposition~\ref{prop:gDul}below.
\begin{defi}\label{def:giis}
Let $\widehat \varphi\in\widehat{\mathcal L}^{\mathrm{id}}(\mathbb R)$. Let $(\alpha,m,\rho)\in\mathbb R_{>1}\times\mathbb Z\times\mathbb R$, and let $V$ be a petal of opening $\frac{2\pi}{\alpha-1}$. We say that $\widehat\varphi$ is \emph{generalized block iterated integrally summable} on $V$ with parameters $(\alpha,m,\rho)$ if there exists a decomposition
\begin{equation}\label{eq:tzu}
\widehat\varphi=\widehat\varphi_1\circ \widehat\varphi_2^{-1},
\end{equation}
where $\widehat \varphi_{1},\ \widehat \varphi_{2}$ are, up to precompositions by $\widehat f_c$, $c\in\mathbb R$, block iterated integrally summable on $V$ with parameters $(\alpha,m,\rho)$. Here, $f_c$, $c\in\mathbb R$, is a time-$c$ map of the $(\alpha,m,\rho)$-model field $X_1$ as in \eqref{eq:norma1}.
\end{defi}

Note that we do not ask for uniqueness of this decomposition.

\begin{prop}
Two (normalized) parabolic generalized Dulac germs $f$ and $g$ are $\widehat{\mathcal L}(\mathbb R)$-\emph{formally conjugated}\footnote{there exists a formal change of variables $\widehat\varphi\in\widehat{\mathcal L}^{\mathrm{id}}(\mathbb R)$ such that, for their generalized Dulac expansions, it holds that:
$$
\widehat g=\widehat\varphi^{-1}\circ\widehat f\circ\widehat\varphi. 
$$}
if and only if they belong to the same $\widehat{\mathcal L}(\mathbb R)$-formal class $(\alpha,m,\rho)$.
\end{prop}
\begin{proof} Let $f_1$ be the $(\alpha,m,\rho)$-normal form as in \eqref{eq:norma1}. Let $\widehat \varphi_1$ be a formal normalization of $f$ that reduces it to $f_1$, $\widehat\varphi_1^{-1}\circ \widehat f\circ\widehat\varphi_1=\widehat f_1$. Putting this in $
\widehat g=\widehat\varphi^{-1}\circ\widehat f\circ\widehat\varphi,
$ we get that $\widehat \varphi_2:=\widehat \varphi^{-1}\circ\widehat\varphi_1\in\widehat{\mathcal L}^{\mathrm{id}}(\mathbb R)$ is a normalization of $\widehat g$ to the same normal form $f_1$.  On the other hand, if $\widehat \varphi_1\in\widehat{\mathcal L}^{\mathrm{id}}(\mathbb R)$ is a formal normalization of $f$ and $\widehat\varphi_2\in\widehat{\mathcal L}^{\mathrm{id}}(\mathbb R)$ is a formal normalization of $g$, then $\widehat\varphi:=\widehat\varphi_1\circ\widehat\varphi_2^{-1}\in \widehat{\mathcal L}^{\mathrm{id}}(\mathbb R)$ is a formal conjugacy of $f$ to $g$.
\end{proof} 

\medskip

\begin{prop}\label{prop:gDul} Let $\widehat f$ and $\widehat g$ be two normalized parabolic generalized Dulac transseries belonging to the formal class $(\alpha,m,\rho)$, $\alpha>1,\ m\in\mathbb Z,\ \rho\in\mathbb R$. Let $V$ be a petal of opening $\frac{2\pi}{\alpha-1}$. Then:

\begin{enumerate}
\item there exists a parabolic formal change of variables $\widehat\varphi\in\widehat{\mathcal L}^{\mathrm{id}}(\mathbb R)$ conjugating\footnote{$\widehat\varphi^{-1}\circ \widehat f\circ \widehat\varphi=\widehat g$} $\widehat f$ to $\widehat g$ which is  generalized block iterated integrally summable on $V$ with parameters $(\alpha,m,\rho)$. 
\item \emph{every} conjugacy $\widehat \varphi\in\widehat{\mathcal L}^{\mathrm{id}}(\mathbb R)$ conjugating $\widehat f$ to $\widehat g$ is generalized block iterated integrally summable. 
\end{enumerate}
\end{prop}

\begin{proof} 

$(1)$ Put $\widehat\varphi:=\widehat \varphi_1\circ\widehat\varphi_2^{-1}$, where $\widehat\varphi_1\in \widehat{\mathcal L}^{\mathrm{id}}(\mathbb R)$ is a reduction of $\widehat f$ to its $(\alpha,m,\rho)$-normal form $\widehat f_1$, and $\widehat\varphi_2\in \widehat{\mathcal L}^{\mathrm{id}}(\mathbb R)$ is a reduction of $\widehat g$ to its $(\alpha,m,\rho)$-normal form $\widehat f_1$ which are both block iterated integrally summable with parameters $(\alpha,m,\rho)$ on $V$, by Proposition~\ref{prop:rednorm}. 

$(2)$ Every $\widehat\varphi$ can be decomposed as:
$$
\widehat\varphi=\widehat\varphi_1\circ \widehat \varphi_2^{-1},
$$
where $\widehat\varphi_2$ is a reduction of $\widehat g$ to $\widehat f_1$ and $\widehat\varphi_1$ is a reduction of $\widehat f$ to $\widehat f_1$. Indeed, take $\widehat\varphi_1$ to be any normalization of $\widehat f$. By Proposition~\ref{prop:rednorm}, up to precomposition by $\widehat f_c$, it is block iterated integrally summable with parameters $(\alpha,m,\rho)$ . Then $\widehat\varphi^{-1}\circ\widehat\varphi_1$ is a normalization of $g$. By Proposition~\ref{prop:rednorm}, up to precomposition by $\widehat f_c$, it is also block iterated integrally summable with parameters $(\alpha,m,\rho)$ .
\end{proof}
Note that this proof gives us the exact description of non-uniqueness of the conjugacy $\widehat\varphi$ of $\widehat g$ and $\widehat f$, as: $$\widehat\varphi_1\circ \widehat f_c\circ\widehat\varphi_2^{-1}, \ c\in\mathbb R,$$ where $\widehat \varphi_1$, $\widehat\varphi_2$ are \emph{any} normalizations of $\widehat f$, $\widehat g$ respectively. Indeed, by Proposition~\ref{prop:ok}, such normalizations are unique, up to precompositions by $\widehat f_c$, $c\in\mathbb R$.

\medskip

\bigskip
Note that the decomposition \eqref{eq:tzu} of a generalized block iterated integrally summable  $\widehat\varphi\in\widehat{\mathcal L}^{\mathrm{id}}(\mathbb R)$ may not be unique. Indeed, suppose that $\widehat\varphi$ is a conjugacy between two parabolic generalized Dulac transseries that is generalized block iterated integrally summable. In general, there may exist different pairs of parabolic generalized Dulac transseries $\widehat f,\,\widehat g$ and $\widehat f_1,\,\widehat g_1$ that are conjugated by $\widehat\varphi$. 
\medskip

We denote the class of generalized block iterated integrally summable parabolic\linebreak transseries with parameters $(\alpha,m,\rho)$ on some petal\footnote{Here, the petal is considered as a germ (its radius is not important).} $V$ of opening $\frac{2\pi}{\alpha-1}$ by: 
$$\widehat{\mathcal S}^{\alpha,m,\rho}(V)\subset \widehat{\mathcal L}^{\mathrm{id}}(\mathbb R).$$  For simplicity of notation, in the sequel we will not mention sector $V$, although the notion of (generalized, iterated) integral summability is always related to a sector.

As a trivial case of Proposition~\ref{prop:rednorm}, $\widehat\varphi:=\mathrm{id}$ is trivially block iterated integrally summable with every triple of parameters and on any petal $V$. Therefore, block iterated integrally summable parabolic transseries with parameters $(\alpha,m,\rho)$ form a subclass of generalized block iterated integrally summable parabolic transseries with parameters $(\alpha,m,\rho)$. 

Moreover, parabolic generalized Dulac series $\widehat f(z)=z-z^{\alpha}\boldsymbol\ell^m+\ldots$ belong to $\widehat{\mathcal S}^{\alpha,m,\rho}(V)$, where $V$ is any petal of opening strictly less than $\frac{2\pi}{\alpha-1}$, for every triple of parameters $(\alpha,*,*)$. It can be proven by decomposition in a sequence of elementary changes whose blocks are $0$-integrally summable, using Proposition~\ref{prop:grupa}.

\begin{obs} Note that $\widehat{\mathcal L}^{\mathrm{id}}(\mathbb R)$ is a group under composition. However, we are not able to prove the group property for the class 
$\widehat{\mathcal S}^{\alpha,m,\rho}(V)$ of all parabolic generalized block iterated integrally summable series on $V$ with parameters $(\alpha,m,\rho)$. We are even not able to prove the group property for the subclass of all parabolic transseries that are \emph{realizable} as formal conjugacies of two parabolic generalized Dulac transseries in the formal class $(\alpha,m,\rho)$. 

Indeed, even if $\widehat \varphi\in\widehat{\mathcal S}^{\alpha,m,\rho}(V)$ can be realized as a conjugacy between two parabolic generalized Dulac series $\widehat f$ and $\widehat g$ from the class $(\alpha,m,\rho)$, and $\widehat \eta\in \widehat{\mathcal S}^{\alpha,m,\rho}(V)$ can be realized as a conjugacy between another two parabolic generalized Dulac series $\widehat h$ and $\widehat k$ from the same formal class $(\alpha,m,\rho)$, we see no reason in general to believe that the composition $\widehat \eta\circ\widehat\varphi$ will be a conjugacy between two generalized parabolic Dulac series if the pairs are not \emph{chainable}. If they are, that is, if $\widehat g=\widehat h$, then $\widehat\varphi\circ\widehat\eta$ is again a conjugacy between $\widehat f$ and $\widehat k$, and, by Proposition~\ref{prop:gDul}, it is indeed generalized block iterated integrally summable with parameters $(\alpha,m,\rho)$.

Moreover, we do not claim that the formal inverse of a parabolic generalized block iterated integrally summable series is again a parabolic generalized block iterated integrally summable series. However, if $\widehat\varphi$ is a parabolic generalized block iterated integrally summable series with parameters $(\alpha,m,\rho)$ that can be realized as a formal conjugacy between two parabolic generalized Dulac germs $\widehat f$ and $\widehat g$, then its formal inverse $\widehat\varphi^{-1}$ is a formal conjugacy between $\widehat g$ and $\widehat f$. Moreover, they belong to the formal class $(\alpha,m,\rho)$. Therefore, by Proposition~\ref{prop:gDul}, $\widehat\varphi^{-1}$ is again generalized block iterated integrally summable with parameters $(\alpha,m,\rho)$. 
\end{obs}




\bigskip

\section{Asymptotic expansions of sectorially analytic reductions for parabolic generalized Dulac germs.}\label{subsec:dva}\

Let $f$ be a parabolic generalized Dulac germ and $\widehat f(z)=z-z^\alpha\boldsymbol\ell^m+\text{h.o.t.},\ \alpha>1,\ m\in\mathbb Z,$ its generalized Dulac expansion. Let $\widehat \Psi\in\widehat{\mathcal L}_2^\infty(\mathbb R)$ be its formal Fatou coordinate and $\widehat \varphi\in\widehat{\mathcal L}^{\mathrm{id}}(\mathbb R)$ a formal change of variables that reduces it to its $\widehat{\mathcal L}(\mathbb R)$-formal normal form $\widehat f_1(z)=\text{Exp}(X_1).\mathrm{id}$, where $X_1$ is as in \eqref{eq:norma1}. In this section we suppose that we have already proven the existence of sectorially analytic Fatou coordinates $\Psi_{j}^\pm$ on petals $V_j^{\pm}$, $j\in\mathbb Z$, of standard quadratic domain $\mathcal R_C$. This will be proven by construction in the proof of Theorem A $(1)$ in Section~\ref{sec:proofA}. Note that for this we use only assumptions that $f$ is analytic on $\mathcal R_C$ and that it admits uniform estimate \eqref{eq:uniest} in the first two terms. It is well-known that the relation between the conjugacy and the Fatou coordinate is given through the Fatou coordinate $\Psi_0$ of the formal normal form $\widehat f_1$, which is analytic and defined on the whole $\mathcal R_C$ (up to an additive constant in $\Psi$, that is, up to precomposition by $f_c=\text{Exp}(cX_1).\mathrm{id}$, $c\in\mathbb R$, in $\varphi$):
$$
\widehat\Psi=\Psi_0\circ\widehat\varphi,\ \Psi_{\pm}^j=\Psi_0\circ\varphi_{\pm}^j \text { on $V_{\pm}^j$, $j\in\mathbb Z$}.   
$$ 
Sectorial analyticity of the Fatou coordinate therefore immediately implies sectorial analyticity of changes of variables $\varphi_{\pm}^j$ on $V_j^{\pm}$, $j\in\mathbb Z$.
\smallskip

In Propositions~\ref{prop:Fatouexp} and \ref{prop:gi} in Subsections~\ref{subsec:fatouexp} and \ref{subsec:konj}, we state that, by the choice of \emph{integral sections}, the formal Fatou coordinate, i.e. the formal conjugacy, become \emph{unique} sectional asymptotic expansions of sectorial Fatou coordinates, i.e. conjugacies. For more on section functions and sectional transserial expansions, see \cite{MRRZ2Fatou} and Definition~\ref{def:assy} in the Appendix (a complex version).
\smallskip

\subsection{Sectorially analytic Fatou coordinates}\label{subsec:fatouexp}\ 

\noindent In Proposition~\ref{prop:Fatouexp} below, by \emph{integral asymptotic expansion} we mean sectional asymptotic expansion where section functions at limit ordinal steps (that is, sums of series in $\boldsymbol\ell$ in each block) are chosen as \emph{integral sums of length $1$}, as defined in Definition~\ref{def:isk} (2). 
\begin{prop}[Integral asymptotic expansion of Fatou coordinate of a parabolic generalized Dulac germ]\label{prop:Fatouexp}
Let $f$ be a parabolic generalized Dulac germ and $\widehat f(z)=z-z^\alpha\boldsymbol\ell^{m}+\text{h.o.t.},\ \alpha>1,\ m\in\mathbb Z,$ be its generalized Dulac expansion. Let\footnote{$\Psi_{\pm}^j(z)\sim_{z\to 0}-\frac{1}{\alpha-1}z^{-\alpha+1}\boldsymbol\ell^{-m}$ means that: $\lim_{z\to 0,\,z\in V_{\pm}^j}\frac{\Psi_{\pm}^j(z)}{-\frac{1}{\alpha-1}z^{-\alpha+1}\boldsymbol\ell^{-m}}=1$ in $V_{\pm}^j.$} $\Psi_{\pm}^j(z)\sim_{z\to 0}-\frac{1}{\alpha-1}z^{-\alpha+1}\boldsymbol\ell^{-m}$ be its analytic Fatou coordinates on petals $V_{\pm}^j$, $j\in\mathbb Z$, which are constructed in Theorem~A. 
\begin{enumerate}
\item Up to a constant, the Fatou coordinates $\Psi_{\pm}^j$ admit common integral asymptotic expansion $\widehat\Psi$, as $z\to 0$ in $V_{\pm}^j$. 

\item Let $\Phi_{\pm}^j$ be any other Fatou coordinate for $f$ on petals $V_{\pm}^j$ such that $$\Phi_{\pm}^j(z)\sim_{z\to 0}-\frac{1}{\alpha-1}z^{-\alpha+1}\boldsymbol\ell^{-m}.$$ Then, on every petal $V_{\pm}^j$, the Fatou coordinate $\Phi_{\pm}^j$ differs from $\Psi_{\pm}^j$ only by an additive constant. Up to a constant, $\Phi_{\pm}^j$ also admit $\widehat\Psi$ as their integral asymptotic expansion, as $z\to 0$ in $V_{\pm}^j.$
\end{enumerate}
\end{prop}
\begin{proof} 

$(1)$ The proof is similar to the proof of \cite[Theorem]{MRRZ2Fatou} for parabolic Dulac germs, the only difference is that \emph{coefficients} are not polynomials in $\boldsymbol\ell^{-1}$, but $\log$-Gevrey sums on $\boldsymbol\ell$-cusps ($\boldsymbol\ell$-images of petals). Suppose that $f(z)=z-z^{\alpha} R_1^{j,\pm}(\boldsymbol\ell)+O(z^{\alpha+\delta})$, $z\in V_j^{\pm}$,\ $\alpha>1$, $\delta>0$. Here, $R_1^{j,\pm}(\boldsymbol\ell)$ are analytic $\log$-Gevrey sums of order strictly bigger than $\frac{\alpha-1}{2}$ of their formal counterpart $\widehat R_1(\boldsymbol\ell)$ on $\boldsymbol\ell$-cusps $\boldsymbol\ell(V_j^{\pm})$, where $V_{j}^{\pm}$ are petals of opening $\frac{2\pi}{\alpha-1}$. Thus, for the Fatou coordinate of a generalized Dulac germ, the \emph{integral sums of length $1$} on $\boldsymbol\ell$-cusps $\boldsymbol\ell(V_{j}^{\pm})$ from Definition~\ref{def:isk}, instead of former integral sums, become the right choice of sections in each block. By Propositions~\ref{prop:closum} - \ref{prop:clodif} of Section~\ref{sec:classes}, the class of all $\log$-Gevrey summable series of order strictly bigger than some fixed $r>0$ on $\boldsymbol\ell(V_{j}^{\pm})$ remains closed to summation, multiplication and differentiation. Note also that the germ $\boldsymbol\ell\mapsto R_1^{j,\pm}(\boldsymbol\ell)$ in the first block of $f$ does not have accumulation of singularities at $0$ in $\boldsymbol\ell(V_j^{\pm}),\ j\in\mathbb Z$. Indeed, if it had, its asymptotic expansion $\widehat R_1(\boldsymbol\ell)$ in power-log scale on $\boldsymbol\ell(V_{j}^{\pm})$ would be $0$. Therefore, the algorithm for constructing the Fatou coordinate block by block, by Taylor expansion of the Abel equation, goes through similarly as in the proof of \cite[Theorem]{MRRZ2Fatou}. 

$(2)$ Fix one petal, e.g. $V_j^{+}$, $j\in\mathbb Z$. Then \begin{equation}\label{eq:sto}\Phi_+^j\circ (\Psi_+^j)^{-1}(w)=w+o(w),\ w\in\Psi_+^j(V_j^+),\ |w|\to\infty.\end{equation} Indeed, $\Psi_+^j(z)$ is injective on $V_{+}^j$, sufficiently close to $0$. This can easily be checked in the logarithmic chart $w=\log z$, where petal $V_+^j$ becomes a convex set. Suppose that $\tilde \Psi_+^j(w):=\Psi_+^j(e^w)$ is not injective in the image of $V_j^+$ in the logarithmic chart, locally around $\{\mathrm{Re}(w)=-\infty\}$. By complex Rolle's theorem, there exist sequences $(w_n)$ and $(v_n)$ such that $\mathrm{Re}(w_n),\,\mathrm{Re}(v_n)\to-\infty$, and such that $\mathrm{Re}\big((\tilde \Psi_+^j)'(w_n)\big)=0$ and $\mathrm{Im}\big((\tilde \Psi_+^j)'(v_n)\big)=0$. Due to continuity of the real and of the imaginary part of the Fatou coordinate, this implies that $(\Psi_+^j)'(0)=0$. This is in contradiction with $(\Psi_+^j)'(z)\sim z^{-\alpha}\boldsymbol\ell^{-m}$, as $z\to 0$ in $V_+^j$, from the construction of the Fatou coordinate $\Psi_+^j$ in Theorem A. Now, from the asymptotic behavior of $\Psi_+^j(z)$, we easily get the exact asymptotic behavior of the inverse $(\Psi_+^j)^{-1}(w)$, as $|w|\to\infty$, $w\in\Psi_+^j(V_j^+)$. It can be computed directly by putting $w=\Psi_+^j(z)=-\frac{1}{\alpha-1}z^{-\alpha+1}\boldsymbol\ell^{-m}\big(1+o(1)\big)$, taking the logarithm and expressing $z$ by $w$. Then the claim \eqref{eq:sto} follows simply by composing with $\Phi_+^j(z)=-\frac{1}{\alpha-1}z^{-\alpha+1}\boldsymbol\ell^{-m}+o(z^{-\alpha+1}\boldsymbol\ell^{-m})$, $z\to 0$. 

Since $\Psi_+^j(V_j^+)$ contains $\{w:\text{Re}(w)>R\}$ for some $R>0$, and since, as in the proof of Lemma~\ref{lem:loic} below, $\Phi_+^j\circ (\Psi_+^j)^{-1}-\mathrm{id}$ is periodic on this domain, we can extend it by periodicity to whole $\mathbb C$, and conclude, using Liouville's theorem, that $\Phi_+^j-\Psi_+^j=C_j$, $C_j\in\mathbb R$.
\end{proof}

\medskip

\subsection{Sectorially analytic normalizations}\label{subsec:konj}\
\smallskip

Let $\widehat\varphi\in\widehat{\mathcal S}^{\alpha,m,\rho}(V)$, $\widehat\varphi(z)=z+z\widehat R_0(\boldsymbol\ell)+\mathrm{h.o.b.}$, where $\widehat R_0(\boldsymbol\ell)\in\boldsymbol\ell\mathbb R[[\boldsymbol\ell]]$, be a block iterated integrally summable series with parameters $(\alpha,m,\rho)\in\mathbb R_{>1}\times\mathbb Z\times\mathbb R$ on petal $V$ of opening $\frac{2\pi}{\alpha-1}$ in the sense of Definition~\ref{def:dai}. Let $\widehat h_0(z)$ be as defined from $\widehat R_0(\boldsymbol\ell)$ in \eqref{eq:kkkk} in Definition~\ref{def:dai}. Let $h_0(z)$ be the analytic germ on $V$, given by the analytic counterpart of formula \eqref{eq:kkkk} for $\widehat h_0(z)$:
\begin{align}\label{eq:kkkka}
h_0(z):=\Big(-\int_{\boldsymbol\ell_0}^{\boldsymbol\ell}&\frac{e^{\frac{\alpha-1}{\eta}}}{\eta^{m+2}}d\eta-\frac{\alpha}{2}(\boldsymbol\ell^{-1}-\boldsymbol\ell_0^{-1})+\big(\frac m 2+\rho\big)(\boldsymbol\ell_2^{-1}+\log\boldsymbol\ell_0)\Big)^{-1}\circ\nonumber\\
&\circ \Big(\int_{\boldsymbol\ell_0}^{\boldsymbol\ell}\frac{e^{\frac{\alpha-1}{\eta}}}{T_0(\eta)\eta^{2}}d\eta-\frac{\alpha}{2}(\boldsymbol\ell^{-1}-\boldsymbol\ell_0^{-1})+b(\boldsymbol\ell_2^{-1}+\log\boldsymbol\ell_0)\Big), \ z\in V.
\end{align}
Here, $T_0$ is the $\log$-Gevrey sum of $\widehat T_0$ from \eqref{eq:kkkk} on $\boldsymbol\ell(V)$, and $\boldsymbol\ell_0\in\boldsymbol\ell(V)$. Note that $h_0$ is unique up to the choice of a constant of integration $\boldsymbol\ell_0\in\boldsymbol\ell(V)$.
\smallskip

Let $\widehat\varphi\circ\widehat h_0$ decompose as in Definition~\ref{def:dai} as:
\begin{equation}\label{eq:ggg} \widehat\varphi\circ\widehat h_0=\circ_{i\in\mathbb N}\widehat\varphi_i,\ \widehat\varphi_i(z)=z+z^{\beta_i}\widehat R_i(\boldsymbol\ell),\end{equation} where $\widehat\varphi_i(z)$ are elementary changes, $\beta_i>1$ are strictly increasing, either finite or tending to $+\infty$, and $\widehat R_i(\boldsymbol\ell)$, $i\in\mathbb N$, belong to a sequence  $\widehat{\mathcal H}$ of integrally summable series and their algebraic combinations of the form \eqref{eq:hak}.
\smallskip

\begin{defi}[Block iterated integral asymptotic expansion]\label{def:iiexp}
Let $\varphi(z)=z+o(z)$ be analytic on some open petal $V$ of opening $\frac{2\pi}{\alpha-1}$, for some $\alpha>1$. Let $\widehat\varphi\in \widehat{\mathcal S}^{\alpha,m,\rho}(V)$ be block iterated integrally summable on $V$ with respect to parameters $(\alpha,m,\rho)$. Let $\widehat h_0\in\widehat{\mathcal L}^{\mathrm{id}}(\mathbb R)$ and $h_0$ analytic on $V$ be as defined in \eqref{eq:kkkka}. Let $\widehat{\mathcal H}$ in decomposition of $\widehat\varphi\circ\widehat h_0$ be as above.

We say that the germ $\varphi$, up to precomposition by $f_c$, $c\in\mathbb R$, admits $\widehat\varphi$ as its \emph{block iterated integral asymptotic expansion} with parameters $(\alpha,m,\rho)$ on $V$ if, for every integration constant $\boldsymbol\ell_0\in\boldsymbol\ell(V)$ in \eqref{eq:kkkka} for $h_0$, and for every integral sum $\mathcal H^{h_0}=\{R_i(\boldsymbol\ell):i\in\mathbb N\}$ of $\widehat {\mathcal H}$ on $\boldsymbol\ell(V)$, there exists $c\in\mathbb R$ such that $(\varphi\circ f_c)\circ h_0$ can be written as the following \emph{infinite asymptotic\footnote{the request is weaker than the pointwise convergence of $\circ_{i\in\mathbb N}\varphi_i$ to $\varphi\circ h_0$} composition}:
\begin{equation}\label{eq:ass}
(\varphi\circ f_c)\circ h_0\sim \circ_{i\in\mathbb N}\varphi_i,\ z\in V,\ z\to 0,
\end{equation}
where $\varphi_i(z)=z+z^{\beta_i}R_i(\boldsymbol\ell)$\footnote{The sums $\varphi_i(\boldsymbol\ell)$ depend on the choice of the constant of integration in $h_0(z)$ and on the choice of constants of integration in all former $R_k(\boldsymbol\ell),\ k\leq i$.}. Here, \emph{asymptotic composition} $\sim$ in \eqref{eq:ass} means that there exists a sequence $(\gamma_n)_n$ of strictly increasing positive numbers tending to $+\infty$, such that, for every $n\in\mathbb N$:
\begin{align}\label{eq:ohh}
(\varphi\circ f_c)\circ h_0=&\varphi_1\circ\ldots\circ\varphi_n+o(z^{\gamma_n}),\ \text{ i.e.}\\
\varphi_n^{-1}\circ\ldots\circ\varphi_1^{-1}\circ\big((\varphi\circ f_c)\circ h_0\big)=&z+o(z^{\gamma_n}),\quad z\in V,\ z\to 0. \nonumber
\end{align}
\end{defi}



Note that $\varphi_i(z)$ admit $\widehat\varphi_i(z)=z+z^{\beta_i}\widehat R_i(\boldsymbol\ell)$, $i\in\mathbb N$, as their Poincar\' e asymptotic expansions, since $R_i(\boldsymbol\ell)$ is integral sum of $\widehat R_i(\boldsymbol\ell)$ of some length. This can be checked in Definition~\ref{def:isk}.


\begin{obs}We do not claim uniqueness of a block iterated integral asymptotic expansion of a general analytic germ $\varphi$ on $V$. However, if $\varphi$ is a normalization on $V$ of a parabolic Dulac germ $f$ from formal class $(\alpha,m,\rho)$ that admits a block iterated integral asymptotic expansion, then its block iterated integral asymptotic expansion on $V$ with parameters $(\alpha,m,\rho)$, if it exists, is unique, up to precomposition by $\widehat f_c$, where $f_c$ is a time-$c$ map of $(\alpha,m,\rho)$-model field $X_1$ as in \eqref{eq:norma1}. Indeed, every block iterated integral asymptotic expansion of analytic normalization on $V$ is a formal normalization of $\widehat f$, by \eqref{eq:ohh} and since $\varphi_i$ admits $\widehat\varphi_i$ as its Poincar\' e asymptotic expansion. On the other hand, formal normalization of $\widehat f$ is by Proposition~\ref{prop:ok} unique up to precomposition by $\widehat f_c$, $c\in\mathbb R$. 
\end{obs} 


\begin{prop} [Block iterated integral asymptotic expansions of analytic reductions to the formal normal form]\label{prop:gi}

Let $f(z)=z-z^\alpha \boldsymbol\ell^m+o(z^\alpha \boldsymbol\ell^m),\ \alpha>1,\ m\in\mathbb Z,$ be a parabolic generalized Dulac germ on $\mathcal R_C$ from the formal class $(\alpha,m,\rho)$. Let $\widehat f$ be its generalized Dulac expansion. Let $V_j^{\pm}$, $j\in\mathbb Z$, be petals of $f$ on $\mathcal R_C$. Let $\widehat\varphi\in\widehat{\mathcal L}^{\mathrm{id}}(\mathbb R)$ be the formal change of variables reducing $\widehat f$ to its $\widehat{\mathcal L}(\mathbb R)$-normal form $\widehat f_1$, which is block iterated integrally summable $($see Proposition~\ref{prop:rednorm}$)$. 

1. There exist analytic changes of variables $\varphi_{\pm}^j(z)=z+o(z)$, $j\in\mathbb Z$, on open petals $V_j^{\pm}$, conjugating\footnote{i.e. such that $f\circ \varphi_j^{\pm}=\varphi_j^{\pm}\circ f_1,\ z\in V_j^{\pm}$.} $f$ to $f_1$ which, up to a precomposition by $f_c$, $c\in\mathbb R$, admit the formal change of variables $\widehat\varphi$ as their  \emph{block iterated integral asymptotic expansion} with parameters $(\alpha,m,\rho)$, as $z\to 0$ on $V_j^\pm$. Here, different choices of constants $c$ are related to different choices of constants in integral sums.

2. Let $\eta_j^{\pm}$ be any other analytic changes of variables on $V_j^{\pm}$ conjugating $f$ to $f_1$, such that $\eta_j^{\pm}(z)=z+o(z),\ z\in V_j^{\pm}$. Then, there exist constants $c_j$, $j\in\mathbb Z$, such that $\eta_j^{\pm}=\varphi_j^{\pm}\circ f_{c_j}$. Moreover, $\eta_j^{\pm}$, up to a precomposition with $f_c$, $c\in\mathbb R$, admit $\widehat\varphi$ as block iterated integral asymptotic expansions with parameters $(\alpha,m,\rho)$, as $z\to 0$ in $V_j^{\pm}$. 
\end{prop}

Due to its importance in view of Definition~\ref{def:jedan} of analytic normalization of a parabolic Dulac germ on a standard quadratic domain, we re-state and prove the statement 2 of Proposition~\ref{prop:gi} in separate Lemma~\ref{lem:loic}, whose proof is in the Appendix. It shows that, in Definition~\ref{def:jedan}, we do not need any assumption on asymptotic expansion of the normalization in power-logarithmic scale, since it necessarily follows.
\begin{lem}[Tangent to identity analytic normalizations of parabolic Dulac germs necessarily admit power-logarithmic asymptotic expansions]\label{lem:loic} Let $f$ be a normalized parabolic generalized Dulac germ in the formal class $(\alpha,m,\rho)$ and let $f_1$ be its $\widehat{\mathcal{L}}^{\mathrm{id}}(\mathbb R)$-formal normal form, with formal conjugacy $\widehat\varphi\in\widehat{\mathcal{L}}^{\mathrm{id}}(\mathbb R)$\footnote{unique only as a \emph{class}, up to precomposition by $f_c$.}. Let $V_j^{\pm}$ be petals of $f$, as in Theorem~A. Let $h_j^{\pm}$ be analytic conjugacies of $f$ to $f_1$ on petals $V_j^{\pm}$, $j\in\mathbb Z$, such that $h_j^{\pm}(z)=z+o(z).$\footnote{Here, $o(z)$ is meant uniformly on petal, but may depend on $j$.} Then, the germs $h_j^{\pm}$ admit the formal conjugacy $\widehat{\varphi}\in\widehat{\mathcal L}^{\mathrm{id}}(\mathbb R)$ as their common block iterated integral asymptotic expansion with parameters $(\alpha,m,\rho)$, as $z\to 0$ in $V_j^{\pm}$, up to precomposition by $f_{c_j}$, $c_j\in\mathbb R$, $j\in\mathbb Z$. 
\end{lem}
\noindent The same statement holds for $\widehat{\mathcal L}_2^{\mathrm{id}}(\mathbb R)$-formal normal form, with formal conjugacy $\widehat\varphi\in\widehat{\mathcal L}_2^\mathrm{id}(\mathbb R)$.
\bigskip

\noindent \emph{Proof of Proposition~\ref{prop:gi}.} \

\emph{Proof of $(1)$}.

\emph{Step 1. Existence of a petalwise analytic normalization $\varphi_j^\pm$ on $V_j^\pm$.} The existence of analytic  conjugating changes $\varphi_j^{\pm}$ on petals $V_j^{\pm}$, $j\in\mathbb Z$, of opening $\frac{2\pi}{\alpha-1}$, where petals are as in Theorem~A, is proven by the existence (by construction) of an analytic Fatou coordinate $\Psi_j^{\pm}$ on $V_j^{\pm}$ in the proof of Theorem~A in Section~\ref{sec:proofA}. Let $\Psi_0$ be a Fatou coordinate of the normal form $f_1$ on $V_j^{\pm}$, and let $\Psi_j^{\pm}$ be a Fatou coordinate of $f$, analytic on $V_j^{\pm}$, $j\in\mathbb Z$, constructed in Theorem~A.  Then $\Psi_j^{\pm}(z),\ \Psi_0(z)\sim_{z\to 0} -\frac{1}{\alpha-1}z^{-\alpha+1}\boldsymbol\ell^{-m},\ z\to 0$. Consequently, $\varphi_j^\pm$ defined by \begin{equation}\label{eq:defif}\varphi_j^\pm(z):=(\Psi_j^{\pm})^{-1}\circ \Psi_0(z)=z+o(z),\ z\in V_j^{\pm},\end{equation} is a petalwise analytic normalization of between $f$. 

Note that, in definition \eqref{eq:defif}, we may have chosen another Fatou coordinate of $f$, differing from $\Psi_j^\pm$ by some additive constants on petals. This, similarly as in the final step of the proof of Lemma~\ref{lem:loic}, gives a conjugacy $\tilde\varphi_j^\pm=\varphi_j^\pm\circ f_{c_j}$, $j\in\mathbb Z$, that differs from $\varphi_j^\pm$ by precomposition by $f_{c_j}$, $c_j\in\mathbb R$.
\smallskip


\smallskip
\emph{Step 2. Proof of the block iterated integral asymptotic expansions of analytic normalizations.} We now prove that \emph{any} conjugacy $\varphi_j^{\pm}$ tangent to the identity admits the block iterated integrally summable formal conjugacy $\widehat\varphi\in\widehat S^{\alpha,m,\rho}(V_j^\pm)$ as its block iterated integral asymptotic expansion on $V_j^\pm$ with parameters $(\alpha,m,\rho)$, in the sense of Definition~\ref{def:iiexp}. 

Recall that $\widehat\varphi$ decomposes uniquely (see \eqref{eq:ggg}) as:
$$
\widehat\varphi\circ\widehat h_0=\circ_{i\in\mathbb N}\widehat\varphi_i,
$$
where $\widehat h_0^{-1}$ is a formal normalization of $\widehat f_2$ (defined in \eqref{eq:t2}) given by \eqref{eq:lll}, taking zero as a constant of formal integration for uniqueness and $b\in\mathbb R$ such that $\widehat h_0^{-1}$ does not contain iterated logarithms. Moreover, \begin{equation}\label{eq:tt1}\widehat\varphi_i(z)=z+z^{\beta_i}\widehat R_i(\boldsymbol\ell),\ i\in\mathbb N,\end{equation} with $\beta_i>1$ strictly increasing, are blockwise \emph{elementary} changes reducing $\widehat f$ to $\widehat f_2$, obtained as in Lemma~\ref{lem:elim1}, taking $0$ as formal integration constant in \eqref{eq:ma1} for $\widehat R_i(\boldsymbol\ell)$, in order to get elementary changes.
Let $\widehat{\mathcal H}=\{\widehat R_i(\boldsymbol\ell):i\in\mathbb N\}$ as described in \eqref{eq:ggg}.

On the other hand, let $(h_0^{j,\pm})^{-1}(z)$ given by formula \eqref{eq:kkkka}, for a fixed choice of integration constant $\boldsymbol\ell_0\in\boldsymbol\ell(V_j^\pm)$, be one analytic normalization of $f_2^{j,\pm}$ on $V_j^\pm$. Compared to formal formula \eqref{eq:lll}, it is an analytic counterpart of $\widehat h_0^{-1}(z)$. Here,
$f_2^{j,\pm}(z),\ z\in V_j^{\pm}$, are defined by \eqref{eq:t2}:
\begin{equation*}
f_2^{j,\pm}(z)=\mathrm{Exp}\Big(\frac{z^\alpha T_0^{j,\pm}(\boldsymbol\ell)}{1+\frac{\alpha}{2}z^{\alpha-1}T_0^{j,\pm}(\boldsymbol\ell)+ bz^{\alpha-1}T_0^{j,\pm}(\boldsymbol\ell)\boldsymbol\ell}\frac{d}{dz}\Big).\mathrm{id}.
\end{equation*}
Here, $\widehat T_0(\boldsymbol\ell)$ is the first block of the generalized Dulac expansion $\widehat f$, and $T_0^{j,\pm}(\boldsymbol\ell)$ are its $\log$-Gevrey sums on $\boldsymbol\ell(V_j^{\pm}),\ j\in\mathbb Z$. 
Note that $f_2^{j,\pm}$ is not necessarily a parabolic generalized Dulac germ, but only a petal-wise germ, since it may not glue on intersections of petals.

Let $\mathcal H^{h_0^{j,\pm}}:=\{R_i^{j,\pm}(\boldsymbol\ell):i\in\mathbb N\}$ be one fixed sequence of integral sums on $\boldsymbol\ell(V_j^{\pm})$, for a fixed choice of integrating constants $\boldsymbol\ell_0\in\boldsymbol\ell(V_j^{\pm})$ in $h_0^{j,\pm}(z)$, and for fixed choices of integrating constants $\boldsymbol\ell_0\in\boldsymbol\ell(V_j^{\pm})$ in succesive integral sums (in each analytic integral \eqref{eq:ma1} defining $R_i^{j,\pm}(\boldsymbol\ell)$, see Definition~\ref{def:isk}). We now put \begin{equation}\label{eq:tt2}\varphi_i^{j,\pm}(z):=z+z^{\beta_i}\widehat R_i(\boldsymbol\ell),\ i\in\mathbb N,\ z\in V_j^{\pm}.\end{equation} The procedure is explained in more detail in the following algorithm for block-by-block changes reducing $\widehat f$ to $\widehat f_2$, deduced in every step formally and analytically in parallel.

\medskip
\emph{Description of block-by-block algorithm.} For simplicity, we do not write here indices of petals. We simply put $V:=V_j^{\pm}$, for any petal. Analogously, $f_2:=f_2^{j,\pm}$, $\varphi_i:=\varphi_i^{j,\pm}$, $h_0:=h_0^{j,\pm}$, analytic on $V$. 

We construct simultaneously the sequence $(\widehat\varphi_i)_{i\in\mathbb N}$ from \eqref{eq:tt1} of formal elementary changes of variables and a sequence $(\varphi_i)_{i\in\mathbb N}$ from \eqref{eq:tt2} of analytic changes of variables on $V$ by \emph{blockwise eliminations}, transforming $f$ to $f_2$ on $V$. 

\begin{enumerate}
\item[0.] Put $\widehat {^1f}:=\widehat f$, $^1f:=f$.

\item[1.] Let $i\in\mathbb N$. By Lemma~\ref{lem:elim1}, we find the $i$-th elementary change $\widehat\varphi_{i}(z)=z+z^{\beta_{i}} \widehat R_{i}(\boldsymbol\ell)$, such that:
\begin{equation}\label{eq:jaoj}\widehat R_i(\boldsymbol\ell)=-e^{-\frac{\beta_i-\alpha}{\boldsymbol\ell}}\boldsymbol\ell^{m}\int e^{\frac{\beta_i-\alpha}{\boldsymbol\ell}}\boldsymbol\ell^{-2m-2}\widehat T_i(\boldsymbol\ell) d\boldsymbol\ell.
\end{equation}
In the first step $(i=1)$, $\widehat T_1(\boldsymbol\ell)$ is the difference of the series in $\boldsymbol\ell$ in the second block of $\widehat f$ and of $\widehat f_2$ ($\log$-Gevrey on $V$). Further, for $i\in\mathbb N$, $i>1$, we compute:
$$\widehat {^i f}(z):=\widehat \varphi_{i-1} \circ \widehat {^{i-1} f} \circ \widehat \varphi_{i-1}^{-1}(z)=\widehat f_2(z)+\big(z^{\gamma_i}\widehat T_i(\boldsymbol\ell)+\text{h.o.b}\big).$$ Here, $\gamma_i> \alpha$ strictly increase to $+\infty$, and $\widehat T_i(\boldsymbol\ell)$ (as also every other \emph{coefficient} of new $\widehat {^i f}$) is an algebraic combination (with operations $+,\cdot,\frac{d}{d\boldsymbol\ell}$) of series integrable of length strictly lower than $i$ (of the \emph{coefficients} of $\widehat {^{i-1} f}$, of $\widehat f_2$ and of $\widehat R_j(\boldsymbol\ell)$ from all the previous elementary changes, $1\leq j<i$).  The constant of formal integration is always taken to be $0$, in order to get \emph{elementary}, one-block changes of variables $\widehat\varphi_i$.

We put the sequence of $\widehat R_i(\boldsymbol\ell)$ obtained in this formal step-by-step construction in set $\widehat{\mathcal H}$. The corresponding integral exponents are by \eqref{eq:jaoj} equal to $\beta_i$, and they are strictly increasing and tend to $+\infty$ by the algorithm.
\smallskip

\item[2.] On the other hand, \eqref{eq:jaoj} can be simultaneously solved \emph{analytically} on the $\boldsymbol\ell$-cusp $\boldsymbol\ell(V)$, 
\begin{equation}\label{eq:jaoj1} R_i(\boldsymbol\ell)=-e^{-\frac{\beta_i-\alpha}{\boldsymbol\ell}}\boldsymbol\ell^{m}\int_{*}^{\boldsymbol\ell} e^{\frac{\beta_i-\alpha}{\eta}}\eta^{-2m-2} T_i(\eta) d\eta,\ \boldsymbol\ell\in\boldsymbol\ell(V).
\end{equation}
Here, $*$ is $0$ if $(\beta_i-\alpha,\mathrm{ord}(\widehat T_i)-2m-1)\succ (0,0)$ (lexicographically) in this step, and $\boldsymbol\ell_0\in\boldsymbol\ell(V)$, if not (if the subintegral function is not bounded at $0$). Note that in the first \emph{finitely many} steps we may choose $\boldsymbol\ell_0$ arbitrarily, which \emph{changes the constant of integration}. The integration path is not important, as long as it stays in $\boldsymbol\ell(V)$, since the subintegral function is analytic on $\boldsymbol\ell(V)$ and the cusp simply connected.

Since $\widehat T_i(\boldsymbol\ell)$ is an algebraic expression (with respect to $+,\cdot,\frac{d}{d\boldsymbol\ell}$) in integrally summable series in $\widehat{\mathcal H}$ of lengths strictly smaller than $i$ (of the \emph{coefficients} of $\widehat {^{i-1} f}$, of $\widehat f_2$ and of $\widehat R_j(\boldsymbol\ell)$ from all the previous elementary changes, $1\leq j<i$), once that we have prescribed the summation constants in previous steps, by Proposition~\ref{prop:uis} it is uniquely summable. Supposing that we have chosen integration constants for $R_j(\boldsymbol\ell)$ in \eqref{eq:jaoj1} in all previous steps $j<i$, this sum is unique, analytic on $V$, and we denote it by $T_i(\boldsymbol\ell)$. Now we again freely chose the $i$-th constant of integration  in \eqref{eq:jaoj1} in the $i$-th step. This free choice is done in first finitely many steps, while $(\beta_i-\alpha,\text{ord}(\widehat T_i)-2m-1)\prec (0,0)$; in later steps we always choose $0$. So the sum $R_i(\boldsymbol\ell)$ of $\widehat R_i(\boldsymbol\ell)$ is \emph{unique} up to choice of integrating constants of $\eqref{eq:jaoj1}$ in all previous steps up to $i$, included. We denote by $\mathcal H^{h_0}=\{R_i(\boldsymbol\ell):i\in\mathbb N\}$ one such choice of integration constants in all steps, as in \eqref{eq:ih}. It thus becomes one integral sum of $\widehat {\mathcal H}$.

We put now:
$$
\varphi_i(z):=z+z^{\beta_i}R_i(\boldsymbol\ell).
$$
The function $\varphi_i$ is analytic on $V$. We do it for every petal, and thus get $\varphi_{j,\pm}^i(z)$ analytic on $V_j^\pm$, $j\in\mathbb Z$. 

Note that the integration path in \eqref{eq:jaoj1} does not matter. For any petal $V=V_j^{\pm}$ and any step of iteration, we may take any path lying inside the petal. Indeed, all \emph{coefficient functions} in blocks of $f$ and $f_2$, since $f$ is a generalized Dulac germ, are analytic functions on corresponding $\boldsymbol\ell$-cusps $\boldsymbol\ell(V)$. Moreover, the operations in algebraic combinations for $\widehat T_i(\boldsymbol\ell)$ exclude division, by the above described algorithm. Therefore, the subintegral function in the formula \eqref{eq:jaoj1} for $R_i(\boldsymbol\ell)$, $i\in\mathbb N$, does not have any singularities on $\boldsymbol\ell$-cusps $\boldsymbol\ell(V)$, which are simply connected. 
\end{enumerate}

\smallskip

\noindent Note that $\varphi_{j,\pm}^i$, $i\in\mathbb N$, from \eqref{eq:tt2}, obtained by the above algorithm, are analytic on petals $V_j^\pm$, $j\in\mathbb Z$, but in general they \emph{do not glue} to global analytic germs $\varphi^i(z)$ on $\mathcal R_C$. Neither we claim that $^i f(z)$, for any $i>1$, are analytic on $\mathcal R_C$ (they are at best analytic petalwise).

Every elementary change $\varphi_i(z)$ is an analytic function on $V$, by \eqref{eq:jaoj1}. For $i$ such that $\big(\beta_i-\alpha,\text{ord}(\widehat T_i)-2m-1\big)\succ (0,0)$, we choose a canonical way of integration, $\int_0^{\boldsymbol\ell} * \,dw$, so the integral is a unique analytic function on $V$. Otherwise, we freely choose the initial point $\boldsymbol\ell_0\in\boldsymbol\ell(V)$ in the integral $\int_{\boldsymbol\ell_0}^{\boldsymbol\ell} * \,dw$, so the integral is unique only up to a term $C_i e^{-\frac{\alpha-\beta_i}{\boldsymbol\ell}}\boldsymbol\ell^{m}$, $C_i\in\mathbb R$. This ambiguity corresponds to adding an exponentially small term $C_i e^{-\frac{\alpha-\beta_i}{\boldsymbol\ell}}\boldsymbol\ell^{m}$ to $\varphi_i(z)$, $C_i\in\mathbb R$, for all the changes of variables before the residual ($\beta_i\leq \alpha$).

\bigskip

Put again $V:=V_j^{\pm}$. Take \emph{any} choice of integrating constant $\boldsymbol\ell_0\in \boldsymbol\ell(V)$ in $h_0(z)$ (that is, take \emph{any} analytic normalization $h_0^{-1}$ of $f_2$ on $V$) and \emph{any} choice of integral sums $\mathcal H^{h_0}=\{R_i(\boldsymbol\ell):i\in\mathbb N\}$ of $\widehat {\mathcal H}$ (choices of integrating constants $\boldsymbol\ell_0\in\boldsymbol\ell(V)$), from the above algorithm. In the algorithm steps $R_i(\boldsymbol\ell)\in\mathcal H^{h_0},\ i\in\mathbb N,$ are deduced solving the corresponding Lie bracket equations for block-by-block eliminations, germwise. Now put:
$$\varphi_i(z):=z+z^{\beta_i}R_i(\boldsymbol\ell),\ i\in\mathbb N.$$ Since, in the algorithm, we remove block by block, we get that there exists a strictly increasing sequence $(\gamma_n)_n$, $\gamma_n>1$, tending to $+\infty$, such that, for every $n\in\mathbb N$, on $V$ it holds that:
\begin{align}\label{eq:prvaa}                                                                                                      
(\varphi_1^{-1}\circ\ldots\circ\varphi_n^{-1})\circ f\circ (\varphi_1\circ\ldots\circ\varphi_n)&=f_2+o(z^{\gamma_n})\\
&=h_0^{-1}\circ f_1\circ h_0+o(z^{\gamma_n}),\ z\to 0.\nonumber
\end{align}
On the other hand, for \emph{any} analytic normalization $\varphi$ of $f$ on $V$ it holds that:
\begin{equation}\label{eq:drugaa}
\varphi^{-1}\circ f\circ\varphi=f_1.
\end{equation}
Now, putting \eqref{eq:drugaa} in \eqref{eq:prvaa}, and denoting by $T_n:=\varphi^{-1}\circ (\varphi_1\circ\ldots\circ\varphi_n)$, we get:
\begin{align}\label{eq:jj}
T_n^{-1}\circ &f_1\circ T_n=h_0^{-1}\circ f_1\circ h_0+o(z^{\gamma_n}),\nonumber \\
h_0\circ T_n^{-1}\circ &f_1\circ T_n\circ h_0^{-1}=f_1+o(z^{\gamma_n}),\ z\to 0.
\end{align}
Using $f_c\circ f_1=f_1\circ f_c$, $c\in\mathbb R$, and defining $r_n:=T_n\circ h_0^{-1}-f_c$, $r_n(z)=O(z^{>1}),\ n\in\mathbb N$, we get from \eqref{eq:jj}:
\begin{align*}
&f_1\circ (f_c+r_n)=(r_n+f_c)\circ f_1+o(z^{\gamma_n}),\\
&f_1\circ (f_c+r_n)-f_1\circ f_c=r_n\circ f_1+o(z^{\gamma_n}).
\end{align*}
We deduce, comparing the leading terms of both sides in the last equality, that $r_n(z)=o(z^{\gamma_n-\alpha}),\ z\to 0,\ n\in\mathbb N$. Thus, we get:
$$\varphi\circ f_c=(\varphi_1\circ\ldots\varphi_n)\circ h_0^{-1}+o(z^{\gamma_n-\alpha}).$$ This proves that $\widehat\varphi$ is the block iterated integral asymptotic expansion of \emph{any} analytic normalization $\varphi$ of $f$ on $V$, by Definition~\ref{def:iiexp}. 
\medskip

\emph{Proof of $(2)$}. Lemma~\ref{lem:loic}.
\hfill $\Box$
\bigskip

Note that the existence of a power-logarithmic asymptotic expansion is not immediate for a sectorial Fatou coordinate or a sectorial conjugacy of a parabolic generalized Dulac germ, as the following Example~\ref{ex:exampl} shows. However, it is verified if we assume an asymptotic behavior as in Propositions~\ref{prop:Fatouexp} and \ref{prop:gi} (2). See also Lemma~\ref{lem:loic}.

\begin{example}\label{ex:exampl}[A Fatou coordinate for a parabolic generalized Dulac germ which does not admit a sectional asymptotic expansion in $\widehat{\mathfrak L}^\infty(\mathbb R)$]

In Theorem~A and its proof in Section~\ref{sec:proofA}, we have constructed a holomorphic Fatou coordinate $\Psi$ of a parabolic generalized Dulac germ $f$ on an invariant petal $V$ of opening $\frac{2\pi}{\alpha-1}$, which admits an integral sectional asymptotic expansion equal to the formal Fatou coordinate $\widehat \Psi\in\widehat{\mathcal L}_2^\infty(\mathbb R)$, up to a constant. Now we define another holomorphic Fatou coordinate $\Psi_1$ on $V$ by
$$
\Psi_1(z):=\Psi(z)+g_0(e^{2\pi i\Psi(z)}),\ z\in V,
$$
where $g_0$ is \emph{any} germ analytic on the doubly punctured sphere (without poles $0$ and $\infty$). We may take, for example,
\begin{align*}
&\Psi_1(z):=\Psi(z)+\sin (2\pi \Psi(z)),\ z\in V,\ \text{or}\\
&\Psi_2(z):=\Psi(z)+ce^{2\pi i\Psi(z)},\ c\in\mathbb R,\ z\in V.
\end{align*}
Due to the unbounded exponential term $\sin (2\pi \Psi(z))$ or $ce^{2\pi i\Psi(z)}$, $\Psi_{1,2}$ are Fatou coordinates of $f$ that do not admit sectional asymptotic expansions in $\widehat{\mathfrak L}^\infty(\mathbb R)$ on $V$, as $z\to 0$.
\end{example}
\smallskip

\subsection{Examples on $\mathbb R_+$.}\

Note that, by Propositions~\ref{prop:Fatouexp} and \ref{prop:gi}, there exists a Fatou coordinate $\Psi$ i.e. conjugacy $\varphi$, \emph{unique} up to a simple transformation, analytic on a petal of opening $\frac{2\pi}{\alpha-1}$ and admitting a sectional asymptotic expansion in $\widehat{\mathfrak L}^\infty(\mathbb R)$. Moreover, the expansion is then necessarily (iterated) integral sectional expansion. However, for parabolic Dulac germs defined only on $\mathbb R_+$, as in \cite{MRRZ2Fatou}, this is not the case.

Indeed, let $f$ be a parabolic Dulac germ on $\mathbb R_+$. We construct different Fatou coordinates analytic on $(0,d)$, $d>0$, that admit sectional power-logarithmic asymptotic expansions in $\widehat{\mathcal L}_2^\infty(\mathbb R)$. In \cite{MRRZ2Fatou}, it is proven that on $(0,d)$ there exists a \emph{unique}, up to an additive constant, Fatou coordinate for a parabolic Dulac germ $f$ that admits an \emph{integral} sectional expansion, as $x\to 0$. Moreover that, up to a constant, this expansion is equal to the formal Fatou coordinate $\widehat\Psi\in\widehat{\mathcal L}_2^\infty(\mathbb R)$. 

We give in Example~\ref{ex:tri} an example of two different Fatou coordinates of a parabolic Dulac germ on $\mathbb R_+$ with the same sectional asymptotic expansion in $\widehat{\mathfrak L}^\infty(\mathbb R)$, as $x\to 0$, but with respect to different section functions.

\begin{example}\label{ex:tri} Take a parabolic Dulac germ $f$ on $(0,d)$. Let
$\Psi$ be its analytic Fatou coordinate on $(0,d)$ constructed algorithmically as in \cite[Theorem]{MRRZ2Fatou}. It admits the integral sectional asymptotic expansion, unique up to an additive constant, and equal to the formal Fatou coordinate $\widehat\Psi=\sum_{i=1}^{\infty}\widehat T_i(\boldsymbol\ell)x^{\alpha_i}\in\widehat{\mathcal L}_2^\infty(\mathbb R),\ \widehat T_i\in \widehat{\mathcal L}_{\boldsymbol\ell}^\infty(\mathbb R)$, with $(\alpha_i)_i$ strictly increasing to $+\infty$ or finite. Let us now define another Fatou coordinate $\Psi_1$ on $(0,d)$ by:
$$
\Psi_1(x):=\Psi(x)+\sin (2\pi \Psi(x)),\ x\in(0,d).
$$
The germ $\Psi_1$ is obviously also a Fatou coordinate for $f$ and analytic on $(0,d)$. Let
$$
\Psi(x)=T_1(\boldsymbol\ell)x^{\alpha_1}+T_2(\boldsymbol\ell)x^{\alpha_2}+T_3(\boldsymbol\ell)x^{\alpha_3}+\mathrm{h.o.b},
$$
be the integral sectional asymptotic expansion for $\Psi$ constructed algorithmically in \cite{MRRZ2Fatou}. Here, $T_i(\boldsymbol\ell)$ are analytic on $(0,d)$ and admit the Poincar\' e asymptotic expansion $\widehat T_i(\boldsymbol\ell)\in\widehat{\mathcal L}_{\boldsymbol\ell}^\infty(\mathbb R)$, $i\in\mathbb N$. Obviously, $\alpha_1<0$. The Fatou coordinate $\Psi_1$ admits the same sectional asymptotic expansion $\widehat\Psi$, as $x\to 0$:
$$
\Psi_1(x)=\tilde T_1(\boldsymbol\ell)x^{\alpha_1}+\tilde T_2(\boldsymbol\ell)x^{\alpha_2}+\tilde T_3(\boldsymbol\ell)x^{\alpha_3}+\mathrm{h.o.b},
$$
but with the following choice of sections:
\begin{align*}
&\tilde T_1(\boldsymbol\ell)=T_1(\boldsymbol\ell)+\sin \big(2\pi \Psi(e^{-\frac{1}{\boldsymbol\ell}})\big)\cdot e^{\frac{\alpha_1}{\boldsymbol\ell}},\\
&\tilde T_k(\boldsymbol\ell)= T_k(\boldsymbol\ell),\ k\geq 2,\ \boldsymbol\ell\in\boldsymbol\ell(0,d).
\end{align*} 
Obviously, for the Poincar\' e power expansions of $\tilde T_1$ and $T_1$, as $\boldsymbol\ell\to 0$, it holds that $\widehat {\tilde T_1}(\boldsymbol\ell)=\widehat T_1(\boldsymbol\ell)$, since $\sin (2\pi \Psi(e^{-1/\boldsymbol\ell}))e^{\frac{\alpha_1}{ \boldsymbol\ell}}$ for $\alpha_1<0$ is exponentially small with respect to $\boldsymbol\ell$, as $\boldsymbol\ell\to 0$. This is due to the boundedness of the sine function on $\mathbb R$. Note that on a sector in $\mathbb C$ around the $x$-axis of an arbitrarily small opening this is not any more true. 
\end{example}               
\medskip

\subsection{Sectorially analytic conjugacies.}\label{subsec:types}\ 
\smallskip

In Subsection~\ref{subsec:konj} in Proposition~\ref{prop:gi} we have proved that there exists a unique, up to precomposition by $f_c$, $c\in\mathbb R$, sectorially analytic reduction of a generalized Dulac germ $f$ to its formal normal form $f_1$ which admits formal reduction $\widehat\varphi\in\widehat{\mathcal L}^{\mathrm{id}}(\mathbb R)$ as its block iterated integral asymptotic expansion. Here, we derive similar results, just slightly more complicated, for the conjugacies conjugating two parabolic generalized Dulac germs. 
\smallskip

We now introduce the following definition:
\begin{defi}\label{def:giiaa} Let $f$ and $g$ be two normalized parabolic generalized Dulac germs which are $\widehat{\mathcal L}(\mathbb R)$-formally conjugated and let $V_{j}^{\pm}$ be their common\footnote{Since $f$ and $g$ belong to the same formal class, the invariants $\alpha$ and $m$ are the same, so, by Theorem~A $(1)$, they share common petals $V_j^{\pm}$ of opening $\frac{2\pi}{\alpha-1}$,\ $j\in\mathbb Z$.} petals, $j\in\mathbb Z$. Let $\varphi_{j}^{\pm}$ be analytic conjugacies of $g$ and $f$ on petals $V_j^{\pm}$, $j\in\mathbb Z$. We say that $\varphi_j^{\pm}$ admit $\widehat\varphi$ as its \emph{generalized block iterated integral asymptotic expansion} if they can be decomposed as $\varphi_j^{\pm}=\varphi_{j,1}^{\pm}\circ (\varphi_{j,2}^{\pm})^{-1}$, $j\in\mathbb Z$, where:

(1) $\varphi_{j,1}^{\pm}$ and $\varphi_{j,2}^{\pm}$ are analytic reductions of $g$ resp. $f$ to the formal normal form $f_1$ on $V_j^{\pm}$,  

(2) $\varphi_{j,1}^{\pm}$ and $\varphi_{j,2}^{\pm}$ admit formal conjugacies $\widehat\varphi_1$ resp. $\widehat\varphi_2$ as their block iterated integral asymptotic expansions.
\end{defi}
We see that, if such a decomposition into analytic reductions of $f$ and $g$ exists, it is unique up to a precomposition of $\varphi_1,\,\varphi_2$ by $f_c,\ c\in\mathbb R$, due to the uniqueness of analytic reductions in Proposition~\ref{prop:gi}. We have shown furthermore in Subsection~\ref{subsec:dva} that block iterated integral asymptotic expansions are always meant only up to such precompositions (Definition~\ref{def:iiexp}). 
\smallskip

Now, we can finally state the more general form of Proposition~\ref{prop:gi} for sectorial analytic conjugacies between two parabolic generalized Dulac germs. The proof is similar and omitted.
\begin{prop} [Generalized block iterated integral asymptotic expansions of sectorially analytic conjugacies]\label{prop:gic}

Let $f$ and $g$ be $\widehat{\mathcal L}(\mathbb R)$-formally conjugated normalized parabolic generalized Dulac germs on $\mathcal R_C$. Let $\widehat\varphi\in\widehat{\mathcal S}(V_j^\pm)$ be a formal change of variables\footnote{unique up to intermediate composition with flow $\widehat f_c$, $c\in\mathbb R$, as described after the proof of Proposition~\ref{prop:gDul}} conjugating $\widehat f$ to $\widehat g$.
Then, there exist analytic changes of variables $\varphi_{\pm}^j(z)=z+o(z)$, $j\in\mathbb Z$, on open petals $V_j^{\pm}$, conjugating $f$ to $g$, which admit the formal change of variables $\widehat\varphi$ as their  \emph{generalized block iterated integral asymptotic expansion}, as $z\to 0$. 

Moreover, any other sectorially analytic changes of variables on $V_j^{\pm}$ conjugating $f$ to $g$, such that $\eta_j^{\pm}(z)=z+o(z),\ z\in V_j^{\pm}$, admit $\widehat\varphi$ as their generalized block iterated integral asymptotic expansions, as $z\to 0$ in $V_j^{\pm}$. 
\end{prop}

Note that the second part of Proposition~\ref{prop:gic} is a counterpart of Lemma~\ref{lem:loic}, but for analytic conjugacies between two parabolic generalized Dulac germs. It proves Proposition~\ref{prop:natkonj} from Section~\ref{sec:introduction}. It states that, if two parabolic generalized Dulac germs from the same $\widehat{\mathcal L}(\mathbb R)$-formal class are analytically conjugated in the \emph{weak} sense of Definition~\ref{def:jedan}, then their conjugacy expands asymptotically in the generalized block iterated integrally summable class. This is important in the proof of Theorem~B, stating that $f$ and $g$ are analytically conjugated inside generalized block iterated integrally summable class if and only if they have the same horn maps.


\medskip
\emph{Proof of Proposition~\ref{prop:natkonj}}.
It is a direct corollary of Proposition~\ref{prop:gic}, in the case that $f$ and $g$ are analytically conjugated. In this case all sectorial changes $\varphi_{\pm}^j$, $j\in\mathbb Z$, glue to an analytic conjugacy $\varphi$ on a standard quadratic domain.

\hfill $\Box$

\section{Proof of Theorem A: the flower dynamics and the Fatou coordinate of a parabolic generalized Dulac germ}\label{sec:proofA}

Let 
$$
f(z)=z-az^\alpha \boldsymbol\ell^m+o(z^\alpha \boldsymbol\ell^m),\ z\in\mathcal R_C,\ a>0,\ \alpha>1,\ m\in\mathbb Z,
$$
be a parabolic generalized Dulac germ defined on a standard quadratic domain $\mathcal R_C$, $C>0$, see Definition~\ref{def:gD}.
\smallskip

We prove here Theorem~A stated in Section~\ref{sec:introduction}. We first prove statement (1) of Theorem~A, separately stated in more general Proposition~\ref{prop:leau} from Section~\ref{sec:introduction}. Proposition~\ref{prop:leau} states that the dynamics of a germ $f$ satisfying uniform estimates \eqref{eq:q} on $\mathcal R_C$ is a generalization to the standard quadratic domain of the \emph{Leau-Fatou} flower-like dynamics for parabolic analytic germs, see e.g. \cite{loray} for description. More precisely, we show, using only \emph{uniform} estimate \eqref{eq:uniest} on the asymptotic behavior of $f$ and the fact that $f$ is analytic on $\mathcal R_C$, that there exist infinitely many overlapping attracting and repelling dynamically invariant petals $V_j^+$ resp. $V_j^-$, $j\in\mathbb Z$, along $\mathcal R_C$ (or along a smaller standard quadratic domain, in the sense of germs). The petals are centered at countably many tangential complex directions $1^{\frac{1}{\alpha-1}}$ resp. $(-1)^{\frac{1}{\alpha-1}}$ on $\mathcal R_C$. Each petal is of opening angle $\frac{2\pi}{\alpha-1}$. Note here that $\alpha-1>0$ is in general a \emph{real number}, and not necessarily an integer. The estimate \eqref{eq:uniest} is important for proving that the size of invariant petals for the dynamics does not drop to $0$ faster than prescribed by a standard quadratic domain.
\smallskip

The following proposition shows that the uniform estimate \eqref{eq:uniest} in the definition of a parabolic generalized Dulac germs is automatically satisfied for parabolic Dulac germs, due to the uniform Dulac asymptotics.

\begin{prop}\label{prop:unidulac} Let $f(z)=z-az^\alpha\boldsymbol\ell^m+o(z^\alpha\boldsymbol\ell^m)$, $a>0,\ \alpha>1,\ m\in\mathbb N_0^-$, be a parabolic Dulac germ, defined on a standard quadratic domain $\mathcal R_C$.
 Then there exists a uniform constant $c>0$ such that:
$$
|f(z)-z+az^{\alpha}\boldsymbol\ell^m|\leq c |z^{\alpha}\boldsymbol\ell^{m+1}|,\ z\in\mathcal R_{C}.
$$
\end{prop}
\noindent The proof is in the Appendix.  
\medskip

Since on repelling petals we work with the germ $f^{-1}$ instead of $f$, we need the following Proposition:
\begin{prop}\label{prop:grupa} \

\begin{enumerate}
\item Formal parabolic generalized Dulac series form a group under composition.
\item Parabolic generalized Dulac germs on a standard quadratic domain (germified) form a group under composition.
\end{enumerate}
\end{prop}
\noindent The proof is in the Appendix.

\bigskip

\noindent \emph{Proof of Theorem~A.} We adapt the methods from e.g. \cite[Theorem 2.4.1]{loray} or \cite{loray2} for parabolic analytic germs. Estimates similar to those that we use in the proof have already been deduced in \cite{MRRZ2Fatou} for parabolic Dulac germs. We repeat here only the crucial steps.
\medskip

\noindent \emph{Proof of statement (1): petals and local discrete dynamics.} 

\smallskip

\emph{Proof of Proposition~\ref{prop:leau}}.
Consider the chart \begin{equation}\label{eq:wch}w=-\frac{1}{a(\alpha-1)}z^{-\alpha+1}\boldsymbol\ell^{-m}.\end{equation} It is a multivalued function, univalued on $\mathcal R_C$. It maps bijectively $\mathcal R_C$ to a neighborhood of the infinity on the Riemann surface $\mathcal R$ of the logarithm, which we will denote by $w\in \widetilde{\mathcal R}$. That is, $\widetilde{\mathcal R}$ is the image of $\mathcal R_C$ in the $w$-chart. 
In this new chart, $z\mapsto f(z)$ transforms to $w\mapsto F(w)$, which is \emph{almost} a translation by $1$. Note that $-\frac{1}{a(\alpha-1)}z^{-\alpha+1}\boldsymbol\ell^{-m}$ is in fact the first asymptotic term of the Fatou coordinate for $f$. Indeed, by the above change of variables, \eqref{eq:q} directly transforms to:
\begin{equation}\label{eq:before}
|F(w)-(w+1)|\leq c (\log R)^{-1},\ c>0,
\end{equation}
where $w\in\widetilde{\mathcal R}$. For the change of variables and a detailed computational proof of \eqref{eq:before}, see \cite[Proof of Proposition 6.2, (6.15),\,(6.16)]{MRRZ2Fatou}. Here, $R>0$ can be taken arbitrarily big, if $|z|$ is made sufficiently small (i.e. if $|w|$ is taken sufficiently big, $|w|>R$). 
Here, the constant $c$ does not depend on the level of the Riemann surface of the logarithm $\widetilde{\mathcal R}$, that is, there exists a global constant for all arguments as $|z|\to 0$, due to the \emph{uniform} estimate \eqref{eq:q}. 

We deduce now \emph{invariant domains for the dynamics} of $F(w)$ in the $w$-chart on $\widetilde{\mathcal R}$, and show that they correspond to open attractive petals at the origin of $\mathcal R_{C}$ for the dynamics of $f(z)$. Then, repeating the same procedure for the inverse $f^{-1}(z)=z+az^\alpha\boldsymbol\ell^{m}+o(z^\alpha\boldsymbol\ell^{m})$, which, by Proposition~\ref{prop:grupa}, satisfies $|f^{-1}(z)-z-az^\alpha\boldsymbol\ell^{m}|\leq d|z^\alpha\boldsymbol\ell^{m+1}| ,\ d>0,\ z\in\mathcal R_C$, we get the repelling petals at the origin of $\mathcal R_C$.

\noindent Define, as in \cite{MRRZ2Fatou}, the angle $\alpha_R>0$ such that:
$$
\sin(\alpha_R):=c(\log R)^{-1}.
$$
Obviously, as $R\to\infty$, $\alpha_R\to 0$. Due to \eqref{eq:before}, there exists $R_0>0$ such that, for every $R>R_0$ sufficiently big, the sectors $W_{R,\alpha_R}\cap \widetilde{\mathcal R}$ at infinity (of radius $|w|>R$ and opening $2\alpha_R$), centered at $\varphi=2k\pi$, $k\in\mathbb Z$, are invariant for the dynamics. Note that $\alpha_R$ do not change with the level of $\mathcal R_{C}$. Increasing the radius $R\to\infty$, we get petals at infinity of opening $2\pi$ centered at $2k\pi,\ k\in\mathbb Z$, covering $\mathcal R_{C}$, as the union. Returning to the original variable $z$, the result follows. Indeed, returning to the variable $z$, the preimages of the directions $2k\pi$ are tangential to $1^{-\frac{1}{\alpha-1}}$ at $z=0$. We repeat the procedure for the inverse $f^{-1}$, using Proposition~\ref{prop:grupa}. We denote invariant petals on $\mathcal R_{C}$ constructed in this way by $V_{j}^\pm$, $j\in\mathbb Z$.
\hfill $\Box$
\smallskip

The statement $(1)$ now follows by Propositions~\ref{prop:unidulac} and \ref{prop:grupa}, which prove that a parabolic generalized Dulac germ and its inverse satisfy uniform bounds \eqref{eq:q}, and then directly applying Proposition~\ref{prop:leau}.
\medskip

\noindent \emph{Proof of statement (2): the sectorial Fatou coordinates and their asymptotic expansion.}

We now repeat the procedure for constructing the Fatou coordinate of a parabolic Dulac germ from \cite[Section 6]{MRRZ2Fatou}, but for a parabolic generalized Dulac germ in the complex domain (on invariant attractive petals just constructed above). 

The construction of the Fatou coordinate is carried out block-by-block, using Abel equation, simultaneously formally and in the sense of germs on petals $V_j^{\pm}$. 

By formal construction of the Fatou coordinate $\widehat\Psi$ of $\widehat f$ in Proposition~\ref{prop:fffatou} in Section~\ref{sec:summa}, the \emph{infinite part} of the Fatou coordinate for generalized parabolic Dulac germs is a sum of \emph{finitely many} first blocks $\widehat\Psi_1,\ldots,\widehat \Psi_n$, $n\in\mathbb N$, given, by \eqref{eq:fati}, by integrals of the type:
$$
\widehat\Psi_i(z)=\int e^{-\frac{\beta_i}{\boldsymbol\ell}}2\widehat H_i(\boldsymbol\ell)d\boldsymbol\ell,
$$ 
where $\beta_i$, $i\in\{1,\ldots,n\}$, are strictly increasing, such that $(\beta_i,\mathrm{ord}_{\boldsymbol\ell}(\widehat H_i))\preceq (0,-1)$, and $\widehat H_i(\boldsymbol\ell)$ is $0$-integrally summable on $\boldsymbol\ell(V_j^{\pm})$. Note that only first \emph{finitely many} blocks  deduced consecutively from the Abel equation have negative powers of $z$, since exponents $\alpha_n$ in $\widehat f(z)$ are strictly increasing and finitely generated. 

Let $\Psi_1^{j,\pm},\ldots,\Psi_n^{j,\pm}$ be holomorphic counterparts of $\widehat \Psi_1,\ldots,\widehat\Psi_n$ on $V_j^{\pm}$, given by integrals of the form:
$$
\Psi_i^{j,\pm}(z)=\int_{\boldsymbol\ell_0}^{\boldsymbol\ell} e^{-\frac{\beta_i}{w}} H_i^{j,\pm}(w)\,dw,
$$
Here, $\boldsymbol\ell_0\in\boldsymbol\ell(V_j^{\pm})$, and $H_i^{j,\pm}(\boldsymbol\ell)$ is $0$-integral sum of $\widehat H_i(\boldsymbol\ell)$ on $\boldsymbol\ell(V_j^{\pm})$. 

We now put $\Psi_\infty^{j,\pm}:=\Psi_1^{j,\pm}+\ldots+\Psi_n^{j,\pm}$. It is well defined and holomorphic on $V_j^{\pm}$, $j\in\mathbb Z$, and we call it the \emph{infinite part} of the analytic Fatou coordinate $\Psi_j^{\pm}$ on $V_j^{\pm}$.  Different choices of the initial condition $\boldsymbol\ell_0\in V_j^{\pm}$ are related only to different additive real constants in $\Psi_\infty^{j,\pm}$. Therefore, $\Psi_\infty^{j,\pm}$ is a holomorphic function on $V_j^{\pm}$, unique up to an additive constant.

\noindent Now, the Abel equation
$$
\Psi_j^{\pm}(f(z))-\Psi_j^{\pm}(z)=1 \text{ on }V_j^{\pm},
$$
transforms to
\begin{equation}\label{eq:later}
R_j^{\pm}(f(z))-R_j^{\pm}(z)=\delta_j^{\pm}(z),\ z\in V_j^{\pm},
\end{equation}
where $\Psi_j^{\pm}=\Psi_{\infty}^{j,\pm}+R_j^{\pm}$, $j\in\mathbb Z$. Note that, to prove the existence (and uniqueness) of petalwise analytic Fatou coordinate $\Psi_j^{\pm}$, we need yet to prove the existence (and uniqueness) of holomorphic $R_j^{\pm}(z)=o(1)$\footnote{We may suppose that $R_j^\pm(z)=o(1)$ , as $z\to 0$ on $V_j^\pm$, since we have eliminated infinite part of the Fatou coordinate and we request the analytic Fatou coordinate to have asymptotic expansion in power-logarithmic scale, equal to the formal Fatou coordinate.}  on $V_j^\pm$ satisfying \eqref{eq:later}.

\noindent We show as in \cite{MRRZ2Fatou} that \begin{equation}\label{eq:oh}\delta_j^\pm(z)=O(z^{\alpha-1}\boldsymbol\ell^{m+2}),\  z\in V_j^{\pm},\end{equation} where $\delta_j^{\pm}(z):=1-\Psi_\infty^{j,\pm}(f(z))+\Psi_{\infty}^{j,\pm}(z)$ is analytic on $V_j^\pm$. Indeed, $\widehat R(z)$ is, by continuation of the formal algorithm from Proposition~\ref{prop:fffatou}, a power-logarithmic transseries, and infinitesimal, so its lowest-order monomial is $\boldsymbol\ell$ or bigger. Since $R_j^{\pm}(z)$ admit $\widehat R(z)$ as asymptotic expansion, as $z\to 0$, by Cauchy integral formula we conclude that $(R_j^{\pm})'(z)=O(\frac{\boldsymbol\ell^2}{z})$. Therefore, $R_j^{\pm}(f(z))-R_j^{\pm}(z)\sim (R_j^{\pm})'(z)\cdot (-z^\alpha\boldsymbol\ell^m)=O(z^{\alpha-1}\boldsymbol\ell^{m+2})$, as $z\to 0$ in $V_j^\pm$.

We prove now the existence of (unique) analytic $R_j^\pm=o(1)$ on $V_j^\pm$ satisfying \eqref{eq:later}. We prove it for the attracting petals $V_j^+$, $j\in\mathbb Z$. On repelling petals, we work with the inverse $f^{-1}$, since, by Abel equation, $\Psi_f:=-\Psi_{f^{-1}}$. Let us consider again the $w$-chart, as defined in \eqref{eq:wch} in part $(1)$. Using \eqref{eq:before}, we get the estimate:
$$                                                                                
|F^{\circ n}(w)-w-n|\leq 
cn(\log R)^{-1},
$$
so \begin{equation}\label{eq:FC}|F^{\circ n}(w)|\geq C_R\cdot n,\ C_R>0,\ w\in W_{R,\alpha_R}\cap \widetilde{\mathcal R}.\end{equation} Here, $\widetilde {\mathcal R}$ and $\alpha_R$ and $W_{R,\alpha_R}$, $R>R_0$, are as in $(1)$. Returning to the variable $z$, we get:
\begin{equation}\label{eq:haa}
|f^{\circ n}(z)|\leq D_R n^{-\frac{1}{\alpha-1}}\big|\boldsymbol\ell(F^{\circ n}(w))\big|^{-\frac{m}{\alpha-1}},\ w=-\frac{1}{a(\alpha-1)}z^{-\alpha+1}\boldsymbol\ell^{-m},\ D_R>0,
\end{equation}
on the sectors on $\mathcal R_{C}$ which are the preimages of $W_{R,\alpha_R}\cap \widetilde{\mathcal R}$ by \linebreak $w=-\frac{1}{a(\alpha-1)}z^{-\alpha+1}\boldsymbol\ell^{-m}$. They are subsectors whose unions form the petal $V_j^+$. Note that $\alpha_R$, $C_R$ and $D_R$ are uniform for all petals $V_j^\pm$, so \eqref{eq:haa} holds uniformly by levels $j\in\mathbb Z$ of the Riemann surface, on subsectors of the same opening of petals $V_j^+$.

\noindent Under the assumption that $R_j^+(z)$ have an asymptotic behavior as $z\to 0$, it holds that $R_j^+(z)=o(1)$, as $z\to 0$ in $V_j^+$, up to an additive constant. Iterating \eqref{eq:later}, since $R_j^+(z)=o(1)$, we get that $R_j^+$ is necessarily of the form:
\begin{equation}\label{eq:eq}
R_j^+(z)=-\sum_{i=1}^{\infty}\delta (f^{\circ i}(z)),\ z\in V_j^+,
\end{equation}
where $\delta(z)=O(z^{\alpha-1}\boldsymbol\ell^{m+2})$, uniformly on all petals. By \eqref{eq:haa}, the series converges uniformly on each subsector which is the preimage of $W_{R,\alpha_R}\cap \widetilde{\mathcal R}$. Indeed, we conclude by \eqref{eq:oh} and \eqref{eq:haa} as follows:
\begin{align*}
|\delta(f^{\circ i}(z))|\leq C_R|f^{\circ i}(z)|^{\alpha-1}\big|\boldsymbol\ell(f^{\circ i}(z))\big|^{m+2}\leq D_R i^{-1} |\boldsymbol\ell(F^{\circ i}(w))\big|^{-m} \big|\boldsymbol\ell(f^{\circ i}(z))\big|^{m+2}.
\end{align*}
Since $f^{\circ i}(z)=h^{-1}\circ F^{\circ i}(w)$, where $h(z):=-\frac{1}{a(\alpha-1)}z^{-\alpha+1}\boldsymbol\ell^{-m}$, there exists $c>0$ such that $\big|\boldsymbol\ell(f^{\circ i}(z))\big|^{m+2}\leq c\,|\boldsymbol\ell(F^{\circ i}(w))|^{m+2}$. By \eqref{eq:FC}, there exist $d_R>0,\, d_R'>0$ such that:
$$
|\delta(f^{\circ i}(z))|\leq d_R i^{-1}|\boldsymbol\ell(F^{\circ i}(w))|^2\leq d_R' i^{-1}(-\boldsymbol\ell (i))^2,\ i\geq 2.
$$
By Weierstrass theorem, the series \eqref{eq:eq} converges to an analytic function $R_j^+$ on $V_j^+$, $j\in\mathbb Z$. This defines the analytic Fatou coordinate on $V_j^+$, $\Psi_j^+:=\Psi_\infty^{j,\pm}+R_j^+$.

\smallskip

It is left to prove that the analytic Fatou coordinate constructed on each petal as the limit of the uniformly convergent series \eqref{eq:eq} admits the integral sectional asymptotic expansion which is, up to an additive constant, equal to the formal Fatou coordinate $\widehat\Psi\in\widehat{\mathcal L}_2^\infty(\mathbb R)$. We follow the procedure similar as in \cite{loray2}, in the same way as in \cite[Subsection 6.1.4]{MRRZ2Fatou} for the Fatou coordinate of a parabolic Dulac germ on the real line. We continue to subtract consecutively blocks from the Fatou coordinate (analytic counterparts of formal blocks, as for $\widehat \Psi_i$ at the beginning of the proof). We prove the asymptotics after subtraction of each block, using the \emph{modified Abel equation} after subtraction:
$$
h(f(z))-h(z)=\delta(z),
$$
$\delta(z)=O(z^\gamma)$, where $\gamma$ tends to $+\infty$ as we subtract blocks. Therefore, we use the following auxiliary proposition, whose proof is in the Appendix.

\begin{prop}[Generalization to complex domain of Proposition~6.2 in \cite{MRRZ2Fatou}, see also \cite{loray2}]\label{prop:auxil}
Let $f$ be a parabolic generalized Dulac germ on $\mathcal R_C$ such that $f(z)=z-az^{\alpha}\boldsymbol\ell^m+o(z^{\alpha}\boldsymbol\ell^m),\ a>0,\ \alpha>1,\ m\in\mathbb Z$. Let $\delta$ be a holomorphic germ\footnote{obtained by the blockwise subtractions in the Abel equation described above.} on an attracting petal $V_j^+,\ j\in\mathbb Z$, such that $\delta(z)=O(z^\gamma\boldsymbol\ell ^r)$, $\gamma>0$, $r\in\mathbb Z$, $(\gamma,r)\succeq (\alpha-1,m+2)$. Then the series
\begin{equation}\label{eq:new}
h_j(z):=-\sum_{k=0}^{\infty} \delta\big(f^{\circ k}(z)\big),
\end{equation}
converges uniformly on each compact subsector $W$ of the attracting petal $V_j^+$, and thus defines a holomorphic function $h_j$ on $V_j^+$. Moreover, for every small $\delta>0$, such that $\gamma-(\alpha-1)-\delta>0$, there exists $C_W>0$, such that it holds:
\begin{equation}\label{eq:old}
|h_j(z)|\leq \begin{cases}C_W\boldsymbol\ell(|z|)^{r-m-1}, &\gamma=\alpha-1,\, r\geq m+2,\\
C_W |z|^{\gamma-(\alpha-1)-\delta}, &\gamma>\alpha-1,
\end{cases}
\end{equation}
on every subsector $W$ of the attracting petal $V_j^{+}$.
\end{prop}

\noindent We now apply Proposition~\ref{prop:auxil} on the right-hand side $\delta(z)$ of the Abel equation and use \eqref{eq:old} to prove the asymptotic expansion. Moreover, each block in the Fatou coordinate evidently has a Poincar\' e expansion correspondent to the associated formal block.
\medskip

Finally, the proof of the uniqueness up to a constant of the analytic Fatou coordinate on each petal admitting the integral sectional asymptotic expansion $\widehat\Psi\in\widehat{\mathcal L}_2^\infty(\mathbb R)$ is given in Proposition~\ref{prop:Fatouexp}.

\medskip

\noindent \emph{Proof of statement (3): the sectorial conjugacies and their asymptotic expansion.} Proven in Proposition~\ref{prop:gi}, up to the first change of variables given by the homotecy $\varphi_0(z)=a^{-\frac{1}{\alpha-1}}z$.
\end{proof}

\bigskip

\section{Proof of Theorem B: moduli of analytic classification for parabolic generalized Dulac germs}\label{sec:proofB}

For simplicity, we restrict to parabolic generalized Dulac germs with $a=1,\ \alpha=2$:
$$
f(z)=z-z^2\boldsymbol\ell^m+o(z^2\boldsymbol\ell^m),\ z\in\mathcal R_C,\ m\in\mathbb Z.
$$

This is a small computational simplification, since there is only one petal (attracting or repelling) on each level of the Riemann surface $\mathcal R_C$. 
It is additionally justified by the following proposition:
\begin{prop}\label{prop:dva} Two parabolic generalized Dulac germs \begin{align*}f(z)=z-z^\alpha\boldsymbol\ell^m+o(z^\alpha\boldsymbol\ell^m)\text{ and } g(z)=z-z^\alpha\boldsymbol\ell^m+o(z^\alpha\boldsymbol\ell^m)\end{align*} are analytically conjugated on a standard quadratic domain if and only if parabolic generalized Dulac germs $$\tilde f=h_\alpha^{-1}\circ f\circ h_\alpha=z-z^2\boldsymbol\ell^m +o(z^2\boldsymbol\ell^m)\text{ and } \ \tilde g=h_\alpha^{-1}\circ g\circ h_\alpha=z-z^2\boldsymbol\ell^m+o(z^2\boldsymbol\ell^m)$$ are analytically conjugated on a standard quadratic domain, where $h_\alpha(z):=(\alpha-1)^{-\frac{m}{\alpha-1}} z^{\frac{1}{\alpha-1}}$.
\end{prop}
\begin{proof}
Let $h_\alpha(z):=(\alpha-1)^{\frac{-m}{\alpha-1}} z^{\frac{1}{\alpha-1}}$. It is an analytic bijection from a standard quadratic domain to a standard quadratic domain. It is easy to check that $\tilde f$ and $\tilde g$ keep the generalized Dulac expansion property, that is, they are again parabolic generalized Dulac germs. To a conjugation $\varphi$ between $f$ and $g$ there corresponds the conjugation $\tilde\varphi=h_\alpha^{-1}\circ \varphi\circ h_\alpha$ between $\tilde f$ and $\tilde g$. It is obviously parabolic. Moreover, $f$ and $g$ are $\widehat{\mathcal L}(\mathbb R)$-formally conjugated if and only if $\tilde f$ and $\tilde g$ are.  
\end{proof}
\medskip

\subsection{The construction of moduli (proof of Theorem~B)}\ 
\smallskip 

\noindent Let $f(z)=z-z^2\boldsymbol\ell^m+o(z^2\boldsymbol\ell^m)$ be a parabolic generalized Dulac germ on a standard quadratic domain $\mathcal R_C$. 
\medskip

\noindent Let $V_j^{-}$ denote the maximal repelling petals for the dynamics of $f$ along $\mathcal R_C$. They are constructed in Theorem A $(1)$. They are of opening $\big(2(j-1)\pi,2 j\pi\big)$, $j\in\mathbb Z$. Note that, by the proof of Theorem A (1), the petals are of the same \emph{shape} along the standard quadratic domain (due to uniform estimates for $\alpha_R$), only getting \emph{smaller} as dictated by the decrease of radius of the standard quadratic domain. Let $r_j^{0,-}$ and $r_j^{\infty,-}$ denote the \emph{radii} of the petal $V_j^-$ measured along directions $2j\pi-\frac{3\pi}{2}$ and $2j\pi-\frac{\pi}{2}$ respectively, $j\in\mathbb Z$. Similarly, let $V_j^+$ denote the maximal attracting petals of opening $\big((2j-1)\pi,(2j+1)\pi\big)$. Let $r_j^{\infty,+}$ and $r_j^{0,+}$ denote the radii of the petal $V_j^+$ measured along directions $2j\pi-\frac{\pi}{2}$ and $2j\pi+\frac{\pi}{2}$ respectively, $j\in\mathbb Z$. We now define: $$r_{0}^{j+1}:=\min\{r_{j+1}^{0,-},r_j^{0,+}\},\ r_\infty^j:=\{r_j^{\infty,-},r_j^{\infty,+}\},\ j\in\mathbb Z.$$ 
Note that radii $r_0^{j+1},\, r_\infty^j$, $j\in\mathbb Z$, correspond to radii of intersections of consecutive petals $V_0^{j+1}$ and $V_\infty^j$ (defined in \eqref{eq:centr} below) at central directions $2 j \pi +\frac{\pi}{2}$ and $2j \pi -\frac{\pi}{2}$. 

Since a parabolic generalized Dulac germ is defined at least on a standard quadratic domain, there exist $C>0,\ K>0$ such that it holds: \begin{equation}\label{eq:radii}r_0^{j},\,r_\infty^j \geq K e^{-C\sqrt {|j|}},\ j\in\mathbb Z.\end{equation} That is, the rate of decrease of radii $r_0^j$ and $r_\infty^j$ is bounded from below by a rate of decrease of radii of a standard quadratic domain at directions $2 j \pi +\frac{\pi}{2}$ and $2j \pi -\frac{\pi}{2}$, as $|j|\to\infty$, see \eqref{eq:sie}. That means that the radii of petals do not drop too quickly to zero along levels of $\mathcal R_C$.
\medskip

\noindent Let $\widehat\Psi\in\widehat{\mathcal L}_2^\infty(\mathbb R)$ be the formal Fatou coordinate of the generalized Dulac asymptotic expansion $\widehat f$ of $f$. Let 
$
\Psi_j^\pm,\ j\in\mathbb Z,
$
be analytic sectorial Fatou coordinates admitting the integral sectional asymptotic expansion $\widehat \Psi$ on respective petals $V_{j}^\pm$, $j\in\mathbb Z$, as $z\to 0$, as in Theorem~A (2). They are unique up to additive constants.
\smallskip

Denote by:
\begin{align}\label{eq:centr}
&V_{0}^{j+1}:=V_{j+1}^{-}\cap V_j^+,\nonumber\\
&V_{\infty}^j:=V_j^{-}\cap V_j^{+},\ j\in\mathbb Z,
\end{align}
the intersections of attracting and repelling petals, of opening $\pi$. The petals $V_0^j$ and $V_{\infty}^j$ satisfy a lower bound \eqref{eq:radii} on their radii $r_0^j$, $r_\infty^j$. 
\medskip

\begin{figure}[h!]
\vspace{-0.3cm}
\includegraphics[scale=0.2]{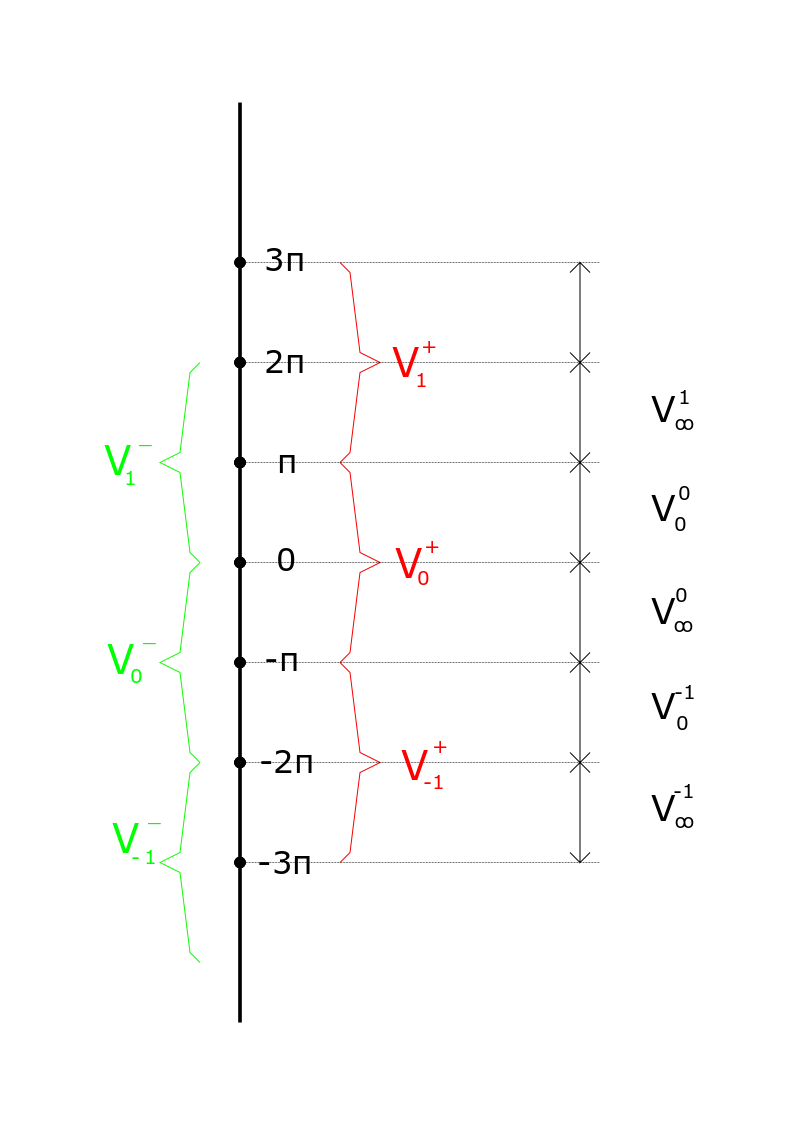}
\vspace{-0.5cm}
\caption{The position of petals along a standard quadratic domain on the Riemann surface of the logarithm.}
\end{figure}

The construction of \emph{horn maps} now mimics the construction of Voronin \cite{ecalle, voronin} for parabolic analytic germs. 

\noindent At each intersection of petals, the difference $\Psi_+^*-\Psi_-^*$ of appropriate realizations of the Fatou coordinate is constant along the closed trajectories of $f$. We represent the space of trajectories of each petal by a doubly punctured Riemann sphere, using the composition of the Fatou coordinate on the petal and of the exponential mapping. Each trajectory corresponds to one point of the sphere. The closed trajectories correspond to points in punctured neighborhoods of poles $0$ and $\infty$. The germs $(h_{0}^j,\ \tilde{h}_{\infty}^j)_{j\in\mathbb Z}$ below, defined at punctured neighborhoods of $0$ i.e. $\infty$, represent the \emph{horn maps} of $f$. They map correspondent closed orbits on neighboring spheres, as dictated by the dynamics of $f$ on the intersection of petals where we have closed trajectories:
\begin{align}\label{eq:moduli}
&h_{0}^j(t):=e^{-2\pi i \Psi_+^{j-1}\circ (\Psi_-^j)^{-1}\big(-\frac{\ln t}{2\pi i}\big)},\ t\approx 0,\\
&\tilde h_{\infty}^j(t):=e^{-2\pi i \Psi_-^{j}\circ (\Psi_+^j)^{-1}\big(-\frac{\ln t}{2\pi i}\big)},\ t\approx \infty,\nonumber \\ 
&h_{\infty}^j(t):=e^{2\pi i \Psi_-^{j}\circ (\Psi_+^j)^{-1}\big(\frac{\ln t}{2\pi i}\big)},\ t\approx 0, \quad
j\in\mathbb Z.\nonumber
\end{align}
By construction, $h_0^j$ and $h_\infty^j$, $j\in\mathbb Z$, are analytic germs of diffeomorphisms defined on punctured neighborhoods of $0$ and tending to $0$, as $t\to 0$. Therefore, they can be analytically extended to $0$ by Riemann's theorem on removable singularities.

\smallskip
In this way we get an \emph{infinite necklace} of Riemann spheres, indexed by $\mathbb Z$, connected by analytic diffeomorphisms at their poles. Each sphere is connected at one pole with the preceding sphere, and at the other pole with the following sphere:

\begin{figure}[h!]
\includegraphics[scale=0.2]{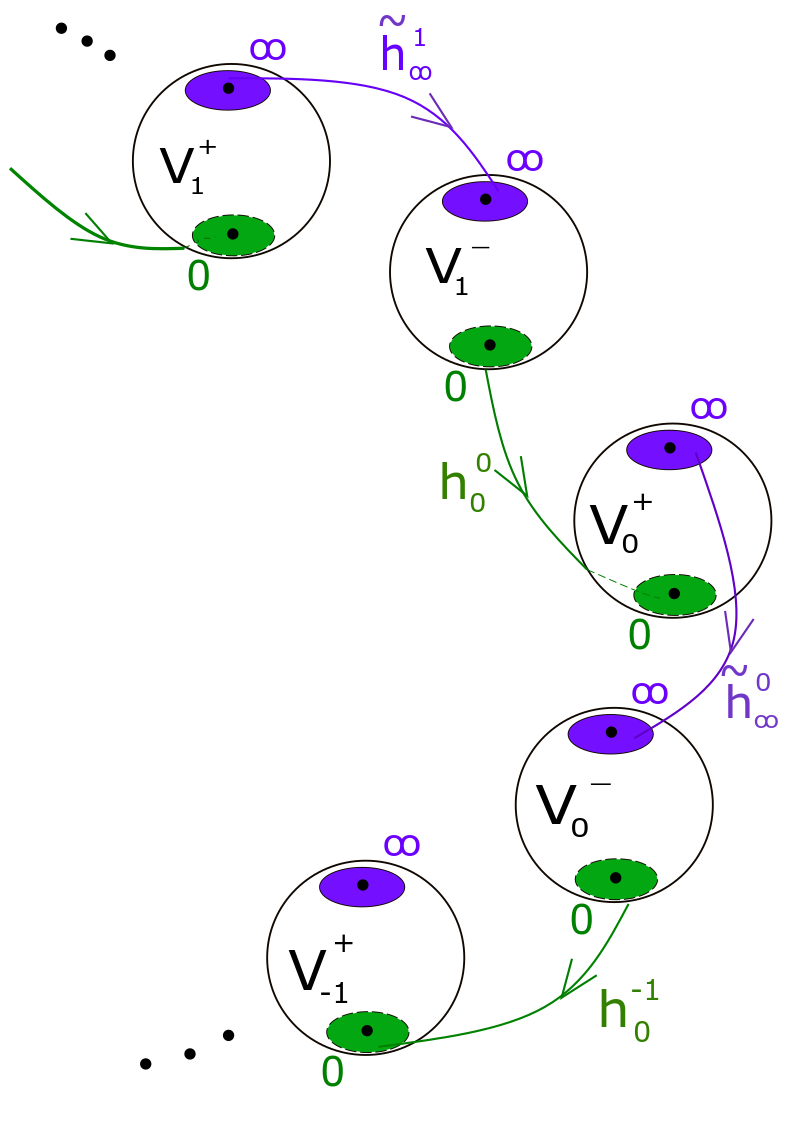}
\caption{The necklace of spheres and the \emph{horn maps} of a parabolic generalized Dulac germ.}
\end{figure}
\medskip

The maximal radii of convergence of $h_{0}^j$ and $h_\infty^j$, $j\in\mathbb Z$, are $R_j$ such that:
\begin{equation}\label{eq:gg}
|t|<R_j\ \Leftrightarrow \Big|(\Psi_\pm^j)^{-1}\big(\pm\frac{\ln t}{2\pi i}\big)\Big|<r_j,
\end{equation}
where $r_j$ is as in \eqref{eq:radii} (the rate of decrease follows a standard quadratic domain).
Since $\Psi_\pm^j(z)\sim -\frac{1}{z}\boldsymbol\ell^{-m}$, \emph{uniformly\footnote{Due to the uniform bound \eqref{eq:uniest} for $f$, see \cite[Lemma 5.2]{drugi}.} in $j\in\mathbb Z$}, as $z\to 0$, we get that, uniformly in $j\in\mathbb Z$:
$$
\Big|(\Psi_\pm^j)^{-1}\big(\pm\frac{\ln t}{2\pi i}\big)\Big|\sim \frac{1}{-\log(|t|)\log^m (-\log|t|)},\ t\to 0.
$$
From \eqref{eq:gg}, we get that $R_j$ are such that the set
$$
\Big\{R_je^{\frac{C}{r_j(-\log r_j)^m}}:\ j\in\mathbb Z\Big\}
$$ is bounded from above and from below by positive constants.
The bound \eqref{eq:radii} on the speed of decrease of $r_j$ now gives the bound for the speed of decrease of $R_j$, $j\to\infty$, as the following. There exist constants $C,\ K,\ K_1>0$ such that:
\begin{equation*}
R_j\geq K_1 e^{-Ke^{C\sqrt{|j|}}(\sqrt{|j|})^m},\ j\in\mathbb Z.
\end{equation*}
Equivalently, there exist (some other) constants $C$ and $K,\ K_1$ such that:
\begin{equation*}\label{eq:velrad}
R_j\geq K_1 e^{-Ke^{C\sqrt{|j|}}},\ j\in\mathbb Z.
\end{equation*}
\smallskip

Let us now justify identifications \eqref{eq:ekvij} in Theorem~B. With this equivalence relation, the horn maps become equivalence classes. Indeed, it is easy to see that the possible change of additive constants in Fatou coordinates of a parabolic generalized Dulac germ $f$ on petals results in conjugation of its horn maps as described in \eqref{eq:ekvij}. Alternatively, this can be considered as reparametrizations of doubly punctured spheres fixing the poles (M\" obius transforms fixing the poles are just homothecies). Moreover, we can always restrict a parabolic generalized Dulac \emph{germ} $f$ to an inscribed standard quadratic domain, resulting in a change in asymptotics of radii of convergence described in \eqref{eq:ekvij}. The type of the asymptotics remains the same.
\bigskip

\noindent \emph{Proof of Theorem~B}.
Suppose $f$ and $g$ are two analytically conjugated normalized parabolic generalized Dulac germs on some standard quadratic domain $\mathcal R_C$. Then there exists a tangent to the identity \emph{global} conjugacy $\varphi(z)=z+o(z)$, admitting the formal conjugacy $\widehat\varphi\in\widehat{\mathcal L}(\mathbb R)$ as its generalized block iterated integral asymptotic expansion, see Proposition~\ref{prop:gic}. Since $f$ and $g$ belong to the same $\widehat{\mathcal L}(\mathbb R)$-formal class, we can identify their petals. Let $\widehat\Psi_f,\ \widehat\Psi_g\in\widehat{\mathcal L}_2^\infty(\mathbb R)$ be the formal Fatou coordinates for $f$ and $g$ respectively, which exist and are unique (up to an additive constant) by Proposition~\ref{prop:fffatou}. By Theorem~A $(2)$, denote by $^f \Psi_{\pm}^j$ the Fatou coordinates of $f$, analytic on petals $V_j^{\pm}$, and admitting the formal Fatou coordinate $\widehat\Psi_f$ as their integral asymptotic expansion, as $z\to 0$ in $V_j^{\pm}$. Then, by Abel equation,
$
^f \Psi_{\pm}^j\circ \varphi
$
are Fatou coordinates for $g$, analytic on respective petals $V_{\pm}^j$, which admit formal Fatou coordinate $\widehat{\Psi}_g$ as their integral asymptotic expansions. By Theorem~A (2), such Fatou coordinates of $g$ are unique on each petal, up to a constant. Therefore, there exists a choice of Fatou coordinates for $f$ and $g$ on petals (that is, of additive constants), such that: $$(^g \Psi_+^{j-1})\circ (^g \Psi_-^j)^{-1}=(^f \Psi_+^{j-1})\circ \varphi\circ\varphi^{-1}\circ (^f \Psi_-^i)^{-1}=(^f \Psi_+^{i-1})\circ (^f \Psi_-^i)^{-1},\ j\in\mathbb Z,$$ and 
$$(^g \Psi_-^{j})\circ (^g \Psi_+^j)^{-1}=(^f \Psi_-^{j})\circ \varphi\circ\varphi^{-1}\circ (^f \Psi_+^j)^{-1}=(^f \Psi_-^{j})\circ (^f \Psi_+^j)^{-1},\ j\in\mathbb Z,$$ on the respective images of intersections $V_{0}^j$ and $V_{\infty}^j,\ j\in\mathbb Z$. Thus, the horn maps given by \eqref{eq:moduli} are equal.

\smallskip
Now suppose that $f$ and $g$ have the same (up to identifications) horn maps. We can take common petals, since $f$ and $g$ belong to the same formal class. By Theorem~A, on respective petals $V_{\pm}^j$, there exist analytic Fatou coordinates $^{f,g}\Psi_{\pm}^j$ for $f$ and $g$, admitting the formal Fatou coordinates as their integral asymptotic expansions. On each petal, let us define an analytic conjugacy function
$$
\varphi_{\pm}^j(z):=(^f\Psi_{\pm}^j)^{-1}\circ (^g\Psi_{\pm}^j)(z),\ z\in V_{\pm}^j.
$$
We show that there exists a choice of constants in Fatou coordinates for $f$ and $g$ on each petal such that $\varphi_{\pm}^j$ \emph{glue} analytically on overlappings $V_{0}^j,\ V_\infty^j$ of consecutive repelling and attracting petals to a global analytic conjugacy germ $\varphi$ on a standard quadratic domain. But this is exactly ensured by the equality of horn maps.
\hfill $\Box$
\smallskip

\begin{prop}[Symmetry of horn maps]\label{prop:sim} Let $f(z)=z-z^2\boldsymbol\ell^m+o(z^2\boldsymbol\ell^m)$ be a parabolic generalized Dulac germ on a standard quadratic domain, and let $(h_0^j,h_{\infty}^j)_{j\in\mathbb Z}$ be its sequence of horn maps. Then, up to identifications \eqref{eq:ekvij}, it holds that
\begin{equation}\label{eq:sim}
\big(h_{0}^{-j+1}\big)^{-1}(t)\equiv \overline{h_{\infty}^j(\,\overline t\,)},\ j\in\mathbb Z.
\end{equation}
That is, the necklace of spheres is \emph{symmetric} with respect to the real axis. The horn maps with negative indices are symmetric in the sense \eqref{eq:sim} to the horn maps with positive indices.
\end{prop}
\begin{proof} Let $f$ be a parabolic normalized generalized Dulac germ on a standard quadratic domain $\mathcal R_C$. It holds that $f\big(\{\mathrm{arg}(z)=0\}\cap \mathcal R_C\big)\subseteq \{\mathrm{arg}(z)=0\}\cap\mathcal R_C$ and $f$ is holomorphic on $\mathcal R_C$. Then, by Schwarz reflection principle, it holds that $f(z)=\overline {f(\overline z)}$ on whole $\mathcal R_C$. It is then easy to see, by the analyticity of the Fatou coordinate and by the Abel equation, that:
\begin{align}\label{eq:hha}
&\Psi_-^{-j+1}=K\circ \Psi_-^j\circ K,\ z\in V_-^{-j+1},\ j\in\mathbb N,\nonumber\\
&\Psi_+^{-j}=K\circ \Psi_+^j\circ K,\ z\in V_+^{-j},\ j\in\mathbb N,\nonumber\\
&\Psi_+^0\Big|_{V_{\infty}^{0}}=K\circ \Psi_+^0\Big|_{V_{0}^1}\circ K.
\end{align}
Here, $K(z)=\overline z$ is the complex conjugation on the Riemann surface of the logarithm. By \eqref{eq:moduli} and \eqref{eq:hha}, \eqref{eq:sim} follows.
More precisely, by \eqref{eq:moduli} and \eqref{eq:hha}, 
\begin{align*}
(h_0^{-j+1})^{-1}(t)&=e^{-2\pi i \Psi_-^{-j+1}\circ (\Psi_+^{-j})^{-1}(-\frac{\log t}{2\pi i})},\\
&=e^{-2\pi i\,K\circ \Psi_-^{j}\circ (\Psi_+^{j})^{-1}\circ K(-\frac{\log t}{2\pi i})},\ \ t\in(\mathbb C,0).
\end{align*}
Since $K\big(\frac{-\log t}{2\pi i}\big)=\frac{\log \overline t}{2\pi i}$, $t\in\mathbb (\mathbb C,0)$, and $K(e^{-2\pi i\overline z})=e^{2\pi i z},\ z\in\mathbb C$, we get:
$$
\overline{(h_0^{-j+1})^{-1}(t)}=h_\infty^j(\overline t),\ t\in (\mathbb C,0).
$$
\end{proof}

\section{Appendix}\label{sec:app}

We give here a generalization of the definition of \emph{sectional asymptotic expansions} on $\mathbb R_+$ (already defined in \cite{MRRZ2Fatou}) to complex sectors:
\begin{defi}[Sectional asymptotic expansions in $\mathbb C$]\label{def:assy} We say that a germ $f$ analytic on a petal $P$ admits on $P$ the sectional asymptotic expansion $$\widehat f(z)=\sum_{i=1}^{\infty}\widehat f_i(\boldsymbol\ell)z^{\alpha_i}\in\widehat {\mathcal L}_n^\infty,$$ where $(\alpha_i)_i,\ \alpha_i\in\mathbb R,$ is a strictly increasing sequence, either finite or tending to $+\infty$, if there exist germs $f_i$ analytic on $\boldsymbol\ell$-cusps $\boldsymbol\ell(P)$ admitting $\widehat f_i\in\widehat{\mathcal L}_{n-1}^\infty$ as their sectional asymptotic expansions on $\boldsymbol\ell(P)$\footnote{In the following sense: for every proper $\boldsymbol\ell$-subcusp $\boldsymbol\ell(V)\subset \boldsymbol\ell(P)$ it holds that $f_i(\boldsymbol\ell)$ admits $\widehat f_i(\boldsymbol\ell)$ as sectional asymptotic expansion, as $\boldsymbol\ell\to 0$.}, such that, on every proper subsector\footnote{By \cite{loray2}, we say that a complex germ admits a power asymptotic expansion on an open domain $D$ with zero in its boundary if this asymptotic expansion holds on every proper subsector $V\subset D$ with vertex at the origin (possibly with different speeds of convergence of remainders, depending on the sector).} $V\subset P$ it holds that
\begin{equation}\label{eq:haj}
f(z)-\sum_{i=1}^{n}f_i(\boldsymbol\ell)z^{\alpha_i}=o(z^{\alpha_{n+1}-\delta}),\ \delta>0,\ n\in\mathbb N,\ z\to 0,\ z\in V.
\end{equation}
\end{defi}
The bounding constants in \eqref{eq:haj} may vary with sectors $V\subset P$, but the germs $f_i$ and the sequence $(\alpha_i)$ are the same for all sectors. 
\smallskip

We have shown in \cite{MRRZ2Fatou} that, for a given choice of section, the sectional asymptotic expansion of a germ is unique.
\bigskip

\noindent \emph{Proof of Proposition~\ref{prop:clomult}}. Let $\widehat f(\boldsymbol\ell)=\sum_{k=0}^{\infty} a_k \boldsymbol\ell^k$ and $\widehat g(\boldsymbol\ell)=\sum_{k=0}^{\infty} b_k \boldsymbol\ell^k$ be the $\log$-Gevrey asymptotic expansions of $f$ and $g$ respectively.
First, by classical theory of asymptotic expansions, it follows that the formal product $\widehat f(\boldsymbol\ell)\cdot \widehat g(\boldsymbol\ell)$ is the power asymptotic expansion of $f\cdot g$ on $S$. Let $(c_k)_k$ be the coefficients of the power series $\widehat f(\boldsymbol\ell)\widehat g(\boldsymbol\ell)$. The sequence $(c_k)_k$ is then the convolution of sequences of coefficients $(a_k)_k$ and $(b_k)_k$. The power asymptotic expansion of $f\cdot g$ on $\boldsymbol\ell$-cusp $S$ is obviously unique. It is left to prove the $\log$-Gevrey bounds of order $r$ for $f\cdot g$ on $S$, for every $0<r<s$. 

Let $S'\subset S$ be a $\boldsymbol\ell$-subcusp of $S$. By direct multiplication of series as in \cite[Theorem 1]{bahlser}, we get:
\begin{align}\label{eq:jaojao}
|f(\boldsymbol\ell)g(\boldsymbol\ell)-\sum_{k=0}^{N-1}c_k\boldsymbol\ell^k|&\leq |f(\boldsymbol\ell)|\cdot\Big|g(\boldsymbol\ell)-\sum_{k=0}^{N-1}c_k\boldsymbol\ell^k\Big|+\nonumber\\
&+\sum_{k=0}^{N-1}\Big(|b_k||\boldsymbol\ell|^k\cdot \big|f(\boldsymbol\ell)-\sum_{i=0}^{N-1-k} a_i \boldsymbol\ell^i\big|\Big), \ \boldsymbol\ell\in S',
\end{align}
for every $N\in\mathbb N$. Since $f$ and $g$ are $\log$-Gevrey of order respectively $m,\,n$, there exists a constant $C>0$ such that:
\begin{align}\label{eq:jjjoj}
|b_k||\boldsymbol\ell|^k&\leq \Big|g(\boldsymbol\ell)-\sum_{j=0}^{k}b_j\boldsymbol\ell^j\Big|+\Big|g(\boldsymbol\ell)-\sum_{j=0}^{k-1}b_j\boldsymbol\ell^j\Big|\nonumber\\
&\leq C n^{-k-1} \log(k+1)^{k+1} e^{-\frac{k+1}{\log(k+1)}}|\boldsymbol\ell|^{k+1}+ C n^{-k} \log^k k e^{-\frac{k}{\log k}}|\boldsymbol\ell|^{k},\nonumber\\
&\leq C n^{-k} \log^k k \cdot e^{-\frac{k}{\log k}}(|\boldsymbol\ell|^k+|\boldsymbol\ell|^{k+1}),\ \boldsymbol\ell\in S',\ k\in\mathbb N,\ k\geq 2.
\end{align}
To derive the last line, we use the fact that the function $x\mapsto n^{-x} \log^{x}x \cdot e^{-\frac{x}{\log x}}$ is strictly increasing on $[2,+\infty)$. Moreover,
\begin{equation*}
\lim_{k\to\infty}\frac{n^{-k-1} \log(k+1)^{k+1} e^{-\frac{k+1}{\log(k+1)}}}{n^{-k} \log^k k \cdot e^{-\frac{k}{\log k}}}\frac{1}{\log \sqrt k}=1.
\end{equation*}
For $0<p<n$, it holds that $(n/p)^{-k}\log \sqrt k\to 0$, as $k\to\infty$. Now dividing \eqref{eq:jjjoj} by $|\boldsymbol\ell|^k$ and passing to the limit as $|\boldsymbol\ell|\to 0$, we get that there exists a constant $C>0$ such that, for every $0<p<n$, it holds that:
\begin{equation}\label{eq:izvod}
|b_k|\leq C p^{-k} \log^k k\cdot e^{-\frac{k}{\log k}},\ k\in\mathbb N,\ k\geq 2.
\end{equation}
Using the fact that $f\in LG_m(S)$ and $g\in LG_n(S)$ and \eqref{eq:izvod}, from \eqref{eq:jaojao} we get that there exist constants $A>0,\ B>0,\ D>0, \ E>0$ such that:
\begin{align*}
&|f(\boldsymbol\ell)g(\boldsymbol\ell)-\sum_{k=0}^{N-1}c_k\boldsymbol\ell^k|\leq B n^{-N}\log^N N\cdot e^{-\frac{N}{\log N}}|\boldsymbol\ell|^N+\\
&\quad +AC\sum_{k=2}^{N-2}p^{-k}\log^k k\cdot e ^{-\frac{k}{\log k}}|\boldsymbol\ell|^k\cdot m^{-(N-k)}\log^{N-k}(N-k)\cdot e ^{-\frac{N-k}{\log(N-k)}}|\boldsymbol\ell|^{N-k}\\
&\quad \leq D s^{-N} \log^N N \cdot e^{-\frac{N}{\log N}}\cdot N\cdot |\boldsymbol\ell|^N\leq E r^{-N} \log^N N \cdot e^{-\frac{N}{\log N}}\cdot |\boldsymbol\ell|^N,\\
&\qquad\qquad\qquad\qquad\qquad\qquad\qquad\qquad \qquad\qquad\qquad\qquad\boldsymbol\ell\in S',\ N\in\mathbb N,\ N\geq 2.
\end{align*}
Here, we put $s:=\min\{m,p\}$ and use the fact that $\log k,\ \log(N-k)\leq \log N$ and $e^{-\frac{1}{\log k}},\ e^{-\frac{1}{\log (N-k)}}\leq e^{-\frac{1}{\log N}}$ for $2\leq k\leq N-2$. In the last line we take any $0<r<s$.
\hfill $\Box$
\bigskip

\noindent \emph{Proof of Proposition~\ref{prop:clodif}}. 
Let $\widehat f(\boldsymbol\ell)=\sum_{k=0}^{\infty}a_k\boldsymbol\ell^k$. We prove inductively, by showing that the first derivative is $\log$-Gevrey of every order $0<r<m$ on $S$, and then continuing similarly (using the bounds for the first derivative) for the second derivative etc. Since $f$ admits $\widehat f$ as $\log$-Gevrey asymptotic expansion of order $m>0$, for every $\boldsymbol\ell$-subcusp $S'=\boldsymbol\ell(V')\subset S$, $V'$ a proper subsector of $V$, there exists a constant $C>0$ such that:
\begin{equation}\label{eq:ne}
\Big|f(\boldsymbol\ell)-\sum_{k=0}^{n-1}a_k\boldsymbol\ell^k\Big|\leq C m^{-n} (\log n)^n e^{-\frac{n}{\log n}}|\boldsymbol\ell|^n, \ \boldsymbol\ell\in S',\ n\in\mathbb N,\ n\geq 2. 
\end{equation}
Take a small $\delta>0$ and a proper $\boldsymbol\ell$-subcusp $S_\delta\subset S'$ such that $S_\delta=\boldsymbol\ell(V_\delta)$, where $V_\delta\subset V'$ is obtained from $V'$ by shrinking by the angle $\arcsin(\delta)$ from both sides. We show that there exists a constant $D>0$ and a function $\delta\mapsto e(\delta)>0$ such that:
\begin{equation}\label{eq:nee}
\Big|\frac{d}{d\boldsymbol\ell}f(\boldsymbol\ell)-\sum_{k=1}^{n-1}ka_k\boldsymbol\ell^{k-1}\Big|\leq D \big(m-e(\delta)\big)^{-n} (\log n)^n e^{-\frac{n}{\log n}}|\boldsymbol\ell|^{n-1},\ \boldsymbol\ell\in S_\delta,\ n\in\mathbb N,\ n\geq 2.
\end{equation}
We also prove that $e(\delta)\searrow 0$, as $\delta\to 0$. Since $S_\delta\nearrow S'$, as $\delta\to 0$, this proves that $\frac{d}{d\boldsymbol\ell}f(\boldsymbol\ell)$ admits the formal derivative $\frac{d}{d\boldsymbol\ell}\widehat f(\boldsymbol\ell)$ as $\log$-Gevrey asymptotic expansion of every order $0<r<m$ on whole $S$.

We now prove \eqref{eq:nee}. We define: 
$$
\tilde H_n(z):=f(\frac{1}{-\log z})-\sum_{k=0}^{n-1} a_k \big(-\frac{1}{\log z}\big)^k,\ z\in V',\ n\in\mathbb N,
$$
where $z=e^{-1/\boldsymbol\ell}$. By the Cauchy integral formula in the variable $z$ and using bound \eqref{eq:ne}, we get:
\begin{align}\label{eq:a}
\Big|\frac{d}{dz}\tilde H_n(z)\Big|&\leq \int_{K(z,\delta|z|)}\frac{|\tilde H_n(\xi)|}{|z-\xi|^2}d|\xi|\\
&\leq Cm^{-n}(\log n)^n e^{-\frac{n}{\log n}}\frac{1}{\delta^2|z|^2}\sup_{|w|=\delta |z|}\Big|\frac{1}{-\log(z+w)}\Big|^n 2\pi \delta |z|\nonumber\\
&\leq D m^{-n}(\log n)^n e^{-\frac{n}{\log n}}\Big|\frac{1}{-\log z}\Big|^n\sup_{|w|=\delta |z|}\Big |1+\frac{\log(1+\frac{w}{z})}{\log z}\Big|^{-n} \frac{1}{|z|}\nonumber\\
&\leq D \big(m(1-e(\delta))\big)^{-n}(\log n)^n e^{-\frac{n}{\log n}}|\boldsymbol\ell|^n \frac{1}{|z|},\ z\in V_\delta,\ n\in\mathbb N,\ n\geq 2. \nonumber
\end{align}
Here, $e(\delta):=c\delta$, $\delta>0$, where $c>0$ depends only on the radius of $V$. Obviously, $e(\delta)\to 0$, as $\delta\to 0$. The constants do not depend on $n\in\mathbb N$. 

Indeed, by the triangle inequality,
\begin{align*}
\Big|1+\frac{\log(1+\frac{w}{z})}{\log z}\Big|^{-n}\leq \Big(1-\frac{|\log(1+\frac{w}{z})|}{|\log z|}\Big)^{-n},\ z\in V_\delta,\ |w|=\delta|z|.
\end{align*}
For $\delta=\big|\frac{w}{z}\big|<1$, there exists a constant $c_1>0$ (independent of $\delta>0$) such that
$$
\frac{|\log(1+\frac{w}{z})|}{|\log z|}\leq \frac{c_1|\frac{w}{z}|}{-\log |z|}=\frac{c_1\delta}{-\log |z|}\leq \frac{c_1\delta}{-\log R}, \ z\in V_\delta,\ |w|=\delta|z|,
$$
where $R$ is the radius of $V$. Now put $c:=\frac{c_1}{-\log R}$ and $e(\delta):=c\delta$. The last inequality in \eqref{eq:a} follows.
\medskip

Now,
\begin{align}\label{eq:aaa}
\frac{d}{dz}\tilde H_n(z)=\frac{d\boldsymbol\ell}{dz}\frac{d}{d\boldsymbol\ell}\tilde H_n(z)=\frac{\boldsymbol\ell^2}{z}\Big(\frac{d}{d\boldsymbol\ell}f(\boldsymbol\ell)-\sum_{k=1}^{n-1} k a_k \boldsymbol\ell^{k-1}\Big).
\end{align}
Now, by \eqref{eq:a} and \eqref{eq:aaa}, we get:
\begin{align}\label{eq:aa}
\Big|\frac{d}{d\boldsymbol\ell}f(\boldsymbol\ell)-&\sum_{k=1}^{n-1} ka_k\boldsymbol\ell^{k-1}\Big|\leq \\
&\leq\frac{|z|}{|\boldsymbol\ell|^2}D \big(m(1-e(\delta))\big)^{-n}(\log n)^n e^{-\frac{n}{\log n}}|\boldsymbol\ell|^n \frac{1}{|z|},\nonumber\\
&\leq \ D \big(m(1-e(\delta))\big)^{-n}(\log n)^n e^{-\frac{n}{\log n}}|\boldsymbol\ell|^{n-2},\ \boldsymbol\ell \in S_\delta,\ n\in\mathbb N,\ n\geq 2.
\nonumber\end{align}
To get \eqref{eq:nee} from \eqref{eq:aa}, we bound the coefficient $|a_{n-1}|,\,n\in\mathbb N,\ n\geq 3$, by \eqref{eq:izvod}. There exists a constant $C_1>0$ such that:
\begin{equation}\label{eq:fini}
|a_{n-1}|\leq C_1 m^{-(n-1)}\log^{n-1} (n-1) e^{-\frac{n-1}{\log (n-1)}}\leq C_1 m^{-n}\log^{n} n e^{-\frac{n}{\log n}},\ n\in\mathbb N,\ n\geq 3.
\end{equation}
From \eqref{eq:aa}, we get:
\begin{align*}
\Big|\frac{d}{d\boldsymbol\ell}f(\boldsymbol\ell)-\sum_{k=0}^{n-2}ka_k\boldsymbol\ell^{k-1}\Big|\leq \Big|\frac{d}{d\boldsymbol\ell}f(\boldsymbol\ell)-\sum_{k=1}^{n-1}ka_k\boldsymbol\ell^{k-1}\Big|+(n-1)|a_{n-1}| |\boldsymbol\ell|^{n-2},\ n\in\mathbb N.
\end{align*}
Therefore, by \eqref{eq:fini}, there exists a constant $D_1>0$ such that, for every $n\geq 3,$
\begin{align*}
\Big| \frac{d}{d\boldsymbol\ell}f(\boldsymbol\ell)-\sum_{k=1}^{n-2}ka_k\boldsymbol\ell^{k-1}\Big|\leq D_1 \big(m(1-e(\delta))\big)^{-n}(\log n)^n& e^{-\frac{n}{\log n}}|\boldsymbol\ell|^{n-2},\\
 &\boldsymbol\ell\in S_\delta,\ n\in\mathbb N,\,n\geq 3.
\end{align*}
Thus, \eqref{eq:nee} is proven. Finally, using bound \eqref{eq:nee} for the first derivative, we repeat the same procedure for the second derivative etc.
\hfill $\Box$
\bigskip

\noindent \emph{Proof of Proposition~\ref{prop:clocomp}}.\

\begin{enumerate}
\item
Let $0<r<m$. Let $r_F:=\mathrm{ord}(F)\in\mathbb Z$ and let $r_f:=\mathrm{ord}(f)\geq 1$. Then $r_{F\circ f}:=\mathrm{ord}(F\circ f)=r_F\cdot r_f$. Let 
\begin{equation}\label{eq:g}
\boldsymbol\ell^{-r_{F\circ f}}\cdot (\widehat F\circ\widehat f)(\boldsymbol\ell):=\sum_{i=0}^\infty a_i \boldsymbol\ell^i,\ a_i\in\mathbb R, \ i\in\mathbb N.
\end{equation}
 We will prove, using \emph{Fa\` a di Bruno's} formula for derivatives, that, on every subcusp $V\subseteq S$, there exists constant $C_V>0$ such that:
\begin{equation*}
\sup_{\boldsymbol\ell\in V}\,\big|\big(\boldsymbol\ell^{-r_{F\circ f}}\cdot (F\circ f)\big)^{(n)}(\boldsymbol\ell)\big|\leq C_{V} \cdot n! \cdot r^{-n}\cdot (\log n)^n \cdot e^{-\frac{n}{\log n}},\ n\in\mathbb N.
\end{equation*}
This is sufficient since $$\frac{d^{n}}{d\boldsymbol\ell^{n}}\Big(\boldsymbol\ell^{-r_{F\circ f}}\cdot (F\circ f)(\boldsymbol\ell)-\sum_{i=0}^{n-1} a_i\boldsymbol\ell^i\Big)=\frac{d^{n}}{d\boldsymbol\ell^{n}}\big(\boldsymbol\ell^{-r_{F\circ f}}(F\circ f)(\boldsymbol\ell)\big),\ n\in\mathbb N,$$ where $a_i$, $i\in\mathbb N$, are as in \eqref{eq:g},
and since

\begin{equation*}
\frac{d^{k}}{d\boldsymbol\ell^{k}}\Big(\boldsymbol\ell^{-r_{F\circ f}}\cdot (F\circ f)(\boldsymbol\ell)-\sum_{i=0}^{n-1} a_i\boldsymbol\ell^i\Big)(0)=0,\ k=0,\ldots,n-1.
\end{equation*}
The last line is the consequence of the standard fact that $\boldsymbol\ell\mapsto \boldsymbol\ell^{-r_{F\circ f}}\cdot (F\circ f)(\boldsymbol\ell)$ admits $\boldsymbol\ell^{-r_{F\circ f}}\cdot (\widehat F\circ \widehat f)(\boldsymbol\ell)$ as its power asymptotic expansion, as $\boldsymbol\ell\to 0$.
\smallskip

Since $F$ is analytic, and since $f$ admits $\log$-Gevrey expansion of order $m$ and $\mathrm{ord}(f)\geq 1$, it holds that, for every subcusp $V\subseteq S$, there exist constants $C_V>0,\ D_V>0$ and $E_V>0$ such that:
\begin{align}\label{eq:eee}
&\sup_{\boldsymbol\ell\in V}\,\big|(\boldsymbol\ell^{-\min\{r_F,0\}}F)^{(k)}(\boldsymbol\ell)\big|\leq C_{V}\cdot k!\cdot D_V^k,\ k\in\mathbb N,\nonumber\\
&\sup_{\boldsymbol\ell\in V}\,\big|f^{(k)}(\boldsymbol\ell)\big|\leq E_{V} \cdot k! \cdot r^{-k}\cdot (\log k)^k \cdot e^{-\frac{k}{\log k}},\ k\in\mathbb N.
\end{align}
Here, since $\boldsymbol\ell^{-r_F}F$ is analytic at $0$, by diminishing the radius of sector $S$, $D_V>0$ can be made arbitrarily small.

By Fa\` a di Bruno's formula \cite{Mloday} and estimates \eqref{eq:eee}, for every $n\in\mathbb N$, it holds that:
\begin{align*}
&|\big((\boldsymbol\ell^{-\min\{r_F,0\}}\cdot F)\circ f\big)^{(n)}|=\\
&=\Big|\sum_{\substack{0\leq k_i\leq n,\\ \sum_{j=1}^n jk_j=n}}\frac{n!}{k_1!\cdots k_{n}!}(\boldsymbol\ell^{-\min\{r_F,0\}}\cdot F)^{(k)}(f(\boldsymbol\ell))\prod_{j=1}^{n} \Big(\frac{f^{(j)}(\boldsymbol\ell)}{j!}\Big)^{k_j}\Big|\\
&\leq C_V n!\cdot (\log n)^n e^{-\frac{n}{\log n}} m^{-n}\sum_{\substack{0\leq k_i\leq n,\\ \sum_{j=1}^n jk_j=n}} \frac{k!}{k_1!\ldots k_n!} \cdot D_V^{k}\cdot E_V^k =\\
&= C_V n!\cdot (\log n)^n e^{-\frac{n}{\log n}} m^{-n}E_VD_V(1+E_VD_V)^{n-1}\\
&\leq C_V'\cdot n!\cdot (\log n)^n e^{-\frac{n}{\log n}} (m(1-\delta))^{-n},
\end{align*}
where $\delta>0$ can be made arbitrarily small by diminishing the radius of $S$ (and thus making $D_V>0$ arbitrarily small). Here, $k:=k_1+\ldots+k_n$. We use $\log j\leq \log n$, $j=1,\ldots n$. The last sum is evaluated by multinomial theorem. We now put $r:=m(1-\delta)$, so $r$ can be any number $0<r<m$.

\noindent
We have now proven that 
\begin{equation}\label{eq:pomi}
(\boldsymbol\ell^{-\min\{r_F,0\}}\cdot F)\circ f\in LG_r(V),
\end{equation} 
for every $0<r<m$, and admits $(\boldsymbol\ell^{-\min\{r_F,0\}}\cdot \widehat F(\boldsymbol\ell))\circ \widehat f(\boldsymbol\ell)$ as its $\log$-Gevrey expansion of every order $0<r<m$. We show now that \eqref{eq:pomi} implies \begin{equation}\label{eq:la}\boldsymbol\ell^{-\min\{r_{F\circ f},0\}}\cdot \big(F\circ f\big)(\boldsymbol\ell)\in LG_r(V),\end{equation} for every $0<r<m$, and admits $\boldsymbol\ell^{-\min\{r_{F\circ f},0\}}\big(\widehat F\circ \widehat f\big)(\boldsymbol\ell)$ as its $\log$-Gevrey asymptotic expansion of order $r$. This finally proves the statement $(2)$.

Indeed, if $r_F\geq 0$, using $r_{F\circ f}=r_F\cdot r_f$ and $r_f>0$, \eqref{eq:la} follows directly from \eqref{eq:pomi}. If $r_F<0$, then $r_{F\circ f}=r_F\cdot r_f<0$ and 
\begin{align}\label{ali}
&\big(\boldsymbol\ell^{-\min\{r_F,0\}}\cdot F\big)(f(\boldsymbol\ell))=\big(\boldsymbol\ell^{-r_F}\cdot F\big)(f(\boldsymbol\ell))=f(\boldsymbol\ell)^{-r_F}\cdot (F\circ f)(\boldsymbol\ell)=\nonumber\\
&\qquad\quad\quad\quad =a^{-r_F}\cdot \boldsymbol\ell^{-r_F\cdot r_f}\cdot \big(1+a^{-1}\boldsymbol\ell^{-r_f}(f-a\boldsymbol\ell^{r_f})\big)^{-r_F}\cdot (F\circ f)(\boldsymbol\ell),\nonumber\\
&\boldsymbol\ell^{-r_F\cdot r_f}\cdot (F\circ f)(\boldsymbol\ell)=\big(\boldsymbol\ell^{-r_F}\cdot F\big)(f(\boldsymbol\ell))\cdot a^{r_F}\big(1+a^{-1}\boldsymbol\ell^{-r_f}(f-a\boldsymbol\ell^{r_f})\big)^{r_F}.
\end{align}
Here, $a\neq 0$ is such that $\boldsymbol\ell^{-r_f}\cdot f=a+o(1)$, $\boldsymbol\ell\to 0$.
Note that $x\mapsto (1+x)^{r_F}$ is an analytic function at $0$, and $\boldsymbol\ell^{-r_f}\cdot (f-a\boldsymbol\ell^{r_f})\in LG_r(V),\ 0<r<m$,\ with $\mathrm{ord}\big(\boldsymbol\ell^{-r_f}\cdot (f-a\boldsymbol\ell^{r_f})\big)\geq 1$. By \eqref{eq:pomi}, $\big(\boldsymbol\ell^{-r_F}\cdot F\big)(f(\boldsymbol\ell))\in LG_r(V),\ 0<r<m$, so by \eqref{ali} and Proposition~\ref{prop:clomult} it follows that $\boldsymbol\ell^{-r_F\cdot r_f}\cdot (F\circ f)(\boldsymbol\ell)\in LG_r(V),\ 0<r<m$.
\smallskip

\item Let $r_g:=\mathrm{ord}(g)$. Write $\frac{h(\boldsymbol\ell)}{g(\boldsymbol\ell)}=h(\boldsymbol\ell)\cdot\boldsymbol\ell^{-\mathrm{ord}(g)+1}\cdot \frac{1}{\boldsymbol\ell^{-\mathrm{ord}(g)+1}g(\boldsymbol\ell)}$. Function $y\mapsto \frac{1}{y}$ is a meromorphic germ at $0$, so the statement follows directly from statement $(1)$ and Proposition~\ref{prop:clomult}.
\end{enumerate}
\hfill $\Box$

\bigskip

\noindent \emph{Proof of Lemma~\ref{lem:elim}}.

1. A particular case is the removal of the first block $z^\alpha \widehat T_0(\boldsymbol\ell)$, $\widehat T_0(\boldsymbol\ell)=-\boldsymbol\ell^m+h.o.t.$ We search for an elementary change $\widehat\varphi_0(z)=z+z \widehat R_0(\boldsymbol\ell)$, $\widehat R_0(\boldsymbol\ell)\in \boldsymbol\ell\,\mathbb R[[\boldsymbol\ell]]$, which eliminates the first block, except for the first term $-\boldsymbol\ell^m$, which by \cite{mrrz2} cannot be eliminated in $\widehat{\mathfrak L}(\mathbb R)$. This is evident, due to the fact that $\widehat R_0(\boldsymbol\ell)$ contains only non-negative powers of $\boldsymbol\ell$. That is, we search for $\widehat\varphi_0$ such that:
\begin{equation}\label{eq:jes}\widehat\varphi_0^{-1}\circ\widehat f\circ\widehat\varphi_0(z)=z-z^\alpha\boldsymbol\ell^m+ z^{\beta_i}\widehat T_i(\boldsymbol\ell)+h.o.b.\footnote{higher-order blocks.},\ \beta_i>\alpha.\end{equation}
Due to the fact that, for the first block, $\mathrm{ord}_z(\widehat\varphi_0-\mathrm{id})=1$, $\widehat\varphi_0$ from \eqref{eq:jes} does not satisfy a simple Lie bracket equation of the type \eqref{eq:laksa}, as is the case for the higher-order blocks later. As a consequence, it cannot be expressed as a solution of a linear ordinary differential equation \eqref{eq:ekki}. Instead, to get a formula for $\widehat\varphi_0$, we proceed as follows. Let $\widehat\Psi$ be the formal Fatou coordinate of $\widehat f$ and $\widehat\Psi_0$ the formal Fatou coordinate of its normal form $\widehat f_1$ from \eqref{eq:norma1}. Then, by Abel equation $\widehat\Psi\circ\widehat f-\widehat\Psi=1$, such $\widehat\varphi_0$ (such that $\widehat\varphi_0^{-1}\circ\widehat f\circ\widehat\varphi_0-\widehat f_1$ contains blocks of order in $z$ strictly bigger than $\alpha$) exists if and only if it satisfies: 
\begin{equation}\label{eq:kk}
\mathrm{Lb}(\widehat\Psi_0)=\mathrm{Lb}(\widehat\Psi\circ \widehat{\varphi}_0)=\mathrm{Lb}(\widehat\Psi)\circ \widehat\varphi_0,\ \widehat\varphi_0=z+z\widehat R_0(\boldsymbol\ell),\ \widehat R_0\in\boldsymbol\ell\,\mathbb R[[\boldsymbol\ell]].
\end{equation}
Here, $\mathrm{Lb}(.)$ denotes the leading block of a transseries. By formal Taylor expansion of the Abel equation, it follows that
$$
\mathrm{Lb}(\widehat\Psi)=\int \frac{z^{-\alpha}}{\widehat T_0(\boldsymbol\ell)}dz=\int\frac{z^{1-\alpha}}{\widehat T_0(\boldsymbol\ell)\boldsymbol\ell^2}d\boldsymbol\ell,
$$
and, similarly,
$$
\mathrm{Lb}(\widehat\Psi_0)=\int\frac{z^{1-\alpha}}{\boldsymbol\ell^{m+2}}d\boldsymbol\ell.
$$
Since $\widehat\varphi_0(z)=z(1+\widehat R_0(\boldsymbol\ell))$, the statement \eqref{eq:ma} for $\beta_1=\alpha$ follows from \eqref{eq:kk}. It can easily be checked by formal integration and composition that $\widehat R_0(\boldsymbol\ell)$ given by \eqref{eq:ma} belongs to $\boldsymbol\ell\,\mathbb R[[\boldsymbol\ell]]$.

\smallskip
2. After removing of the first block, $\widehat f$ is transformed to \eqref{eq:jes}. To remove the first block $z^{\beta_i}\widehat T_i(\boldsymbol\ell)$, $\widehat T_i(\boldsymbol\ell)\in\mathbb R((\boldsymbol\ell))$, $\beta_i> \alpha,$ we search for $\gamma_i>1$ and $\widehat R_i(\boldsymbol\ell)\in\mathbb R((\boldsymbol\ell))$ that solve the Lie bracket\footnote{Given two germs $f$ and $g$ of one variable $z$, we define (by abuse) their Lie bracket $[f,g]$
as the function $[f,g]:=f'g-g'f$. In fact, considering the vector fields $X_f=f\frac{\partial}{\partial z}$ and $X_g=g\frac{\partial}{\partial z}$ and their Lie bracket  $[X_f,X_g]$, then 
$[X_f,X_g]=[f,g]\frac{\partial}{\partial z}$. We make the same abuse on formal level. The notation was introduced in \cite[Section 3]{mrrz2}.} equation (for details, see the proof of \cite[Theorem A]{mrrz2}):
\begin{equation}\label{eq:laksa}
[-z^\alpha\boldsymbol\ell^{m},\,z^{\gamma_i} \widehat R_i(\boldsymbol\ell)]=-z^{\beta_i} \widehat T_i(\boldsymbol\ell).
\end{equation}
 
Evaluating the Lie bracket equation
\begin{align*}
-z^{\gamma_i+\alpha-1}\Big[\big((\alpha-\gamma_i)\boldsymbol\ell^{
m}+m\boldsymbol\ell^{m+1}\big)\widehat R_i(\boldsymbol\ell)-\boldsymbol\ell^{m+2}\widehat R_i'(\boldsymbol\ell)\Big]=-z^{\beta_i}\widehat T_i(\boldsymbol\ell).
\end{align*}
We choose $\gamma_i$ such that $\gamma_i=\beta_i-\alpha+1$, and $\widehat R_i$ as a formal solution of a linear ordinary differential equation:
$$
\big((\alpha-\gamma_i)\boldsymbol\ell^{-2}+m\boldsymbol\ell^{-1}\big)\widehat R_i(\boldsymbol\ell)-\widehat R_i'(\boldsymbol\ell)=\boldsymbol\ell^{-m-2}\widehat T_i(\boldsymbol\ell).
$$
The solution with zero constant of integration is given by:
\begin{equation}\label{eq:ekki}
\widehat R_i(\boldsymbol\ell)=-z^{\alpha-\gamma_i}\boldsymbol\ell^{m}\int z^{\gamma_i-\alpha}\boldsymbol\ell^{-2m-2}\widehat T_i(\boldsymbol\ell) d\boldsymbol\ell.
\end{equation}
\hfill $\Box$
\bigskip

\noindent \emph{Proof of Proposition~\ref{prop:u}}. We prove for the length $n=1$. Suppose that there exist two exponents of integration for $\widehat F$, $(\alpha,p)\neq (\beta,q)\in(\mathbb R,\mathbb Z)$. There exist $\widehat P_1,\ \widehat Q_1\,\in \widehat {LG}_r(S_\theta)$, $r>\pi/\theta,$ such that:
$$
\frac{d}{d\boldsymbol\ell}\big(z^\alpha \boldsymbol\ell^p\widehat F(\boldsymbol\ell)\big)=z^{\alpha} \boldsymbol\ell^{2p-2}\widehat P_1(\boldsymbol\ell),\ \frac{d}{d\boldsymbol\ell}\big(z^\beta \boldsymbol\ell^q \widehat F(\boldsymbol\ell)\big)=z^{\beta} \boldsymbol\ell^{2q-2}\widehat Q_1(\boldsymbol\ell). 
$$
We compute:
\begin{align*}
z^{\alpha}\boldsymbol\ell^{2p-2}\widehat P_1(\boldsymbol\ell)&=\frac{d}{d\boldsymbol\ell}\big(z^\alpha \boldsymbol\ell^p\widehat F(\boldsymbol\ell)\big)=\frac{d}{d\boldsymbol\ell}\big(z^{\alpha-\beta}\boldsymbol\ell^{p-q}z^\beta\boldsymbol\ell^q\widehat F(\boldsymbol\ell)\big)=\\
&=\frac{d}{d\boldsymbol\ell}(z^{\alpha-\beta}\boldsymbol\ell^{p-q})\cdot z^\beta\boldsymbol\ell^q\widehat F(\boldsymbol\ell)+z^{\alpha-\beta}\boldsymbol\ell^{p-q}z^{\beta}\boldsymbol\ell^{2q-2}\widehat Q_1(\boldsymbol\ell)=\\
&=z^{\alpha}\boldsymbol\ell^{p-2}\big((\alpha-\beta)+(p-q)\boldsymbol\ell\big)\cdot \widehat F(\boldsymbol\ell)+z^{\alpha}\boldsymbol\ell^{p+q-2}\widehat Q_1(\boldsymbol\ell).
\end{align*}
We get:
\begin{align*}
\big((\alpha-\beta)+(p-q)\boldsymbol\ell\big)\widehat F(\boldsymbol\ell)&=\boldsymbol\ell^{p}\widehat P_1(\boldsymbol\ell)-\boldsymbol\ell^{q}\widehat Q_1(\boldsymbol\ell).
\end{align*}
The term $\boldsymbol\ell^{p}\widehat P_1(\boldsymbol\ell)-\boldsymbol\ell^{q}\widehat Q_1(\boldsymbol\ell)$ by Propositions~\ref{prop:closum}-\ref{prop:clodif} belongs to $\widehat{LG}_s(S_\theta)$ for some $r>s>\pi/\theta$. For $(\alpha,p)\neq (\beta,q)$ this is a contradiction, since the right-hand side is integrally summable of order $0$, while the left-hand side is not.

Finally, by the similar proof, if $\widehat F$ is integrally summable of length $n\in\mathbb N$, $n\geq 2,$ its exponent of integration $(\alpha,p)\in(\mathbb R,\mathbb Z)$ is unique. 
\hfill $\Box$
\bigskip

\noindent \emph{Proof of Lemma~\ref{lem:axi}}. 
Take $\widehat R_k^r(\boldsymbol\ell)$, $1\leq r\leq r_k$, an integrally summable series of length $k$, $k\leq n$, with respect to $\{\widehat R_i^{j_i}(\boldsymbol\ell),\ i=1,\ldots,k-1,\ j_i=1,\ldots,r_i\}$, with exponent of integration $(\alpha_k,p_k)$. 
That is,
$$
\widehat R_k^r(\boldsymbol\ell)=\frac{\int e^{-\frac{\alpha_k}{\boldsymbol\ell}} \boldsymbol\ell^{-2p_k-2}\widehat S_{\leq k-1}^{\widehat R_1^{1},\ldots,\widehat R_1^{r_1};\ldots;\widehat R_{k-1}^{1},\ldots,\widehat R_{k-1}^{r_{k-1}}}(\boldsymbol\ell) d\boldsymbol\ell}{e^{-\frac{\alpha_k}{\boldsymbol\ell}}\boldsymbol\ell^{p_k}},
$$
where $\widehat S_{\leq k-1}^{\widehat R_1^{1},\ldots,\widehat R_1^{r_1};\ldots;\widehat R_{k-1}^{1},\ldots,\widehat R_{k-1}^{r_{k-1}}}(\boldsymbol\ell)$ is an algebraic combination (with operations $+,\cdot,/,\frac{d}{d\boldsymbol\ell}$) of integrally summable series $\widehat R_i^{j_i}(\boldsymbol\ell),\ i=1,\ldots,k-1,\ j_i=1,\ldots,r_i,$ with respect to previous ones of strictly lower lengths and of integrally summable series of length $0$. Then, for its derivatives, it holds:
\begin{align*}
\frac{d}{d\boldsymbol\ell}\widehat R_k^r(\boldsymbol\ell)=&\frac{e^{-\frac{2\alpha_k}{\boldsymbol\ell}}\boldsymbol\ell^{-p_k-2}\widehat S_{\leq k-1}^{\widehat R_1^{1},\ldots,\widehat R_1^{r_1};\ldots;\widehat R_{k-1}^{1},\ldots,\widehat R_{k-1}^{r_{k-1}}}(\boldsymbol\ell)-\widehat R_k(\boldsymbol\ell)\cdot e^{-\frac{2\alpha_k}{\boldsymbol\ell}}\boldsymbol\ell^{2p_k-2}(\alpha_k-p_k\boldsymbol\ell)}{e^{-\frac{2\alpha_k}{\boldsymbol\ell}}\boldsymbol\ell^{2p_k}}=\\
=&\boldsymbol\ell^{-3p_k-2}\widehat S_{\leq k-1}^{\widehat R_1^{1},\ldots,\widehat R_1^{r_1};\ldots;\widehat R_{k-1}^{1},\ldots,\widehat R_{k-1}^{r_{k-1}}}(\boldsymbol\ell)-\widehat R_k^r(\boldsymbol\ell)\boldsymbol\ell^{-2}(\alpha_k-p_k\boldsymbol\ell).
\end{align*}
We proceed further by induction.
\hfill $\Box$
\bigskip

\noindent \emph{Proof of Proposition~\ref{prop:ok}.}

(1) Let $\widehat{\varphi}_1,\ \widehat{\varphi}_2\in\widehat{\mathcal L}(\mathbb R)$ be two formal changes that reduce $\widehat f$ to $\widehat f_1$:
$$
\widehat f=\widehat{\varphi}_1\circ\widehat f_2\circ \widehat{\varphi}_1^{-1}=\widehat{\varphi}_2\circ\widehat f_2\circ \widehat{\varphi}_2^{-1}.
$$
Let $\widehat \Psi_0$ as in the statement be a formal Fatou coordinate for $\widehat f_1$, without the constant term.  Directly by Abel equation, $\widehat\Psi_0\circ\widehat{\varphi}_1^{-1},\ \widehat\Psi_0\circ\widehat{\varphi}_2^{-1}\in \widehat{\mathcal L}_2^\infty(\mathbb R)$ are Fatou coordinates for $\widehat f$. Since by \cite{MRRZ2Fatou} the Fatou coordinate of $\widehat f$ is unique in $\widehat{\mathfrak L}^\infty(\mathbb R)$ up to a constant term, we get that there exists $c\in\mathbb R$ such that:
$$
\widehat \Psi_0\circ \widehat{\varphi}_1^{-1}=\widehat \Psi_0\circ \widehat{\varphi}_2^{-1}+c.
$$
Composing by $\widehat \varphi_2\in\widehat{\mathcal L}(\mathbb R)$ from the left, we get
$$
\widehat \Psi_0\circ \widehat{\varphi}_1^{-1}\circ\widehat{\varphi}_2=\widehat\Psi_0+c,\ \widehat{\varphi}_1^{-1}\circ\widehat{\varphi}_2=\widehat \Psi_0^{-1}\circ (\widehat \Psi_0+c)=\widehat f_c.
$$

(2) By Abel equation, $\widehat \Psi_0\circ\widehat\varphi^{-1}$ is a Fatou coordinate for $\widehat f$ in $\widehat{\mathfrak L}^\infty(\mathbb R)$, so it differs from its other Fatou coordinate $\widehat\Psi$ by a constant $C$, $C\in\mathbb R$.
\hfill $\Box$  

\bigskip

\noindent \emph{Proof of Lemma~\ref{lem:loic} $($The idea of the proof by Lo\" ic Teyssier$)$}. For simplicity, denote by $V$ a fixed petal $V_{j}^{\pm}$ and put $h:=h_j^{\pm}$. Then $h(z)=z+o(z)$ analytic on $V$, and $f=h\circ f_1\circ h^{-1}$ on $V$. Let $\Psi_0$ be the Fatou coordinate of $f_1$ on $V$. By Proposition~\ref{prop:gi} (1), there exists an analytic conjugacy $\varphi$ on $V$ conjugating $f$ to $f_1$, of the form $\varphi(z)=z+o(z)$ and admitting, up to precomposition by $f_c$, $c\in\mathbb R$, block iterated integral asymptotic expansion $\widehat\varphi$. By Abel equation, $\Psi_1:=\Psi_0\circ h^{-1}$ and $\Psi_2:=\Psi_0\circ \varphi^{-1}$ are two analytic Fatou coordinates of $f$ on $V$. Since $h$ and $\varphi$ are both tangent to the identity, $\Psi_1(z),\,\Psi_2(z)\sim_{z\to 0}-\frac{1}{\alpha-1}z^{-\alpha+1}\boldsymbol\ell^m$. Indeed, they have the same asymptotic behavior as $\Psi_0$, as $z\to 0$. The image $\Psi_2(V)$ contains, for $R$ sufficiently big, the half-plane $\{w\in\mathbb C:\text{Re}(w)>R\}\subset \mathbb C$. Now consider the difference
$$
\Psi_1\circ \Psi_2^{-1}(w+1)=\Psi_1\circ \Psi_2^{-1}(w)+1,\ \text{Re}(w)>R.
$$
Here, $\Psi_1\circ \Psi_2^{-1}-\mathrm{id}$ is obviously an analytic, periodic function on $\text{Re}(w)>R$, for some $R>0$. Due to  the periodicity, it can be holomorphically extended to whole $\mathbb C$. On the other hand,
$$
\Psi_1\circ \Psi_2^{-1}(w)-w=\Psi_0\circ h^{-1}\circ \varphi\circ \Psi_0^{-1}(w)-w,\ \text{Re}(w)>R.
$$
Note that $h^{-1}\circ \varphi(z)=z+o(z)$, as $z\to 0$ on $V$, and $\Psi_0(z)\sim z^{-\alpha+1}\boldsymbol\ell^m+o(z^{-\alpha+1}\boldsymbol\ell^m)$ has an asymptotic behavior, so we get
\begin{align*}
&\Psi_1\circ \Psi_2^{-1}(w)-w=o(w),\ |w|\to \infty, \ \text{Re}(w)>R,\\
&\frac{\Psi_1\circ \Psi_2^{-1}(w)-w}{w}=o(1),\ |w|\to \infty, \ \text{Re}(w)>R.
\end{align*} 
Due to the periodicity of the function in the numerator, we conclude that for the entire periodic extension $L(w)$ of the function $\Psi_1\circ \Psi_2^{-1}(w)-w$ in the numerator defined on the image $\Psi_2(V)$ it also holds:
$$
\frac{L(w)}{w}=o(1),\ |w|\to \infty.
$$
The function in the numerator is entire, so the function on the left-hand side has a pole in $0$. Subtracting $\frac{C}{w}$, for the appropriate residuum $C\in\mathbb C$, we get an entire function on the left-hand side, and on the right-hand-side we get:
$$
\frac{L(w)}{w}-\frac{C}{w}=o(1)-\frac{C}{w},\ |w|\to \infty.
$$
The right-hand side is then also entire. Moreover, it is obviously bounded as $|w|\to\infty$, and tends to zero. Thus, by Liouville's theorem, it is equal to the constant zero.
Thus, $L(w)=C$ on $\mathbb C$ and $\Psi_1\circ \Psi_2^{-1}(w)-w=C$ on the image $\Psi_2(V)$. Therefore, $\Psi_1$ and $\Psi_2$ differ on each petal $V$ \emph{only by a constant}, and the following holds on $V$:
\begin{align*}
&\Psi_2-\Psi_1=C,\\
&\Psi_0\circ h^{-1}\circ\varphi-\Psi_0=C,\\  
&h^{-1}\circ\varphi=\Psi_0^{-1}\big(\Psi_0+C\big),\\
&h^{-1}\circ\varphi=f_C,\ \ C\in\mathbb R.
\end{align*}
Therefore, $h=\varphi\circ f_{-C}$. Since $\varphi$ admits $\widehat\varphi$ as its block iterated integral asymptotic expansion, which is by Definition~\ref{def:iiexp} well-defined up to precomposition by $f_c,\ c\in\mathbb R,$ it follows that also $h$ admits $\widehat\varphi$ as its block iterated integral asymptotic expansion.
\hfill $\Box$
\medskip

\noindent \emph{Proof of Proposition~\ref{prop:unidulac}}. By \cite[Section 24E]{ilya}, for small $\delta>0$, we get that there exists $C>0$ such that:
$$
\big|f(z)-z-z^{\alpha}P(-\log z)\big|\leq C|z|^{\alpha+\delta},\ z\in\mathcal R_C.
$$
Here, $P(-\log z)$ denotes the first polynomial block of $f$. 
Note furthermore that it holds that $z^\alpha\boldsymbol\ell^m=o(z^{\beta}\boldsymbol\ell^k)$, uniformly as $|z|\to 0$ on $\mathcal R_C$, if and only if $(\beta,k)\prec (\alpha,m)$, due to the fact that, on a standard quadratic domain, the argument is \emph{controlled} by the radius (for $z=re^{i\varphi}$ from a standard quadratic domain, it holds that $\varphi<\log^2 r$). We conclude that there exists a uniform constant $D>0$ such that:
\begin{align*}
|f(z)-z+az^{\alpha}\boldsymbol\ell^m|&\leq |f(z)-z-z^\alpha P(-\log z)|+|z|^\alpha \big|P(-\log z)+a\boldsymbol\ell^m\big|\\
&\leq D |z|^{\alpha}|\boldsymbol\ell|^{m+1},\ z\in\mathcal R_C.
\end{align*}\hfill $\Box$
\medskip

\noindent \emph{Proof of Proposition~\ref{prop:grupa}}. We prove $(1)$ and $(2)$ simultaneously, by considering inverses and compositions separately.

\emph{Step 1. The inverse.} Since $f(z)=z+o(z)$ is a tangent to the identity germ, with $o(z)$, $z\to 0$, uniform on $\mathcal R_C$, it is injective and, by the inverse function theorem, analytically invertible around $z=0$ on a standard quadratic domain $\mathcal R_C$. This can be checked in the logarithmic chart, similarly as in the proof of Proposition~\ref{prop:Fatouexp} $(2)$. Note that $f'(z)=1+o(1)$, $z\to 0$, where $o(1)$ is uniform on $\mathcal R_C$, by the Cauchy formula.  Due to bijectivity, $f^{-1}$ maps the positive real line again to the positive real line. Moreover, $f^{-1}(w)=w+o(w)$, $o(w)$ uniform on $\mathcal R_C$, as $w\to 0$, follows directly from the similar bound on $f$. Furthermore, an uniform bound for $f^{-1}(w)$:
$$
|f^{-1}(w)-w-aw^{\alpha}\boldsymbol\ell^m|\leq D|w^\alpha\boldsymbol\ell^{m+1}|, \ w\in\mathcal R_C,\ D>0,
$$
follows directly by putting $z=f^{-1}(w)$ in the uniform bound \eqref{eq:uniest}. By Proposition~\ref{prop:leau}, the attracting petals for $f$ thus become repelling for the inverse $f^{-1}$, and the other way round. 

We now prove that $f^{-1}(w)$ admits a generalized Dulac asymptotic expansion on every petal $V_j^{\pm}$. We use the fact that $f$ admits a generalized Dulac expansion on $V_j^\pm$. First, we prove by the operator formula (see e.g. \cite{mrrz2}) that the formal inverse $\widehat f^{-1}$ of its generalized Dulac expansion $\widehat f$ is again parabolic generalized Dulac series:
\begin{align}\label{eq:kl}
\widehat f^{-1}=&\mathrm{id}+\widehat g+(\widehat g\circ\widehat f-\widehat g)+\big((\widehat g\circ\widehat f-\widehat g)\circ\widehat f-(\widehat g\circ\widehat f-\widehat g)\big)+\ldots\nonumber\\
=&\mathrm{id}+\widehat g+(\widehat g'\cdot\widehat g+\frac{1}{2!}\widehat g''\cdot\widehat g^2+\ldots)+\nonumber\\
&\qquad+\Big((\widehat g'\cdot\widehat g+\frac{1}{2!}\widehat g''\cdot\widehat g+\ldots)'\cdot \widehat g+\frac{1}{2!}(\widehat g'\cdot\widehat g+\frac{1}{2!}\widehat g''\cdot\widehat g+\ldots)''\cdot \widehat g^2\Big)+\ldots
\end{align}
Here, $\widehat g=\mathrm{id}-\widehat f$. From this formula, we can conclude that the \emph{coefficient} of each block of $\widehat f^{-1}$ is a finite combination, with operations $+,\cdot,\frac{d}{d\boldsymbol\ell}$, of \emph{coefficients} $\widehat R_i(\boldsymbol\ell)$ of $\widehat f$. Since $f$ is a generalized Dulac germ, $\widehat R_i(\boldsymbol\ell)\in\widehat {LG}_m(\boldsymbol\ell(V_j^{\pm}))$, for some $m>\frac{\alpha-1}{2}$, $i\in\mathbb N$. By Propositions~\ref{prop:closum}-\ref{prop:clodif}, their finite algebraic combinations belong to $\widehat {LG}_m(\boldsymbol\ell(V_j^{\pm}))$, for some $r<m$. We take $r>\frac{\alpha-1}{2}$. Their sums respect the operations. Therefore, $\widehat f^{-1}$ is a parabolic generalized Dulac series. We now prove that it is generalized Dulac expansion of $f^{-1}$ on petals.
 
Now, deducing the expansion of $f^{-1}(w)$ block-by-block from the expansion of $f$ (in blocks in strictly increasing powers of $w$), it can be checked that the \emph{coefficients} in $\boldsymbol\ell$ of every block are $\log$-Gevrey sums of order $r$ of their formal counterparts deduced in the formal inverse \eqref{eq:kl}.
\smallskip

\emph{Step 2. The composition.} The proof for the composition is done similarly, using the fact that the germs and their expansions are parabolic. We first prove that the formal composition of two parabolic generalized Dulac series is again a parabolic generalized Dulac series. Then we prove that, for $f$ and $g$ parabolic generalized Dulac germs, $\widehat f\circ \widehat g$ is the generalized Dulac expansion of $f\circ g$ on appropriate petals.
\hfill $\Box$

\bigskip

\noindent \emph{Proof of Proposition~\ref{prop:auxil}}. Exactly as in \cite{loray2}, we compute that, for a sector $W_\theta$ at infinity of opening $2\pi-2\theta$
for any small $\theta>0$, there exists a sufficiently big radius $R_\theta$, such that, for $w\in W_\theta\cap \widetilde{\mathcal R}$, $|w|>R_\theta$, there exists a constant $C_\theta>0$ such that:
\begin{equation}\label{eq:too}
|F^{\circ n}(w)|\geq C_\theta(|w|+n),\ n\in\mathbb N,\ w\in W_{\theta},\,|w|>R_\theta.
\end{equation}
Note that $R_\theta$ does not depend on the level of $\widetilde{\mathcal R}$. Since $f^{\circ n}(z)=h^{-1}\circ F^{\circ n}(w)$, where $w=h(z)=-\frac{1}{a(\alpha-1)}z^{-\alpha+1}\boldsymbol\ell^{-m}$, there exists $C>0$ such that:
$$
|f^{\circ n}(z)|\leq C|F^{\circ n}(w)|^{-\frac{1}{\alpha-1}}\big|\boldsymbol\ell(F^{\circ n}(w))\big|^{-\frac{m}{\alpha-1}}.
$$
Thus, for every sub-sector $W$ of an attracting petal $V_+^j$, there exists a constant $C_W>0$ such that:
\begin{equation}\label{eq:i}
|f^{\circ n}(z)|\leq C_W \cdot (|w|+n)^{-\frac{1}{\alpha-1}}\big|\boldsymbol\ell(F^{\circ n}(w))\big|^{-\frac{m}{\alpha-1}},\ n\in\mathbb N,\ z\in W.
\end{equation}
Now, 
\begin{align*}
|\delta(f^{\circ n}(z))|\leq C|f^{\circ n}(z)|^\gamma|\boldsymbol\ell(f^{\circ n}(z))|^r&\leq C_W (|w|+n)^{-\frac{\gamma}{\alpha-1}}|\boldsymbol\ell(F^{\circ n}(w))|^{r-\frac{m\gamma}{\alpha-1}}\\
&\!\!\!\!\!\!\!=C_W (|z|^{-\alpha+1}|\boldsymbol\ell|^{-m}+n)^{-\frac{\gamma}{\alpha-1}}|\boldsymbol\ell(F^{\circ n}(w))|^{r-\frac{m\gamma}{\alpha-1}}.
\end{align*}
We now distinguish two cases:
\begin{enumerate}
\item[1.] Suppose $\gamma=\alpha-1$. Then $r\geq m+2$. Then, by \eqref{eq:too},
\begin{align*}
&|\delta(f^{\circ n}(z))|\leq C_W\big(|z|^{-\alpha+1}|\boldsymbol\ell|^{-m}+n\big)^{-1}\Big(\frac{1}{\log (|z|^{-\alpha+1}|\boldsymbol\ell|^{-m}+n)}\Big)^{r-m}\\
&\ \leq C_W |z|^{\alpha-1}|\boldsymbol\ell|^m  \boldsymbol\ell(|z|)^{r-m} \cdot\\
&\ \quad \cdot\frac{1}{1+n|z|^{\alpha-1}|\boldsymbol\ell|^m}\Big(\frac{1}{1-\frac{m}{\alpha-1}\boldsymbol\ell(|z|)\log|\boldsymbol\ell|+\frac{1}{\alpha-1}\boldsymbol\ell(|z|)\log(1+n|z|^{\alpha-1}|\boldsymbol\ell|^{m})}\Big)^{r-m}.
\end{align*}
Indeed, $\frac{1}{|\log F^{\circ n}(w)|}\leq \frac{1}{\log|F^{\circ n}(w)|}\leq \frac{1}{\log(|w|+n)}$, and $r-m>0$.
\bigskip

\item[2.] Suppose $\gamma>\alpha-1$. Then, for $\delta_1>0$ and $\varepsilon>0$, $\delta(z)=O(z^{\gamma-\delta_1})$. From \eqref{eq:i}, it follows:
$$
|f^{\circ n}(z)|\leq C_W \cdot (|w|+n)^{-\frac{1}{\alpha-1}+\varepsilon},\ n\in\mathbb N,\ z\in W.
$$
Now, by \eqref{eq:too},
\begin{align*}
|\delta(f^{\circ n}(z))|&\leq C_W\big(|z|^{-\alpha+1}|\boldsymbol\ell|^{-m}+n\big)^{-\frac{\gamma}{\alpha-1}+\delta_2}\\
&\leq C_W |z|^{\gamma-\delta_2(\alpha-1)}|\boldsymbol\ell|^{m(\frac{\gamma}{\alpha-1}-\delta_2)}(1+n|z|^{\alpha-1}|\boldsymbol\ell|^{m})^{-\frac{\gamma}{\alpha-1}+\delta_2}.
\end{align*}
Here, by choice of $\varepsilon>0$ and $\delta_1>0$, $\delta_2:=\frac{\delta_1}{\alpha-1}+\varepsilon(\gamma-\delta_1)$ can be made arbitrarily small, such that still $\frac{\gamma}{\alpha-1}-\delta_2>1$. Now $\delta:=\delta_2(\alpha-1)$ can be made arbitrarily small, such that $\gamma-(\alpha-1)-\delta>0$.
\end{enumerate}

\noindent Putting 1. or 2. in \eqref{eq:new}, we get uniform convergence of the series on each subsector $W\subset V_j^+$, so $h_j$ is analytic on $V_+^j$ by Weierstrass theorem. By integral approximation of the series as in \cite{MRRZ2Fatou}, we get \eqref{eq:old}. 
\hfill $\Box$

\bigskip

\textbf{Acknowledgement.} The authors would like to warmly thank Jean-Philippe Rolin for numerous discussions on the subject, and to Lo\" ic Teyssier and Daniel Panazzolo for ideas and comments that helped in the realization of this paper. We also thank the referee for his or her very detailed report and many suggestions that helped us to improve the paper.

\bigskip

\emph{Address:}$\quad$$^{1}$: Institut de Math\' ematiques de Bourgogne, UMR 5584, CNRS, Universit\' e Bourgogne Franche-Comt\' e, F-21000 Dijon, France 

$^{2}$ : University of Zagreb, Faculty of Science, Department of Mathematics, Bijeni\v cka 30, 10000 Zagreb, Croatia

\end{document}